\documentclass[11pt]{amsart}
\usepackage{stmaryrd}
\usepackage{nicematrix}
\usepackage{amsfonts,amssymb,mathtools,stmaryrd,graphicx,a4,bbm,mdframed,mathrsfs,hyperref,verbatim, subcaption}
\usepackage{xcolor}
\graphicspath{{./figures/}}
\usepackage{float}
\usepackage{tikz, tkz-euclide}
\usepackage{tikz-cd}
\usepackage[title]{appendix}
\usepackage{bm}
\usetikzlibrary{calc}
\usetikzlibrary{arrows}
\usetikzlibrary{positioning}
\usetikzlibrary{patterns}
\usetikzlibrary{decorations.markings}
\usetikzlibrary{tikzmark}

\usepackage[normalem]{ulem}
\usepackage{euscript} 
\usepackage{color}\definecolor{Red}{cmyk}{0,1,1,0}
\usepackage{enumitem}

\newcommand{\spann}{span}

\newcommand*\circled[1]{\tikz[baseline=(char.base)]{
            \node[shape=circle,draw,inner sep=2pt] (char) {#1};}}

\newcommand{\Var}[0]{\text{Var}}
\newcommand\redsout{\bgroup\markoverwith{\textcolor{red}{\rule[0.5ex]{2pt}{0.4pt}}}\ULon}

\theoremstyle{plain}
\newtheorem{theorem}{Theorem}
\newtheorem{lemma}[theorem]{Lemma}
\newtheorem{corollary}[theorem]{Corollary}
\newtheorem{proposition}[theorem]{Proposition}
\newtheorem*{theorem*}{Theorem}
\newtheorem*{conjecture*}{Conjecture}
\newtheorem*{corollary*}{Corollary}
\newtheorem*{lemma*}{Lemma}
\newtheorem*{claim*}{Claim}

\theoremstyle{definition}
\newtheorem{definition}[theorem]{Definition}
\newtheorem{remark}[theorem]{Remark}
\newtheorem{notation}[theorem]{Notation}
\newtheorem{example}[theorem]{Example}
  
\newcommand{\g}{\mathbb{G}}

\makeatletter
\def\blfootnote{\gdef\@thefnmark{}\@footnotetext}
\makeatother
  
\DeclareMathOperator{\diag}{diag}

\begin{document}

\voffset=-1.5truecm\hsize=16.5truecm    \vsize=24.truecm
\baselineskip=14pt plus0.1pt minus0.1pt \parindent=12pt
\lineskip=4pt\lineskiplimit=0.1pt      \parskip=0.1pt plus1pt

\def\ds{\displaystyle}\def\st{\scriptstyle}\def\sst{\scriptscriptstyle}

\global\newcount\numsec\global\newcount\numfor
\gdef\profonditastruttura{\dp\strutbox}
\def\senondefinito#1{\expandafter\ifx\csname#1\endcsname\relax}
\def\SIA #1,#2,#3 {\senondefinito{#1#2}
\expandafter\xdef\csname #1#2\endcsname{#3} \else
\write16{???? il simbolo #2 e' gia' stato definito !!!!} \fi}
\def\etichetta(#1){(\veroparagrafo.\veraformula)
\SIA e,#1,(\veroparagrafo.\veraformula)
 \global\advance\numfor by 1
 \write16{ EQ \equ(#1) ha simbolo #1 }}
\def\etichettaa(#1){(A\veroparagrafo.\veraformula)
 \SIA e,#1,(A\veroparagrafo.\veraformula)
 \global\advance\numfor by 1\write16{ EQ \equ(#1) ha simbolo #1 }}
\def\BOZZA{\def\alato(##1){
 {\vtop to \profonditastruttura{\baselineskip
 \profonditastruttura\vss
 \rlap{\kern-\hsize\kern-1.2truecm{$\scriptstyle##1$}}}}}}
\def\alato(#1){}
\def\veroparagrafo{\number\numsec}\def\veraformula{\number\numfor}
\def\Eq(#1){\eqno{\etichetta(#1)\alato(#1)}}
\def\eq(#1){\etichetta(#1)\alato(#1)}
\def\Eqa(#1){\eqno{\etichettaa(#1)\alato(#1)}}
\def\eqa(#1){\etichettaa(#1)\alato(#1)}
\def\equ(#1){\senondefinito{e#1}$\clubsuit$#1\else\csname e#1\endcsname\fi}
\let\EQ=\Eq


\def\\{\noindent}
\let\io=\infty

\def\VU{{\mathbb{V}}}
\def\EE{{\mathbb{E}}}
\def\GI{{\mathbb{G}}}
\def\TT{{\mathbb{T}}}
\def\C{\mathbb{C}}
\def\LL{{\cal L}}
\def\RR{{\cal R}}
\def\SS{{\cal S}}
\def\NN{{\cal N}}
\def\HH{{\cal H}}
\def\GG{{\cal G}}
\def\PP{{\cal P}}
\def\AA{{\cal A}}
\def\BB{{\cal B}}
\def\FF{{\cal F}}
\def\vv{\vskip.2cm}
\def\gt{{\tilde\g}}
\def\E{{\mathcal E} }
\def\I{{\rm I}}

\def\cal{\mathcal}

\def\tende#1{\vtop{\ialign{##\crcr\rightarrowfill\crcr
              \noalign{\kern-1pt\nointerlineskip}
              \hskip3.pt${\scriptstyle #1}$\hskip3.pt\crcr}}}
\def\otto{{\kern-1.truept\leftarrow\kern-5.truept\to\kern-1.truept}}
\def\arm{{}}
\font\bigfnt=cmbx10 scaled\magstep1

\newcommand{\card}[1]{\left|#1\right|}
\newcommand{\und}[1]{\underline{#1}}
\newcommand{\Dom}[0]{\text{Dom}}

\def\Var{\text{Var}}

\thispagestyle{empty}

\begin{center}
{\LARGE Extendable Shift Maps and the Canonical Map on Generalized Countable Markov Shifts}
\vskip.5cm
Rodrigo Bissacot$^{1, (a)}$, Iv\'an D\'iaz-Granados$^{1}$, and Thiago Raszeja$^{2,3}$
\vskip.3cm
\begin{footnotesize}
$^{1}$Institute of Mathematics, Statistics and Computer Science (IME-USP), University of S\~{a}o Paulo, Brazil\\
$^{2}$ Departamento de Matem\'atica, Pontif\'icia Universidade Cat\'olica do
Rio de Janeiro (PUC-Rio), Brazil \\
$^{3}$ Institute of Physics (IFUSP), University of S\~{a}o Paulo, Brazil
\end{footnotesize}
\vskip.1cm
\begin{scriptsize}
emails: rodrigo.bissacot@gmail.com; diazgranados.ivanf@gmail.com; tcraszeja@gmail.com\\
keywords: symbolic dynamics, C$^*$-algebras, generalized countable Markov shifts, Exel-Laca algebras
\end{scriptsize}

\end{center}

\def\be{\begin{equation}}
\def\ee{\end{equation}}

\vskip1.0cm
\begin{quote}
{\small

\textbf{Abstract.} \begin{footnotesize} We obtain an operator-algebraic
characterization of when the shift map can be continuously extended from a
standard countable Markov shift $\Sigma_A$ to its generalized countable Markov
shift $X_A$, a compactification of $\Sigma_A$. We show that a canonical unital
completely positive map on the Exel--Laca algebra $\mathcal{O}_A$ can be defined
exactly when the shift extends continuously from $\Sigma_A$ to $X_A$, and that
in this case its restriction to the maximal commutative C$^*$-subalgebra
$\mathcal{D}_A$ is a $*$-endomorphism to which the extended shift is dual. We
also give a criterion on the matrix $A$ itself, in terms of the convergence of
its columns, and exhibit classes of transition matrices satisfying these
equivalent conditions. Column-finiteness of $A$ is necessary for the existence
of the continuous extension. Finally, we construct an example in which the set
of empty words is a Cantor set and the extended dynamics on it is the inverse of
the dyadic odometer.
\end{footnotesize}

}
\end{quote}

\tableofcontents
\numsec=2\numfor=1
\section*{Introduction}
\vv
\noindent

\blfootnote{(a) Corresponding author.}

The \textit{Cuntz algebras} \cite{Cuntz1977} were the first concrete examples of C$^*$-algebras that are infinite, simple, and separable; the existence of such algebras had been proved by J. Dixmier \cite{Dixmier1963} more than a decade earlier. A few years later, J. Cuntz and W. Krieger \cite{CuntzKrieger1980} extended this construction to a class of C$^*$-algebras, later called \textit{Cuntz--Krieger algebras} and denoted by $\mathcal{O}_A$, and established the connection between them and the Markov shifts $\Sigma_A$, where $A$ is the transition matrix. The uniqueness of $\mathcal{O}_A$ holds under suitable assumptions on $A$; for example, when $A$ is irreducible and is not a permutation matrix. Their results are therefore very general from a dynamical systems point of view when the alphabet is finite (compact $\Sigma_A$).

At that point, however, the only infinite-alphabet case available was the full shift considered by Cuntz \cite{Cuntz1977}, and the generalization to arbitrary matrices over infinite alphabets remained an open problem for almost two decades. In the late 1990s, A. Kumjian, D. Pask, I. Raeburn, and J. Renault \cite{KumjPaskRaeb1998, KumjPaskRaebRen1997} extended these algebras to row-finite matrices (which, for irreducible $A$, correspond to locally compact $\Sigma_A$) via groupoid C$^*$-algebra theory \cite{Renault1980}, and in 1999 the row-finiteness condition was finally removed by R. Exel and M. Laca \cite{EL1999}. Later, in 2000, they computed the K-theory of these C$^*$-algebras \cite{EL2000}, and in the same year J. Renault developed the corresponding groupoid C$^*$-algebra approach \cite{Renault1999}.

Cuntz and Krieger introduced in their seminal paper \cite{CuntzKrieger1980} a unital completely positive map that is now known as the \textit{canonical map} on Cuntz--Krieger algebras, namely $\Theta:\mathcal{O}_A\to \mathcal{O}_A$ defined by $\Theta(X)=\sum_{j=1}^N S_jXS_j^*$, where $S_j$, with $1\leq j\leq N$, are the partial isometries that generate $\mathcal{O}_A$. The commutative C$^*$-subalgebra $\mathcal{D}_A$ is described as the C$^*$-algebra generated by the elements of the form $\Theta^k(S_iS_i^*)$, with $k \geq 0$ and $i$ an element of the alphabet. Moreover, $\Theta$ has been used to prove that the noncommutative topological entropy of Cuntz--Krieger algebras coincides with the Gurevich entropy of the associated Markov shift for finite alphabets \cite{BocaGoldstein2000}. An analogous result holds for noncommutative pressure and Gurevich pressure \cite{KerrPinzari2002}.

In their work, R. Exel and M. Laca \cite{EL1999} constructed, among other results, a
topological space that provides a local compactification (and, in many cases, a
compactification) of the standard Markov shift $\Sigma_A$. The complement of
$\Sigma_A$ in this space consists of finite words and the empty word, where the same
word may appear with different labels, giving rise to a type of multiplicity. This
space was later named the \emph{generalized countable Markov shift} $X_A$
\cite{BEFR2018, BEFR2022}. Many compactifications of $\Sigma_A$ can be considered;
one example is the Ott--Tomforde--Willis shift space
\cite{OttTomfordeWillis2014}, in which a single additional symbol is added to the
alphabet, yielding finite words with a unique label. In contrast, the Stone--\v{C}ech
compactification \cite{cech:1937, stone:1937} and, in many cases, the blur shift
spaces \cite{Almeida_Sobottka:21} add so many elements that the resulting spaces are
no longer metrizable. The generalized countable Markov shift $X_A$ lies between these
compactifications: it adds enough points to compactify, but few enough to remain
metrizable, and its construction depends exclusively on the matrix $A$ \cite{EL1999}.

For a finite alphabet, the Cuntz--Krieger algebra contains a commutative C$^*$-subalgebra $\mathcal{D}_A \simeq C(\Sigma_A)$. For an infinite alphabet, $\mathcal{O}_A$ contains an analogous subalgebra $\mathcal{D}_A$, generated by the corresponding pairwise commuting projections. What, then, is its Gelfand spectrum? The answer is $X_A$.

The standard countable Markov shift spaces $\Sigma_A$ already have a rich history as important objects in dynamical systems, ergodic theory, operator algebras, and thermodynamic formalism. In particular, we can mention the contributions of different groups over the last thirty years to the thermodynamic formalism for countable Markov shifts \cite{BuzSa2003,DenUr1991,FieFie95,FieFieYu2002,FreireVargas2018,Gurevich1969topological,Gurevich1970shift,Gurevich1984variational,GurSav1998thermodynamical,Iommi2021,MaulUr2001,MaulUr2003graph,Pesin2014work,sarig:09, sarig1999thermodynamic,Sarig2001PT,Sarig2001_Null, Sarig2015,Shwartz2019}. Then, a natural question arises concerning the construction of the thermodynamic formalism on $X_A$, an investigation started in recent years by two of the authors jointly with R. Exel and R. Frausino \cite{BEFR2018,BEFR2022,Raszeja2020}, and further developed by Y. Michimoto, Y. Nakano, H. Toyokawa and K. Yoshida \cite{MichiNakaHisaYoshi2022}. We observe that the quantum counterpart of this thermodynamic formalism, concerning Kubo--Martin--Schwinger (KMS) states, was started earlier by R. Exel and M. Laca \cite{EL2003}.

At this point, each investigation must decide whether the shift map on $X_A$ is
partially defined or defined on the entire space $X_A$. If it is defined on the whole
space, then determining whether the extension of $\sigma$ from $\Sigma_A$ to $X_A$ is
continuous becomes a delicate question, since $\Sigma_A$ is dense in $X_A$. For
important examples of countable Markov shifts, such as the renewal shift, the
existence of a continuous extension is well known \cite{MichiNakaHisaYoshi2022}; for the full shift, however, the
shift map admits no continuous extension. One of the contributions of this article is
to characterize, in operator-algebraic terms and also in terms of the matrix $A$, when
a continuous extension does exist, and to provide classes of examples in which the
shift map can be continuously extended from $\Sigma_A$ to $X_A$.

In the present work, our main topics of investigation are the existence of a unital completely positive (UCP) map $\Theta:\mathcal{O}_A\to\mathcal{O}_A$ on the Exel--Laca algebra which generalizes the canonical map on Cuntz--Krieger algebras, the existence of an endomorphism $\alpha: \mathcal{D}_A \to \mathcal{D}_A$ to which the shift map is dual (explicitly, using the Gelfand transform, we obtain $\widehat{\alpha(a)}=\widehat{a}\circ\sigma$ for every $a\in\mathcal{D}_A$), and the continuous extendability of the shift map to the whole of $X_A$. We show that all these properties are equivalent and that the column-finiteness of $A$ is a necessary condition for the existence of a continuous extension of the shift map. Furthermore, we provide a practical requirement on the matrix $A$ that is equivalent to the aforementioned existence conditions. 

We illustrate these ideas with several examples. In particular, we present three classes of generalized countable Markov shifts, namely the \emph{single empty-word}, \emph{periodic renewal}, and \emph{permutation renewal} classes, for which the shift map extends continuously to $X_A$. In all three classes, the set of empty words $E_A$ is finite and invariant in the sense that $\sigma(E_A) = E_A$. In a more general setting, we show that the inclusion may be strict, that is, $\sigma(E_A) \subsetneq E_A$. We also prove that the extension of the shift map in the latter two classes is a local homeomorphism. The same is not true in general for the single empty-word class, for which we provide counterexamples. Although all the classes mentioned above have a finite set of empty words, we construct two further examples that do not fit into them and for which the shift map still admits a continuous extension. One of them has countably infinitely many empty words, while the other has uncountably many.

This paper is organized as follows.

In Section \ref{sec:Exel_Laca_GCMS}, we recall some basic notions of the Exel--Laca algebra $\mathcal{O}_A$ via a faithful representation $\pi$ of it in $\mathfrak{B}(\ell^2(\Sigma_A))$, as well as the set $X_A$, a generalization of the Markov shift space $\Sigma_A$, which corresponds to the spectrum of the maximal abelian C$^*$-subalgebra (MASA) $\mathcal{D}_A$ of the Exel--Laca algebra when $A$ is irreducible (see Subsection 6.3 of \cite{Renault2008Cartan}), our standing assumption. This section includes the perspective, proposed by Exel and Laca in \cite{EL1999}, of identifying the elements of $X_A$ with configurations on Cayley trees; these elements are determined by an admissible word (finite or infinite), called the \emph{stem}, together with a subset of the alphabet, called the \emph{root}, the latter being used only for elements associated with finite words. We briefly explain the topology of $X_A$ via \emph{generalized cylinder sets} and present the shift map, which is defined on the open subset of $X_A$ consisting of the non-empty words. 

In Section \ref{sec:Representations for D_A}, we explicitly describe the Gelfand isomorphism between the C$^*$-algebras $C_0(X_A)$ and $\mathcal{D}_A$. For the Cuntz--Krieger algebras (finite alphabet), the Gelfand transform maps the operators $S_w S_w^*$ to the characteristic functions of the cylinders $[w]$, and this correspondence extends to a $*$-isomorphism from the whole C$^*$-subalgebra $\mathcal{D}_A$ onto $C(\Sigma_A)$; see \cite{KessStadStrat2007}. Our construction shows the similarities and the differences between the finite and infinite alphabet cases. 

In Section \ref{sec: Canonical_UCP}, we introduce the map $\theta: \mathfrak{B}(\ell^2(\Sigma_A)) \to \mathfrak{B}(\ell^2(\Sigma_A))$ given by $\theta(Y)=\sum_{i\in\mathbb{N}}T_iYT_i^*$, where $T_i=\pi(S_i)$ and $\pi$ is the faithful representation above, and show that it is a normal unital completely positive map. Under certain conditions, we also identify $\Theta=\pi^{-1}\circ\theta\circ\pi$ as the canonical map on Exel--Laca algebras, first studied in \cite{MichiNakaHisaYoshi2022}, which generalizes the corresponding notion for Cuntz--Krieger algebras.

Section \ref{sec:dual_endomorphism} begins by defining when a dynamics is dual to a $*$-endomorphism, that is, given an endomorphism of a commutative Banach algebra, we can construct a partial dynamical system defined on the spectrum of the algebra. We then focus on the restriction $\alpha:=\theta\vert_{\mathcal{D}_A}$, which is a $*$-homomorphism. We show that column-finiteness is necessary for the continuous extension of the shift. This section contains our main result (Theorem \ref{Theorem: alpha well-defined iff sigma extended}): we prove that $\alpha$ is a $*$-endomorphism of $\mathcal{D}_A$ if and only if $\sigma$ has a continuous extension to $X_A$, in which case $\sigma$ is dual to $\alpha$. Using the Gelfand representation, this can be equivalently written as $\Psi(f) = f \circ \sigma$ for every $f \in C_0(X_A)$, where $\Psi$ is the endomorphism of $C_0(X_A)$ corresponding to $\alpha$. Furthermore, we present two additional equivalent conditions: the existence of the canonical map on Exel--Laca algebras, and a condition on $A$ related to the convergence of the columns of the matrix.

In Section \ref{sec:examples}, we introduce three classes of compact generalized countable Markov shifts: one with a single empty word and the other two with finitely many empty words. We show that, for these classes, the shift map admits a continuous extension. Several examples are presented to justify some of the choices made in the definitions of these classes, as well as to exhibit matrices whose associated spaces $X_A$ belong to one of them. These classes are not disjoint, since the renewal shift belongs to all three families. In all of them, the set of empty-word configurations $E_A$ is finite and invariant; more precisely, $\sigma(E_A)=E_A$, and we explicitly describe the dynamics on this set.

In Section \ref{sec:when_is_ homeo}, we establish two sufficient conditions for $\sigma$ not to be a local homeomorphism. We observe that, for the generalized countable Markov shifts in the first class of Section \ref{sec:examples}, the extended shift map is not necessarily a local homeomorphism, whereas it is a local homeomorphism for all generalized countable Markov shifts in the second and third classes.

In Section \ref{sec:Example_infinit_empty_words}, we construct a matrix $A$ whose associated generalized shift map admits a continuous extension to the entire space $X_A$ and for which the set of empty words $E_A$ is uncountable. We characterize the limit points of the columns of $A$, proving that they are in bijective correspondence with both $E_A$ and the Cantor space $\{0,1\}^\mathbb{N}$. Furthermore, we analyze the dynamics restricted to $E_A$, which turns out to be the inverse of the dyadic odometer. 

We began this paper when Iván Díaz-Granados was a Master's student (see \cite{DiazGranados2024}), and some of the results presented here are part of his Ph.D. dissertation.

\numsec=2\numfor=1
\section{Exel--Laca algebras and the generalized countable Markov shifts} \label{sec:Exel_Laca_GCMS}

Consider the alphabet $\mathbb{N}$ and a transition matrix $A = (A(i,j))_{(i,j)\in \mathbb{N}^2}$, with $A(i,j) \in \{0,1\}$. The (standard) \textit{countable Markov shift} is the set $\Sigma_A=\{x=(x_i)_{i \in\mathbb{N}_0} \in \mathbb{N}^{\mathbb{N}_0}: A(x_i,x_{i+1})=1 \ \text{for all} \ i\in\mathbb{N}_0\}$, where $\mathbb{N}_0 = \mathbb{N}\cup\{0\}$, equipped with the topology given by the metric $d(x,y) = 2^{-\min\{n\in\mathbb{N}_0:x_n\neq y_n\}}$ for $x\neq y$, and zero otherwise. Given a \textit{finite word} $\gamma \in \mathbb{N}^n$, $n \in \mathbb{N}$, $|\gamma| = n$ denotes its length, and its respective cylinder set is
$[\gamma] := \{x\in\Sigma_A : x_i=\gamma_i \text{ for all } 0 \leq i < n\}$. We say that $\gamma$ is \textit{admissible} when its respective cylinder is non-empty, and we recall that the topology generated by $d$ is the same one generated by the collection of the cylinder sets, which forms a basis of clopen sets. The dynamics is given by the \emph{shift map} $\sigma: \Sigma_A \to \Sigma_A$, defined by $\sigma(x)_i = x_{i+1}, \forall \ i \in \mathbb{N}_0$. In this paper, every matrix $A$ is \textit{irreducible}, that is, given $(i,j) \in \mathbb{N}^2$, there exists an admissible word $\gamma = x_0 x_1 \dots x_{n-1}$ such that $x_0=i$ and $x_{n-1}=j$. Then $[i]\neq \emptyset$ for all $i \in \mathbb{N}$, and the shift map is always topologically \textit{transitive}, which means that, given $(i,j) \in \mathbb{N}^2$, there exists $N=N(i,j) \in \mathbb{N}$ such that $[i]\cap \sigma^{-N}([j]) \neq \varnothing$; we will also say that $\Sigma_A$ is transitive. The matrix $A$ is \emph{row-finite} (resp.\ \emph{column-finite}) if each of its rows (resp.\ columns) has finitely many non-zero entries. A symbol $i\in\mathbb{N}$ is an \emph{infinite emitter} if the $i$-th row of $A$ has infinitely many non-zero entries, i.e., if $\{j\in\mathbb{N}: A(i,j)=1\}$ is infinite.

In the finite alphabet case, Markov shifts are intrinsically connected to their respective Cuntz--Krieger algebras $\mathcal{O}_A$ \cite{CuntzKrieger1980}; for instance, $\mathcal{O}_A$ contains a commutative C$^*$-subalgebra $\mathcal{D}_A$ whose Gelfand spectrum is homeomorphic to $\Sigma_A$, and hence $\mathcal{D}_A$ is isometrically $*$-isomorphic to $C(\Sigma_A)$. For a finite alphabet, $\Sigma_A$ is compact. However, it is well known that, for infinite countable alphabets, $\Sigma_A$ is locally compact if and only if $A$ is row-finite, and in this case the aforementioned C$^*$-subalgebra is isometrically $*$-isomorphic to $C_0(\Sigma_A)$. 

For an infinite countable alphabet, the generalization of the Cuntz--Krieger algebras beyond row-finiteness is the Exel--Laca algebra \cite{EL1999}, which also contains an object analogous to $\mathcal{D}_A$ -- for which we keep the notation of the compact case -- and which is again a maximal commutative C$^*$-subalgebra. The spectrum of $\mathcal{D}_A$ is what we call the \textit{generalized countable Markov shift}, denoted by $X_A$. When $\Sigma_A$ is not locally compact, the space $X_A$ plays the role of $\Sigma_A$, and we have $X_A=\Sigma_A \cup Y_A$ with $\Sigma_A$ and $Y_A:=X_A\setminus\Sigma_A$ dense subsets of $X_A$; for details, see \cite{BEFR2018,EL1999,Raszeja2020}.

Let $A$ be a transition matrix. The \textit{unital Exel--Laca algebra}, denoted by $\widetilde{\mathcal{O}}_A$, is the unital universal C$^*$-algebra generated by a family of partial isometries $\{S_i : i \in \mathbb{N}\}$, satisfying the following:

\begin{description}
  \renewcommand{\makelabel}[1]{\normalfont #1} 
  \item[(EL1)] $S_i^* S_i $ and $S_j^* S_j $ commute for every $i, j \in \mathbb{N}$;
  \item[(EL2)]  $S_i^* S_j = 0$ whenever $ i \neq j$;
  \item[(EL3)] $(S_i^* S_i)S_j = A(i, j)S_j$ for all $i, j \in \mathbb{N}$;
  \item[(EL4)] for every pair of finite sets $X, Y\subseteq\mathbb{N}$ such that the quantity
$$A(X, Y, i) \coloneqq \prod_{x \in X} A(x, i) \prod_{y \in Y} (1 - A(y, i)), \quad i \in \mathbb{N},$$ is non-zero only for a finite number of $i$'s, we have
$$\Biggl(\prod_{x \in X} S_x^* S_x\Biggr) \!\Biggl(\prod_{y \in Y} (1 - S_y^* S_y)\Biggr) = \sum_{i \in \mathbb{N}} A(X, Y, i) S_iS_i^*.$$
\end{description}
The C$^*$-subalgebra $\mathcal{O}_A\subseteq\widetilde{\mathcal{O}}_A$ generated by $\{S_i : i \in \mathbb{N}\}$ is called the \textit{Exel--Laca algebra}, and, under a certain condition, $\mathcal{O}_A$ coincides with $\widetilde{\mathcal{O}}_A$ (Proposition 8.5 in \cite{EL1999}). An alternative construction via the strong operator topology can be found in \cite{BoaCas2022}.

In the present manuscript, we use the faithful representation of the Exel--Laca algebra in the algebra of bounded operators $\mathfrak{B}(\ell^2(\Sigma_A))$ (Proposition 9.1 in \cite{EL1999}) to define $\mathcal{D}_A$ as follows (for details, see \cite{BEFR2018}).

\begin{definition}[Canonical Representation]\label{Definition: representation of Exel--Laca algebras} Let $\{\delta_\omega\}_{\omega\in\Sigma_A}$ be the canonical basis of $\ell^2(\Sigma_A)$. For each $s\in\mathbb{N}$, let $T_s\in\mathfrak{B}(\ell^2(\Sigma_A))$ be the partial isometry given by

\begin{equation*}
    T_s(\delta_\omega)=\begin{cases}
        \delta_{s\omega},& \text{if } A(s,\omega_0)=1  ,\\ 
        0,& \text{otherwise};\\
    \end{cases} \quad \text{and} \quad
    T_s^*(\delta_\omega)=\begin{cases}
        \delta_{\sigma(\omega)},& \text{if }\omega\in[s],\\
        0,& \text{otherwise};\\
    \end{cases}
\end{equation*}
By Proposition 9.1 of \cite{EL1999}, the assignment $S_i\mapsto T_i$, $i\in\mathbb{N}$, extends to a faithful $*$-representation $\pi:\mathcal{O}_A\to\mathfrak{B}(\ell^2(\Sigma_A))$, which we call the \emph{canonical representation}. We also define the projections $P_s=T_sT_s^*$ and $Q_s=T_s^*T_s$, which act as follows:
\begin{equation*}
    P_s(\delta_\omega)=\begin{cases}
                            \delta_\omega, & \text{if } \omega \in [s], \\
                            0, & \text{otherwise;} 
                        \end{cases}  \quad \text{and} \quad
    Q_s(\delta_\omega)=\begin{cases}
                        \delta_\omega, & \text{if } A(s,\omega_0)=1  ,\\ 
                            0, & \text{otherwise.}
                        \end{cases}
\end{equation*}
\end{definition}

\begin{remark}
Note that the Exel--Laca algebra, up to a $*$-isomorphism, is given by the image of $\mathcal{O}_A$ under the faithful representation $\pi$, which is the C$^*$-algebra generated by the set $\{T_s : s \in \mathbb{N}\}$, while the image of $\widetilde{\mathcal{O}}_A$ under the (unital extension of the) same representation $\pi$ is the C$^*$-algebra generated by $\{T_s : s \in \mathbb{N}\} \cup \{1\}$.
\end{remark}

An alternative description of the Exel--Laca algebra, as the closed linear span of a family of operators, can be found in \cite[Proposition 13]{BEFR2018}; we recall it here.

\begin{proposition}\label{Proposition: O_A is written with its generators}{\cite[Proposition 13 and Remark 14]{BEFR2018}}
    The Exel--Laca algebra can be written as $\pi(\mathcal{O}_A) = \overline{\mathcal{A}_{\mathcal{F}_A}}$, where $\mathcal{A}_{\mathcal{F}_A}$ is the $*$-subalgebra generated by $\mathcal{F}_A$, with
\begin{equation*}
    \mathcal{F}_A:=\left\{e_{\gamma,F, \beta} \in \mathfrak{B}(\ell^2(\Sigma_A)) : \begin{array}{l} F\subset \mathbb{N} \text{ finite }; \gamma, \beta \text{ admissible finite words};\\ F \neq \emptyset \text{ or } \gamma \text{ is not the empty word}  \text{ or } \\ 
    \beta \text{ is not the empty word} \end{array}\right\},
\end{equation*}
and $e_{\gamma,F, \beta} := T_\gamma \left(\prod_{j \in F}Q_j\right)T_\beta^*$.
\end{proposition}

\begin{remark}\label{remark:elements_of_F_A_acting_on_the_basis}
Note that the operators $e_{\gamma,F, \beta} \in \mathcal{F}_A$ act on the elements $\delta_\omega$ in the following way:
    \begin{equation}\label{eq: how the elements of F_A acts}
        \begin{split}
            e_{\gamma,F, \beta}(\delta_\omega)=\begin{cases}
            \delta_{\gamma \sigma^{|\beta|}(\omega)}, & \text{ if }\omega\in[\beta]\text{ and }A(j,\omega_{|\beta|})=1 \text{ for all }j\in F\cup\{\gamma_{|\gamma|-1}\}; \\
                \,\,\,0, & \text{ otherwise,}
        \end{cases}
        \end{split}
    \end{equation}
where we adopt the convention that $\{\gamma_{|\gamma|-1}\}=\varnothing$ when $\gamma$ is the empty word. In the case $\gamma=\beta$, note that $ e_{\gamma,F, \gamma}$ is a projection.
\end{remark}

\begin{remark}\label{rmk:properties_O_A} If $\mathcal{O}_A$ is unital, then $\mathcal{O}_A \simeq \widetilde{\mathcal{O}}_A$, otherwise $\widetilde{\mathcal{O}}_A$ is the canonical unitization of $\mathcal{O}_A$. In addition, the following are well-known properties of the operators $T_s$, $s \in \mathbb{N}$, \cite{EL1999}:
    \begin{enumerate}
        \item $T_s^*T_t=0$ and $P_sP_t=0$ if $s\neq t$;
        \item $Q_s$ and $Q_t$ commute for every $s,t$;
        \item $Q_sT_t=A(s,t)T_t$ and $T_t^*Q_s=A(s,t)T_t^*$;
    \end{enumerate}
\end{remark}

Let $\mathbb{F}_\mathbb{N}$ be the free group generated by the alphabet $\mathbb{N}$, and consider the partial group representation
\begin{align*}
    T:\mathbb{F}_\mathbb{N} &\to \pi\left(\widetilde{\mathcal{O}}_A\right),\\
    s &\mapsto T_s, \\
    s^{-1} &\mapsto T_{s^{-1}}:= T_s^*,
\end{align*}
$T_e = 1$, where $e$ is the identity, and 
\begin{equation*}
    g \mapsto T_g:= T_{x_1} \cdots T_{x_n},
\end{equation*}
for every $g=x_1\dots x_n \in \mathbb{F}_\mathbb{N}$ in its reduced form. The positive cone $\mathbb{F}^+_\mathbb{N} \subseteq \mathbb{F}_\mathbb{N}$ is the unital subsemigroup of \emph{positive words} of the group, that is, the set of elements whose reduced forms do not contain inverses of elements of $\mathbb{N}$, and we extend the notion of admissible word to $\mathbb{F}^+_\mathbb{N}$ in the obvious way. In particular, $e$ is admissible. We now describe a construction analogous to Proposition \ref{Proposition: O_A is written with its generators}, focusing on the commutative C$^*$-subalgebra $\mathcal{D}_A$, whose Gelfand spectrum is the generalized countable Markov shift.

\begin{definition}\label{I_A} Denote by $\mathbb{P}_F(\mathbb{N})$ the collection of finite subsets of $\mathbb{N}$. We define the set $I_A$ of \emph{admissible pairs} by
\begin{equation*}
    I_A := \left\{(\beta,F) \in \mathbb{F}^+_\mathbb{N} \times \mathbb{P}_F(\mathbb{N}): \beta \text{ is admissible}, \text{ and } \beta \neq e \text{ or } F \neq \varnothing\right\}.
\end{equation*}
\end{definition}

\begin{definition}\label{def:generators_D_A_indexed_by_pairs} Consider the commutative C$^*$-subalgebra $\mathcal{D}_A$ of $\mathcal{O}_A$, in the representation $\pi$, defined as  $\mathcal{D}_A := \overline{\mathcal{A}_{\mathcal{G}_A}} \subseteq \pi(\mathcal{O}_A)$, where $\mathcal{A}_{\mathcal{G}_A}$ is the $*$-subalgebra generated by $\mathcal{G}_A$, that is, $\mathcal{A}_{\mathcal{G}_A}:=\spann \, \mathcal{G}_A$, where
\begin{equation*}
    \mathcal{G}_A:=\left\{e_{\gamma,F} \in \mathfrak{B}(\ell^2(\Sigma_A)) : (\gamma,F) \in I_A\right\},
\end{equation*}
and $e_{\gamma,F} :=e_{\gamma,F,\gamma}= T_\gamma \left(\prod_{j \in F}Q_j\right)T_\gamma^*$, see Proposition \ref{Proposition: O_A is written with its generators}.
\end{definition}

\begin{remark}\label{Rem:How_act_operator_e_gamma,F}
In the particular case of Remark \ref{remark:elements_of_F_A_acting_on_the_basis}, note that the operators $e_{\gamma,F} \in \mathcal{G}_A\subseteq\mathcal{F}_A$ act on $\ell^2(\Sigma_A)$ in the following way:
    \begin{equation*}
        \begin{split}
            e_{\gamma,F}(\delta_\omega)=\begin{cases}
            \delta_\omega, & \text{ if }\omega\in[\gamma]\text{ and }A(i,\omega_{|\gamma|})=1 \text{ for all }i\in F; \\
                \,\,\,0, & \text{ otherwise.}
        \end{cases}
        \end{split}
    \end{equation*}
\end{remark}

\begin{remark}\label{rmk:trivial_terms_projections} Observe that the set of indices $I_A$ is merely used to avoid a class of trivial projections, since for every pair $(\gamma,F) \in  \mathbb{F}^+_\mathbb{N} \times \mathbb{P}_F(\mathbb{N}) \setminus (I_A \sqcup \{(e,\varnothing)\})$ we have $e_{\gamma,F} = 0 \in \mathcal{D}_A$. Note that when $\mathcal{D}_A$ is unital, by Proposition 8.5 of \cite{EL1999}, there exists a finite set $X \subseteq \mathbb{N}$ with $1=1-\prod_{x\in X}(1-Q_x)$, so that $1$ is a finite linear combination of elements of $\mathcal{G}_A$.
\end{remark}

\begin{proposition}\label{Proposition: D_A is MASA}
    $\mathcal{D}_A$ is a maximal commutative C$^*$-subalgebra of $\mathcal{O}_A$. Moreover, $\mathcal{D}_A$ is precisely the set of diagonal operators of $\pi(\mathcal{O}_A)$ with respect to the basis $\{\delta_x\}_{x \in \Sigma_A}$.
\end{proposition}
\begin{proof} Let $G = G(X_A,\sigma)$ be the  semidirect product groupoid (Definition 2.4 of \cite{Renault1999}), that is,
\begin{equation*}
    G = \left\{  (x,n-m,y) : n,m \in \mathbb{N}_0, x \in \Dom (\sigma^n), y \in \Dom (\sigma^m), \sigma^n(x) = \sigma^m(y)\right\},
\end{equation*}
whose topology is given by the basic open sets given by $$\mathcal{U}(U;n,m;V) =\left\{(x,n-m,y) \in G: (x,y) \in U \times V, \sigma^n(x) = \sigma^m(y)\right\},$$ where $n,m \in \mathbb{N}_0$, and $U \subseteq \Dom (\sigma^n)$ and $V \subseteq \Dom (\sigma^m)$ are open sets, making $G$ a locally compact Hausdorff second countable \'etale groupoid. Since $A$ is irreducible and the alphabet is infinite, every loop in its symbolic graph has an exit, and by Proposition 4.3 of \cite{Renault1999}, we have that $(\widetilde{X}_A, \sigma)$ is topologically principal, and so is $(X_A, \sigma)$. Propositions 2.4 and 4.8 of the same aforementioned paper give that $\mathcal{O}_A$ is isometrically $*$-isomorphic to the reduced groupoid C$^*$-algebra $C^*_r(G)$. Since $C^*_r(G) \simeq C^*_r(G; G \times \mathbb{T})$, where $G \times \mathbb{T}$ is the trivial twist (see Remark 11.1.12 of \cite{SSW2020}), by Theorem 4.2 (ii) of \cite{Renault2008Cartan} we conclude that $C_0(X_A) \simeq \mathcal{D}_A$ is a MASA in $\mathcal{O}_A$. For the second statement, we observe that the property of being a MASA does not depend on the faithful representation, so $\mathcal{D}_A$ is a MASA of $\pi(\mathcal{O}_A)$. Since the elements of $\mathcal{D}_A$ are diagonal operators with respect to the canonical basis, any diagonal operator of $\pi(\mathcal{O}_A)$ commutes with $\mathcal{D}_A$, and therefore belongs to $\mathcal{D}_A$ by maximality.
\end{proof}

\begin{remark} The notation $X_A$ in \cite{Renault1999} corresponds to $\widetilde{X}_A$ in the present work, also $X_A \setminus \{(\varnothing;\varnothing)\}$ corresponds to $X_A$ here. The explanation for these different notations is that we adopted most of ours from \cite{BEFR2018}. In addition, the assertion (i) in Proposition 4.3 of \cite{Renault1999} uses the term \emph{essentially principal}, which corresponds to the term \emph{topologically principal} in \cite{Renault2008Cartan}. We also highlight the role of the Hausdorff property of the groupoid $G$ in the proof of Proposition \ref{Proposition: D_A is MASA}: although the Hausdorff property is not essential for the construction of the C$^*$-algebra of a groupoid \cite{exelpitts:22}, it is crucial to ensure that $C_0(G^{(0)})$ is the MASA of $C_r^*(G)$; there are known examples of non-Hausdorff groupoids for which the algebra of continuous functions vanishing at infinity is not a MASA of the reduced C$^*$-algebra \cite{Exel2011}.
\end{remark}

Since $\mathcal{D}_A$ contains every diagonal operator of $\pi(\mathcal{O}_A)$ and the identity is diagonal, $\pi(\mathcal{O}_A)$ is unital if and only if $\mathcal{D}_A$ is unital, if and only if $X_A$ is compact. By Proposition 8.5 of \cite{EL1999} (see also Lemma 27 of \cite{BEFR2018}), this happens if and only if the zero column is not a limit point of the sequence of columns of $A$ in $\{0,1\}^{\mathbb{N}}$.

\begin{remark}\label{Remark:D_A_Tilde_Putting_Unity} A similar construction is done for the unital version of $\mathcal{D}_A$, denoted by $\widetilde{\mathcal{D}}_A \subseteq \pi(\widetilde{\mathcal{O}}_A)$, using the set $\widetilde{I}_A$ instead of $I_A$, where
\begin{equation*}
    \widetilde{I}_A = \left\{(\beta,F) \in \mathbb{F}^+_\mathbb{N} \times \mathbb{P}_F(\mathbb{N}): \beta \text{ is admissible}\right\}.
\end{equation*}
That is, $\widetilde{\mathcal{D}}_A:= \overline{\mathcal{A}_{\widetilde{\mathcal{G}}_A}}$, with $\widetilde{\mathcal{G}}_A:=\left\{e_{\gamma,F} \in \mathfrak{B}(\ell^2(\Sigma_A)) : (\gamma,F) \in \widetilde{I}_A\right\}$.
Observe that $I_A \subseteq \widetilde{I}_A$, and that if $\mathcal{D}_A$ is unital, then $\mathcal{D}_A \simeq \widetilde{\mathcal{D}}_A$.
\end{remark}

Now we present the generalization of the notion of Markov shift.

\begin{definition} Let $A$ be an irreducible matrix. The \emph{generalized countable Markov shift} $X_A$ is the Gelfand spectrum of $\mathcal{D}_A$, endowed with the weak$^*$ topology.
\end{definition}

As mentioned before, $\Sigma_A$ is a dense subset of $X_A$, and if $Y_A := X_A\setminus \Sigma_A$ is non-empty, then it is also dense. The latter set corresponds to a set of finite words, including empty words, and multiplicity of these words may occur, that is, there may exist different elements of $Y_A$ corresponding to the same word. The set $X_A$ is at least locally compact, and in many cases compact even when $\Sigma_A$ is not locally compact. The inclusion $i_1:\Sigma_A \hookrightarrow X_A$ is given by the evaluation maps $\varphi_\omega \vert_{\mathcal{D}_A}$, that is,
\begin{equation}\label{eq:inclusion_Sigma_A_in_X_A}
    i_1(\omega):=\varphi_\omega\vert_{\mathcal{D}_A}, \quad \text{where} \quad\varphi_\omega(S) := \left\langle S \delta_\omega, \delta_\omega\right\rangle, \, S \in \diag\mathfrak{B}(\ell^2(\Sigma_A)), \, \omega \in \Sigma_A,
\end{equation}
and $\left\langle \cdot, \cdot \right\rangle$ is the inner product in $\ell^2(\Sigma_A)$. This inclusion is a topological embedding \cite[Lemma 30]{BEFR2018}. Moreover, if $\Sigma_A$ is locally compact, i.e., $A$ is row-finite (under the irreducibility hypothesis), then $\Sigma_A$ coincides with $X_A$.

As a pictorial way of understanding the generalized space $X_A$, we can see it as a set of elements of $\{0,1\}^{\mathbb{F}_\mathbb{N}}$. In other words, we consider the oriented graph given by the Cayley tree generated by $\mathbb{N}$, where the edges between two vertices correspond to multiplication on the right by an element of $\mathbb{N}$, the opposite direction corresponding to multiplication by the respective inverse on the same side, as shown in the figure below.
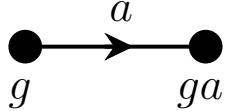
\begin{figure}[!ht]
    \centering
    \tikzset{every picture/.style={line width=0.75pt}} 

\begin{tikzpicture}[x=0.75pt,y=0.75pt,yscale=-1,xscale=1]

\draw  [fill={rgb, 255:red, 0; green, 0; blue, 0 }  ,fill opacity=1 ] (27,40) .. controls (27,35.58) and (30.58,32) .. (35,32) .. controls (39.42,32) and (43,35.58) .. (43,40) .. controls (43,44.42) and (39.42,48) .. (35,48) .. controls (30.58,48) and (27,44.42) .. (27,40) -- cycle ;
\draw  [fill={rgb, 255:red, 0; green, 0; blue, 0 }  ,fill opacity=1 ] (118,40) .. controls (118,35.58) and (121.58,32) .. (126,32) .. controls (130.42,32) and (134,35.58) .. (134,40) .. controls (134,44.42) and (130.42,48) .. (126,48) .. controls (121.58,48) and (118,44.42) .. (118,40) -- cycle ;
\draw [line width=1.5]    (43,40) -- (82,40) ;
\draw [shift={(86,40)}, rotate = 180] [fill={rgb, 255:red, 0; green, 0; blue, 0 }  ][line width=0.08]  [draw opacity=0] (10.4,-6.43) -- (-3,0) -- (10.4,6.44) -- (5.9,0) -- cycle    ;
\draw [line width=1.5]    (75,40) -- (118,40) ;

\draw (25,55) node [anchor=north west][inner sep=0.75pt]  [font=\LARGE]  {$g$};
\draw (76,15.4) node [anchor=north west][inner sep=0.75pt]  [font=\LARGE]  {$a$};
\draw (110,55) node [anchor=north west][inner sep=0.75pt]  [font=\LARGE]  {$ga$};

\end{tikzpicture}
\end{figure}

Any point of $\{0,1\}^{\mathbb{F}_\mathbb{N}}$ is called a \emph{configuration}, and the inclusion $X_A \hookrightarrow \{0,1\}^{\mathbb{F}_\mathbb{N}}$ is described next. 
By \cite[Proposition 17]{BEFR2018}, we also have
\begin{equation*}
    \widetilde{\mathcal{D}}_A = C^*\left(\left\{e_g:g \in \mathbb{F}_{\mathbb{N}}\right\}\right),
\end{equation*}
where $e_g= T_gT_g^*$, $g \in \mathbb{F}_{\mathbb{N}}$ in reduced form, are pairwise commuting projections. Since every $\xi\in X_A$ extends uniquely to a character of $\widetilde{\mathcal{D}}_A$, still denoted by $\xi$, every element $\xi\in X_A$ can be identified with the collection $(\xi_g)_{g\in\mathbb{F}_\mathbb{N}}$, where $\xi_g:=\xi(e_g)$; and since each $e_g$ is a projection, we conclude that $\xi_g \in \{0,1\}$. So $(\xi_g)_{g \in \mathbb{F}_{\mathbb{N}}} \in \{0,1\}^{\mathbb{F}_\mathbb{N}}$, and $\xi_g = \pi_g(\xi)$, where $\pi_g: \{0,1\}^{\mathbb{F}_\mathbb{N}} \to \{0,1\}$ is the canonical projection at the coordinate $g$. We say that $\xi$ \emph{is filled in} $g$ when its projection in $g$, denoted by $\xi_g$, equals $1$; otherwise we say that $\xi$ \emph{is not filled in} $g$. By endowing $\{0,1\}^{\mathbb{F}_\mathbb{N}}$ with the product topology, the inclusion $X_A \hookrightarrow \{0,1\}^{\mathbb{F}_\mathbb{N}}$ is a topological embedding \cite[Proposition 35]{BEFR2018}. The topology of $X_A$ is generated by the subbasis consisting of the \emph{generalized cylinders}
\begin{equation*}
    C_g := \left\{\xi \in X_A : \xi_g = 1 \right\}, \, g \in \mathbb{F}_\mathbb{N},
\end{equation*}
and their complements (for more details, see Section 4 of \cite{BEFR2018}). The elements of $X_A$ in such characterization are uniquely determined by what we call \emph{stem} and \emph{root} of a configuration \cite[Definition 43]{BEFR2018}, presented next.

\begin{definition} Let $\xi \in X_A$. The \emph{stem} $\kappa(\xi)$ of $\xi$ is the unique word $\omega$ (finite or not) that satisfies
\begin{equation*}
    \{g \in \mathbb{F}_\mathbb{N}^+: \xi_g = 1\} = \llbracket \omega \rrbracket,
\end{equation*}
where
\begin{equation*}
    \llbracket \omega \rrbracket = \{e,\omega_0, \omega_0 \omega_1, \omega_0 \omega_1 \omega_2 \dots\}
\end{equation*}
is the set of finite subwords of $\omega$, including $e$. If $\xi \in Y_A$, the \emph{root} $R_\xi$ of $\xi$ is the set
\begin{equation*}
    R_\xi := \{j \in \mathbb{N}: \xi_{\kappa(\xi)j^{-1}}=1\}
\end{equation*}
of labels of the edges of the Cayley tree of $\mathbb{F}_\mathbb{N}$ pointing to the vertex $\kappa(\xi)$ that are filled by $\xi$.
\end{definition}
The stem of $\xi = \varphi_{\omega}$, $\omega \in \Sigma_A$, is $\omega$ itself and it is uniquely determined. If $\xi \in Y_A$, then its stem is a finite word whose last letter is an infinite emitter, and different elements of $Y_A$ may share the same stem; however, the pair $(\kappa(\xi),R_\xi)$ determines $\xi$ uniquely. Moreover, $R_\xi$ corresponds to a limit point of the sequence $(A(\cdot,n))_{n \in \mathbb{N}}$ of columns of $A$ in $\{0,1\}^\mathbb{N}$ (product topology), via the correspondence $j \in R_\xi$ if and only if the $j$-th entry of the respective limit point equals $1$. In order to evaluate $\xi_g$ on every $g \in \mathbb{F}_{\mathbb{N}}$ for every $\xi \in X_A$, we recall that $X_A \subseteq \Omega_A^\tau \subseteq \{0,1\}^{\mathbb{F}_\mathbb{N}}$ \cite[Theorem 4.6]{EL1999} (see also \cite[Definition 39]{BEFR2018}), a compact set defined via the four rules below:
\begin{itemize}
    \item[$(R1)$] $\xi_e = 1$;
    \item[$(R2)$] $\xi$ is connected, in the sense that, for every $g,h \in \mathbb{F}_{\mathbb{N}}$, if $\xi_g = \xi_h = 1$, then $\xi$ is filled along the whole shortest vertex path in the Cayley tree between $g$ and $h$;
    \item[$(R3)$] for every $g \in \mathbb{F}_{\mathbb{N}}$ satisfying $\xi_g = 1$, there exists at most one symbol $i \in \mathbb{N}$ such that $\xi_{gi} = 1$;
    \begin{figure}[H]
        \centering
        \begin{center}
		\begin{tikzpicture}[scale=1.5,decoration={markings, mark=at position 0.5 with {\arrow{>}}}]
		\node[circle, draw=black, fill=black, inner sep=1pt,minimum size=5pt] (0) at (-3.5,0) {0};
		\node[circle, draw=black, fill=black, inner sep=1pt,minimum size=5pt] (1) at (-3.5,1) {1};
		\node[circle, draw=black, fill=black, inner sep=1pt,minimum size=5pt] (2) at (-2.5,0) {2};
    	\draw[postaction={decorate}, >=stealth] (0)  to (1);
    	\draw[postaction={decorate}, >=stealth] (0)  to (2);
   		\node[left] at (-3.5,0.5) {$j$};
   		\node[left] at (-3.6,0) {$g$};
   		\node[right] at (-2.4,0) {$g i$};
   	    \node[below] at (-3.0,0) {$i$};
   	    \node[above] at (-3.5,1.1) {$g j$};
        \node[below] at (-3.0,-0.5) {Forbidden filling};
        \node[circle, draw=black, fill=black, inner sep=1pt,minimum size=5pt] (3) at (-0.6,0) {3};
		\node[circle, draw=black, fill=white, inner sep=1pt,minimum size=10pt] (4) at (-0.6,1) {\textcolor{white}{4}};
		\node[circle, draw=black, fill=black, inner sep=1pt,minimum size=5pt] (5) at (0.4,0) {5};
    	\draw[postaction={decorate}, >=stealth] (3)  to (4);
    	\draw[postaction={decorate}, >=stealth] (3)  to (5);
   		\node[left] at (-0.6,0.5) {$j$};
   		\node[left] at (-0.7,0) {$g$};
   		\node[right] at (0.5,0) {$g i$};
   	    \node[below] at (0,0) {$i$};
   	    \node[above] at (-0.6,1.1) {$g j$};
        \node[circle, draw=black, fill=black, inner sep=1pt,minimum size=5pt] (6) at (1.5,0) {6};
		\node[circle, draw=black, fill=white, inner sep=1pt,minimum size=10pt] (7) at (1.5,1) {\textcolor{white}{7}};
		\node[circle, draw=black, fill=white, inner sep=1pt,minimum size=10pt] (8) at (2.5,0) {\textcolor{white}{7}};
    	\draw[postaction={decorate},>=stealth] (6)  to (7);
    	\draw[postaction={decorate}, >=stealth] (6)  to (8);
   		\node[left] at (1.5,0.5) {$j$};
   		\node[left] at (1.4,0) {$g$};
   		\node[right] at (2.6,0) {$g i$};
   	    \node[below] at (2,0) {$i$};
   	    \node[above] at (1.5,1.1) {$g j$};
   	    \node[below] at (1,-0.5) {Allowed fillings};
		\end{tikzpicture}
	\end{center}
        \caption{Forbidden and allowed fillings for configurations in $\Omega_A^\tau$ via rule R3.}
        \label{fig:R3}
    \end{figure}
    \item[$(R4)$] given $g \in \mathbb{F}_{\mathbb{N}}$ and $k \in \mathbb{N}$ satisfying $\xi_g = \xi_{gk} = 1$, then
    \begin{equation*}
        \xi_{gj^{-1}} = 1 \iff A(j,k) = 1, \, j \in \mathbb{N}.
    \end{equation*}
    \begin{figure}[h!]
        \begin{center}
		\begin{tikzpicture}[scale=1.5,decoration={markings, mark=at position 0.5 with {\arrow{>}}}]
		\node[circle, draw=black, fill=black, inner sep=1pt,minimum size=5pt] (0) at (0,0) {0};
		\node[circle, draw=black, fill=black, inner sep=1pt,minimum size=5pt] (1) at (0,1) {1};
		\node[circle, draw=black, fill=black, inner sep=1pt,minimum size=5pt] (2) at (1,1) {2};
    	\draw[postaction={decorate}, >=stealth] (0)  to (1);
    	\draw[postaction={decorate}, >=stealth] (1)  to (2);
   		\node[left] at (0,0.5) {$j$};
   	    \node[above] at (0.5,1) {$k$};
   	    \node[above] at (0,1.1) {$g$};
        \node[above] at (1,1.1) {$gk$};
        \node[left] at (-0.1,0) {$gj^{-1}$};
        \node[above] at (0.5,1.5) {$A(j,k)=1$};
       \node[circle, draw=black, fill=white, inner sep=1pt,minimum size=10pt] (3) at (3,0) {};
		\node[circle, draw=black, fill=black, inner sep=1pt,minimum size=5pt] (4) at (3,1) {4};
		\node[circle, draw=black, fill=black, inner sep=1pt,minimum size=5pt] (5) at (4,1) {5};
    	\draw[postaction={decorate},>=stealth] (3)  to (4);
    	\draw[postaction={decorate}, >=stealth] (4)  to (5);
   		\node[left] at (3,0.5) {$j$};
   	    \node[above] at (3.5,1) {$k$};
   	    \node[above] at (3,1.1) {$g$};
        \node[above] at (4,1.1) {$gk$};
        \node[left] at (2.9,0) {$gj^{-1}$};
        \node[above] at (3.5,1.5) {$A(j,k)=0$};
		\end{tikzpicture}
	\end{center}
        \caption{Fillings in $\Omega_A^\tau$ via rule R4.}
        \label{fig:R4}
    \end{figure}
\end{itemize}

\begin{definition} Let $F_A$ denote the set of configurations of $X_A$ whose stem is a finite non-empty word, and let $E_A$ denote the set of empty-stem (empty word) configurations. The shift map $\sigma:\Sigma_A \sqcup F_A \to X_A$ is defined as
\begin{equation}
    \sigma(\xi)_g := \xi_{\kappa(\xi)_0g}.
\end{equation}    
\end{definition}

\begin{remark}\label{rmk:preimages_generating_Y_A} In the definition above, $\Sigma_A \sqcup F_A$ is an open subset of $X_A$, and $\sigma$ is a local homeomorphism. In particular, every element of $Y_A$ belongs to $\sigma^{-n}(\xi^0)$ for some $\xi^0 \in E_A$ and $n \in \mathbb{N}_0$ \cite[Proposition 54]{BEFR2018}.
\end{remark}

In order to illustrate these configurations, we present some examples of $X_A$ for the renewal and the pair renewal shifts. Later in this paper, we will use them and other examples to show when it is possible to extend the shift map to the whole space and when it is not.

\begin{notation}
    Since $X_A=\Sigma_A \cup Y_A$, when an element $\xi=(\xi_g)_{g\in \mathbb{F}_{\mathbb{N}}}$ lies in $\Sigma_A$ we identify $i_1(\Sigma_A)$ with $\Sigma_A$ and $\xi$ with its stem, and we write $\omega = \kappa(\xi) \in \Sigma_A$. Accordingly, we also write $\sigma(\xi) = \sigma(\omega) = \sigma (\kappa(\xi))$. When $\xi=(\xi_g)_{g\in \mathbb{F}_{\mathbb{N}}}$ lies in $Y_A$, we identify $\xi$ with the stem-root pair and write $\xi = (\kappa(\xi),R_{\xi})$; in this case $\sigma(\xi) = (\sigma(\kappa(\xi)),R_{\xi})$. From this point on, we use these identifications without further mention.
\end{notation}

\begin{example}[Renewal shift]\label{exa:renewal_shift} The transition matrix for the renewal shift is given by $A(1,n) = A(n+1,n) = 1$, $n \in \mathbb{N}$, and zero otherwise. Its symbolic graph is shown next.

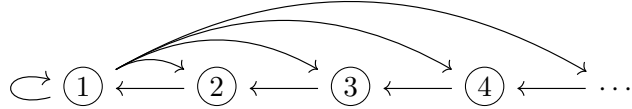
\[
\begin{tikzcd}
\circled{1}\arrow[loop left]\arrow[r,bend left]\arrow[rr,bend left]\arrow[rrr,bend left]\arrow[rrrr, bend left]&\circled{2}\arrow[l]&\circled{3}\arrow[l]&\circled{4}\arrow[l]&\arrow[l]\cdots
\end{tikzcd}
\]

The unique limit point of $(A(\cdot,n))_{n \in \mathbb{N}}$ is the column $(R_1)_j = \delta_{1,j}$, $j \in \mathbb{N}$; hence the only possible empty-stem configuration is the one whose root is the set $\{1\}$, and the non-empty-stem elements of $Y_A$ are in bijection with the words ending in $1$, endowed with the root above. 
\end{example}

\begin{example}[Lazy renewal shift \cite{MichiNakaHisaYoshi2022}] \label{exa:lazy_renewal_shift} The transition matrix for the lazy renewal shift is given by $A(1,n) = A(n+1,n) = A(n,n) = 1$, $n \in \mathbb{N}$, and zero otherwise. Its symbolic graph is shown below.

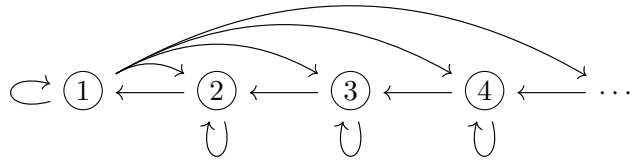
\[
\begin{tikzcd}
\circled{1}\arrow[loop left]\arrow[r,bend left]\arrow[rr,bend left]\arrow[rrr,bend left]\arrow[rrrr, bend left]&\circled{2}\arrow[loop below]\arrow[l]&\circled{3}\arrow[loop below]\arrow[l]&\circled{4}\arrow[loop below]\arrow[l]&\arrow[l]\cdots
\end{tikzcd}
\]

As in the renewal shift, the unique empty-stem configuration is the one whose root is the set $\{1\}$, and the non-empty-stem elements of $Y_A$ are in bijection with the words ending in $1$, their root being $\{1\}$ as well.
\end{example}

\begin{example}[Pair renewal shift] \label{exa:pair_renewal_shift} The pair renewal shift has a transition matrix as follows: $A(1,n) = A(2,2n) = A(n+1,n) = 1$, $n \in \mathbb{N}$, and zero otherwise. Its symbolic graph is presented below.
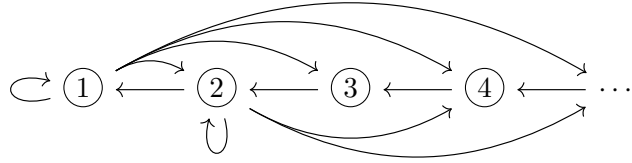
\[
\begin{tikzcd}
\circled{1}\arrow[loop left]\arrow[r,bend left]\arrow[rr,bend left]\arrow[rrr,bend left]\arrow[rrrr, bend left]&\circled{2}\arrow[l]\arrow[loop below]\arrow[rr,bend right, swap,]\arrow[rrr,bend right, swap,]&\circled{3}\arrow[l]&\circled{4}\arrow[l]&\arrow[l]\cdots
\end{tikzcd}
\]
Here, the sequence $(A(\cdot,n))_{n \in \mathbb{N}}$ has exactly two limit points, namely $(R_1)_j = \delta_{1,j} + \delta_{2,j}$ and $(R_2)_j = \delta_{1,j}$, representing the roots $\{1,2\}$ and $\{1\}$ respectively. By Remark \ref{rmk:preimages_generating_Y_A} we have two families of finite words, namely
\begin{equation*}
    Y_A^1 = \bigsqcup_{n \in \mathbb{N}_0} \sigma^{-n}(\{(e,R_1)\}) \; \text{ and } \; Y_A^2 = \bigsqcup_{n \in \mathbb{N}_0} \sigma^{-n}(\{(e,R_2)\}).
\end{equation*}
The sets $Y_A^1$ and $Y_A^2$ are disjoint, and within each of them the configurations are in bijection with their stems \cite[Proposition 54]{BEFR2018}. The set $Y_A^1$ is formed by an empty word together with the words ending in $1$ and in $2$, while $Y_A^2$ consists of an empty word together with the words ending in $1$. The empty words of the pair renewal shift are shown in Figure \ref{fig:empty_stem_pair_renewal}.

\begin{figure}[H]
\centering
\begin{subfigure}{.5\textwidth}
  \centering
  \input{figures/pair_renewal_1}
  \caption{Empty word configuration of $Y_A^1$.}
  \label{fig:pair_renewal_2_elements_in_root}
\end{subfigure}%
\begin{subfigure}{.5\textwidth}
  \centering
  \input{figures/pair_renewal_2}
  \caption{Empty word configuration of $Y_A^2$.}
  \label{fig:pair_renewal_1_element_in_root}
\end{subfigure}
\caption{Empty words configurations of the pair renewal shift\label{fig:empty_stem_pair_renewal}}
\end{figure}

\end{example}

\section{Representations for \texorpdfstring{$\mathcal{D}_A$}{DA}}\label{sec:Representations for D_A}
In this section, we present explicitly the Gelfand $*$-isomorphism between $\mathcal{D}_A$ and $C_0(X_A)$ for an arbitrary irreducible transition matrix $A$, by mapping the generators of one C$^*$-algebra to the generators of the other. It is well known that $\mathcal{D}_A$ and $C_0(X_A)$ are $*$-isomorphic, see for instance Theorem 8.4 of \cite{EL1999}, and that, in the finite alphabet case, $C(\Sigma_A)$ is the C$^*$-algebra generated by the projections $\{T_wT_w^*: w \text{ admissible word of } \Sigma_A\}$, see for example Proposition 2.5 of \cite{CuntzKrieger1980}. Also, for $X_A$, the algebra $\widetilde{\mathcal{D}}_A$ is generated by the projections $e_g:= T_gT_g^*$, $g \in \mathbb{F}_\mathbb{N}$ in reduced form, see Proposition 17 of \cite{BEFR2018}, which play a role similar to that of the projections $T_wT_w^*$. The $*$-isomorphism we construct follows the spirit of the ideas mentioned above: we present a family of characteristic functions of clopen sets that spans a dense $*$-subalgebra of $C_0(X_A)$ and associate them with a family of projections that generates $\mathcal{D}_A$.

Throughout this paper, we assume $A$ irreducible and we use the following notation.
\begin{definition} \label{def:V_beta_F}
Given $(\beta,F) \in I_A$, we define
\begin{equation}\label{eq:V_set}
    V_{\beta,F}:= C_\beta \cap \bigcap_{j \in F}C_{\beta j^{-1}},
\end{equation}
where $C_e = \bigcap_{j \in \varnothing}C_{\beta j^{-1}} = X_A$, and $C_{e j^{-1}} := C_{j^{-1}}$.
\end{definition}

\begin{remark} Although the intersection in \eqref{eq:V_set} might seem redundant, it is not, because $F$ may be empty, in which case $V_{\beta,\varnothing} = C_\beta$.  
\end{remark}

We illustrate the elements in $V_{\beta,F}$ from Definition \ref{def:V_beta_F} in Figure \ref{fig:V_beta_F_2_symbols}. 

\begin{figure}
    \centering
    \scalebox{.8}{

\tikzset{every picture/.style={line width=0.75pt}} 

\begin{tikzpicture}[x=0.75pt,y=0.75pt,yscale=-1,xscale=1]

\draw [color={rgb, 255:red, 128; green, 128; blue, 128 }  ,draw opacity=1 ][line width=3]  [dash pattern={on 7.88pt off 4.5pt}]  (442,153) -- (555.58,153.04) ;
\draw [color={rgb, 255:red, 128; green, 128; blue, 128 }  ,draw opacity=1 ][line width=3]    (322,153) -- (386,153) ;
\draw [shift={(392,153)}, rotate = 180] [fill={rgb, 255:red, 128; green, 128; blue, 128 }  ,fill opacity=1 ][line width=0.08]  [draw opacity=0] (18.75,-9.01) -- (0,0) -- (18.75,9.01) -- (12.45,0) -- cycle    ;
\draw [color={rgb, 255:red, 128; green, 128; blue, 128 }  ,draw opacity=1 ][line width=3]    (378.42,152.96) -- (442,153) ;

\draw [color={rgb, 255:red, 128; green, 128; blue, 128 }  ,draw opacity=1 ][line width=3]    (201,153) -- (265,153) ;
\draw [shift={(271,153)}, rotate = 180] [fill={rgb, 255:red, 128; green, 128; blue, 128 }  ,fill opacity=1 ][line width=0.08]  [draw opacity=0] (18.75,-9.01) -- (0,0) -- (18.75,9.01) -- (12.45,0) -- cycle    ;
\draw [color={rgb, 255:red, 128; green, 128; blue, 128 }  ,draw opacity=1 ][line width=3]    (257.42,152.96) -- (321,153) ;

\draw [color={rgb, 255:red, 128; green, 128; blue, 128 }  ,draw opacity=1 ][line width=3]    (80,153) -- (144,153) ;
\draw [shift={(150,153)}, rotate = 180] [fill={rgb, 255:red, 128; green, 128; blue, 128 }  ,fill opacity=1 ][line width=0.08]  [draw opacity=0] (18.75,-9.01) -- (0,0) -- (18.75,9.01) -- (12.45,0) -- cycle    ;
\draw [color={rgb, 255:red, 128; green, 128; blue, 128 }  ,draw opacity=1 ][line width=3]    (136.42,152.96) -- (200,153) ;

\draw  [fill={rgb, 255:red, 0; green, 0; blue, 0 }  ,fill opacity=1 ] (310,153) .. controls (310,147.48) and (314.48,143) .. (320,143) .. controls (325.52,143) and (330,147.48) .. (330,153) .. controls (330,158.52) and (325.52,163) .. (320,163) .. controls (314.48,163) and (310,158.52) .. (310,153) -- cycle ;
\draw  [fill={rgb, 255:red, 0; green, 0; blue, 0 }  ,fill opacity=1 ] (190,153) .. controls (190,147.48) and (194.48,143) .. (200,143) .. controls (205.52,143) and (210,147.48) .. (210,153) .. controls (210,158.52) and (205.52,163) .. (200,163) .. controls (194.48,163) and (190,158.52) .. (190,153) -- cycle ;
\draw  [fill={rgb, 255:red, 0; green, 0; blue, 0 }  ,fill opacity=1 ] (70,153) .. controls (70,147.48) and (74.48,143) .. (80,143) .. controls (85.52,143) and (90,147.48) .. (90,153) .. controls (90,158.52) and (85.52,163) .. (80,163) .. controls (74.48,163) and (70,158.52) .. (70,153) -- cycle ;
\draw [color={rgb, 255:red, 128; green, 128; blue, 128 }  ,draw opacity=1 ][line width=3]    (398.58,48.78) .. controls (391.38,79.26) and (390.04,101.58) .. (440.5,153.31) ;
\draw [color={rgb, 255:red, 128; green, 128; blue, 128 }  ,draw opacity=1 ][line width=3]    (398.58,48.78) .. controls (390.87,76.45) and (396.61,92.19) .. (400.9,100.73) ;
\draw [shift={(403.66,105.98)}, rotate = 243.69] [fill={rgb, 255:red, 128; green, 128; blue, 128 }  ,fill opacity=1 ][line width=0.08]  [draw opacity=0] (18.75,-9.01) -- (0,0) -- (18.75,9.01) -- (12.45,0) -- cycle    ;
\draw  [color={rgb, 255:red, 128; green, 128; blue, 128 }  ,draw opacity=1 ][fill={rgb, 255:red, 0; green, 0; blue, 0 }  ,fill opacity=1 ] (389.92,43.78) .. controls (392.69,38.99) and (398.8,37.35) .. (403.58,40.12) .. controls (408.37,42.88) and (410.01,48.99) .. (407.24,53.78) .. controls (404.48,58.56) and (398.37,60.2) .. (393.58,57.44) .. controls (388.8,54.68) and (387.16,48.56) .. (389.92,43.78) -- cycle ;

\draw [color={rgb, 255:red, 128; green, 128; blue, 128 }  ,draw opacity=1 ][line width=3]    (397.87,259.27) .. controls (391.53,228.29) and (390.19,205.97) .. (440.22,153.49) ;
\draw [color={rgb, 255:red, 128; green, 128; blue, 128 }  ,draw opacity=1 ][line width=3]    (397.87,259.27) .. controls (391.68,227.68) and (395.92,217.17) .. (401.58,205.56) ;
\draw [shift={(404.18,200.2)}, rotate = 115.3] [fill={rgb, 255:red, 128; green, 128; blue, 128 }  ,fill opacity=1 ][line width=0.08]  [draw opacity=0] (18.75,-9.01) -- (0,0) -- (18.75,9.01) -- (12.45,0) -- cycle    ;
\draw  [color={rgb, 255:red, 128; green, 128; blue, 128 }  ,draw opacity=1 ][fill={rgb, 255:red, 0; green, 0; blue, 0 }  ,fill opacity=1 ] (389.21,264.27) .. controls (386.44,259.49) and (388.08,253.37) .. (392.87,250.61) .. controls (397.65,247.85) and (403.76,249.49) .. (406.53,254.27) .. controls (409.29,259.06) and (407.65,265.17) .. (402.87,267.93) .. controls (398.08,270.7) and (391.97,269.06) .. (389.21,264.27) -- cycle ;

\draw  [fill={rgb, 255:red, 0; green, 0; blue, 0 }  ,fill opacity=1 ] (430,153) .. controls (430,147.48) and (434.48,143) .. (440,143) .. controls (445.52,143) and (450,147.48) .. (450,153) .. controls (450,158.52) and (445.52,163) .. (440,163) .. controls (434.48,163) and (430,158.52) .. (430,153) -- cycle ;

\draw (73,167.4) node [anchor=north west][inner sep=0.75pt]  [font=\huge]  {$e$};
\draw (190,167.4) node [anchor=north west][inner sep=0.75pt]  [font=\LARGE]  {$\beta _{0}$};
\draw (292,167.4) node [anchor=north west][inner sep=0.75pt]  [font=\LARGE]  {$\beta _{0} \beta _{1}$};
\draw (430,167.4) node [anchor=north west][inner sep=0.75pt]  [font=\LARGE]  {$\beta $};
\draw (383,15) node [anchor=north west][inner sep=0.75pt]  [font=\LARGE]  {$\beta i^{-1}$};
\draw (383,270.4) node [anchor=north west][inner sep=0.75pt]  [font=\LARGE]  {$\beta j^{-1}$};
\draw (131,116.4) node [anchor=north west][inner sep=0.75pt]  [font=\LARGE,color={rgb, 255:red, 128; green, 128; blue, 128 }  ,opacity=1 ]  {$\textcolor[rgb]{0.5,0.5,0.5}{\beta _{0}}$};
\draw (250,117.4) node [anchor=north west][inner sep=0.75pt]  [font=\LARGE]  {$\textcolor[rgb]{0.5,0.5,0.5}{\beta _{1}}$};
\draw (370,117.4) node [anchor=north west][inner sep=0.75pt]  [font=\LARGE]  {$\textcolor[rgb]{0.5,0.5,0.5}{\beta _{2}}$};
\draw (412,75.4) node [anchor=north west][inner sep=0.75pt]  [font=\LARGE]  {$\textcolor[rgb]{0.5,0.5,0.5}{i}$};
\draw (411,212.4) node [anchor=north west][inner sep=0.75pt]  [font=\LARGE]  {$\textcolor[rgb]{0.5,0.5,0.5}{j}$};

\end{tikzpicture}
}
    \caption{Description of the elements of $V_{{\beta,F}}$. Given, for instance, $\beta = \beta_0\beta_1\beta_2$ and $F = \{i,j\}$, $i \neq j$, the elements of $V_{{\beta,F}}$ are precisely the elements of $X_A$ whose set of filled vertices contains the configuration depicted above. }
    \label{fig:V_beta_F_2_symbols}
\end{figure}

\begin{lemma}\label{lemma:C_j^-1_compact} Let $j \in \mathbb{N}$ and let $\alpha$ be an admissible word (including the empty word $e$). Then $C_{\alpha j^{-1}}$ is a compact subset of $X_A$.    
\end{lemma}

\begin{proof} The proof follows similar steps as in Theorem 57 of \cite{BEFR2018}.
\end{proof}

\begin{lemma} For every $(\beta, F) \in I_A$, we have $\mathbbm{1}_{V_{\beta,F}} \in C_c(X_A)$. 
\end{lemma}

\begin{proof} For every $(\beta,F) \in I_A$, the set $V_{\beta,F}$ is either empty or a finite non-empty intersection of generalized cylinders, and hence clopen since generalized cylinders are clopen, and it is compact because of Lemma \ref{lemma:C_j^-1_compact} and Theorem 57 of \cite{BEFR2018}. We conclude that $\mathbbm{1}_{V_{\beta,F}} \in C_c(X_A)$.
\end{proof}

\begin{lemma} \label{lemma:span_U_beta_F_star_algebra} The space
\begin{equation*}
    \mathscr{B}_A = \spann \, \{\mathbbm{1}_{V_{\beta,F}}: (\beta, F) \in I_A\}
\end{equation*}
is a $*$-subalgebra of $C_0(X_A)$.    
\end{lemma}

\begin{proof} The vector space properties and the closure under the involution are straightforward. We prove that $\mathscr{B}_A$ is also closed under products. Let $(\beta, F),  (\gamma, G) \in I_A$; then
\begin{align*}
    U = \mathbbm{1}_{V_{\beta,F}} \mathbbm{1}_{V_{\gamma,G}} &= \mathbbm{1}_{C_\beta \cap C_\gamma} \mathbbm{1}_{\bigcap_{j \in F}C_{\beta j^{-1}}} \mathbbm{1}_{\bigcap_{k \in G}C_{\gamma k^{-1}}},
\end{align*}
and the product above vanishes unless either $\gamma \in \llbracket \beta \rrbracket$ or $\beta \in \llbracket \gamma \rrbracket$. Without loss of generality, assume $\gamma \in \llbracket \beta \rrbracket$. We have
\begin{align*}
    \mathbbm{1}_{V_{\beta,F}} \mathbbm{1}_{V_{\gamma,G}} &= \mathbbm{1}_{C_\beta} \mathbbm{1}_{\bigcap_{j \in F}C_{\beta j^{-1}}} \mathbbm{1}_{\bigcap_{k \in G}C_{\gamma k^{-1}}}.
\end{align*}
We divide the proof by cases.
\begin{itemize}
    \item[\textbf{Case (a):}] $\beta = e$. We necessarily have $\gamma = e$, and $F,G \neq \varnothing$. One gets
    \begin{align*}
        U = \mathbbm{1}_{\bigcap_{j \in F}C_{j^{-1}}} \mathbbm{1}_{\bigcap_{k \in G}C_{k^{-1}}} = \mathbbm{1}_{\bigcap_{j \in F \cup G}C_{j^{-1}}} = \mathbbm{1}_{V_{e,F \cup G}} \in \mathscr{B}_A.
    \end{align*}
    \item[\textbf{Case (b):}] $\beta \neq e$. We have
    \begin{align*}
        W := C_\beta \cap C_\gamma \cap \bigcap_{j \in F}C_{\beta j^{-1}} \cap \bigcap_{k \in G}C_{\gamma k^{-1}} = C_\beta \cap \bigcap_{j \in F}C_{\beta j^{-1}} \cap \bigcap_{k \in G}C_{\gamma k^{-1}}.
\end{align*}
    \begin{itemize}
        \item[\textit{Subcase (b1)}:] $\gamma = e$. We get 
        \begin{align*}
            W &= C_\beta \cap \bigcap_{j \in F}C_{\beta j^{-1}} \cap \bigcap_{k \in G} C_{k^{-1}}.
        \end{align*}
        If there exists $k \in G$ such that $A(k,\beta_0) = 0$, then $W = \varnothing$ because of (R1) and (R4), and hence $U = 0$. Otherwise, i.e., $A(k,\beta_0) = 1$ for every $k \in G$, one gets $C_{\beta_0} \subseteq C_{k^{-1}}$ and hence
        \begin{align*}
            W &= C_\beta \cap \bigcap_{j \in F}C_{\beta j^{-1}},
        \end{align*}
        therefore $U = \mathbbm{1}_{V_{\beta,F}} \in \mathscr{B}_A$.
        \item[\textit{Subcase (b2)}:] $\gamma \notin \{e,\beta\}$. We have $\gamma = \beta_0 \dots \beta_{|\gamma|-1}$, and so 
        \begin{align*}
            W &= C_\beta \cap \bigcap_{j \in F}C_{\beta j^{-1}} \cap \bigcap_{k \in G}C_{\beta_0 \dots  \beta_{|\gamma|-1} k^{-1}}.
        \end{align*}
        Similarly to the previous case, if there exists $k \in G$ such that $A(k, \beta_{|\gamma|}) = 0$, then $W = \varnothing$. Otherwise, $C_\beta \subseteq C_{\beta_0 \dots  \beta_{|\gamma|-1} k^{-1}}$ and hence
        \begin{align*}
            W &= C_\beta \cap \bigcap_{j \in F}C_{\beta j^{-1}},
        \end{align*}
        and we obtain $U = \mathbbm{1}_{V_{\beta,F}}$.
        \item[\textit{Subcase (b3)}:] $\gamma = \beta$. One gets
        \begin{align*}
            W &= C_\beta \cap \bigcap_{j \in F}C_{\beta j^{-1}} \cap \bigcap_{k \in G}C_{\beta k^{-1}} = C_\beta \cap \bigcap_{j \in F \cup G}C_{\beta j^{-1}},
        \end{align*}
        and we obtain $U = \mathbbm{1}_{V_{\beta,F \cup G}}$.
    \end{itemize}
\end{itemize}
\end{proof}

\begin{proposition}\label{prop:B_A_dense} Let $\mathscr{B}_A$ be as in the statement of Lemma \ref{lemma:span_U_beta_F_star_algebra}. Then $\overline{\mathscr{B}_A} = C_0(X_A)$.
\end{proposition}

\begin{proof} By Lemma \ref{lemma:span_U_beta_F_star_algebra}, $\mathscr{B}_A$ is a $*$-subalgebra of $C_0(X_A)$ and hence self-adjoint. We prove the following facts:
\begin{enumerate}
    \item for every $\xi \in X_A$, there exists $f \in \mathscr{B}_A$ such that $f(\xi) \neq 0$;
    \item $\mathscr{B}_A$ separates points in $X_A$.
\end{enumerate}
In fact, for every $\xi \in X_A$, there exists $i \in \mathbb{N}$ satisfying $\xi_{i^{-1}}=1$, and hence $\mathbbm{1}_{V_{e, \{i\}}}(\xi) = 1$ and (1) is proved. Indeed, if $\kappa(\xi)\neq e$, then $\xi_{i^{-1}}=1$ for every $i$ with $A(i,\kappa(\xi)_0)=1$ by rule (R4), and such an $i$ exists because $A$ is irreducible; if $\kappa(\xi)=e$, then $R_\xi=\{i:\xi_{i^{-1}}=1\}$ is non-empty, since $\xi$ is a non-zero character.

 Now, to prove (2), let $\xi^1, \xi^2 \in X_A$, $\xi^1 \neq \xi^2$. We have the following cases:
\begin{itemize}
    \item[$(2.a)$] $\xi^1, \xi^2 \in \Sigma_A$. Since $\Sigma_A$ is Hausdorff, there exist two disjoint cylinders $[\alpha^1]$ and $[\alpha^2]$, containing respectively $\xi^1$ and $\xi^2$. It is straightforward that $C_{\alpha^1} \cap C_{\alpha^2} = \varnothing$ and then
    \begin{equation*}
        \mathbbm{1}_{V_{\alpha^1, \varnothing}}(\xi^1) = 1 \neq 0 = \mathbbm{1}_{V_{\alpha^1, \varnothing}}(\xi^2);
    \end{equation*}
    \item[$(2.b)$] $\xi^1, \xi^2 \in Y_A$. If $\kappa(\xi^1) \neq \kappa(\xi^2)$, we may assume without loss of generality that $|\kappa(\xi^1)| \geq |\kappa(\xi^2)|$, and hence
    \begin{equation*}
        \mathbbm{1}_{V_{\kappa(\xi^1), \varnothing}}(\xi^1) = 1 \neq 0 = \mathbbm{1}_{V_{\kappa(\xi^1), \varnothing}}(\xi^2),
    \end{equation*}
    because in this case we have either $|\kappa(\xi^1)| > |\kappa(\xi^2)|$ or $|\kappa(\xi^1)| = |\kappa(\xi^2)|$ with at least one different symbol on the same coordinate. Now, if $\kappa(\xi^1) = \kappa(\xi^2) = \alpha$, then w.l.o.g. we may assume that there exists $i \in \mathbb{N}$, $i \neq \alpha_{|\alpha|-1}$ when $\alpha \neq e$, satisfying $\xi^{1}_{\alpha i^{-1}} = 1$ and $\xi^{2}_{\alpha i^{-1}} = 0$, and hence
    \begin{equation*}
        \mathbbm{1}_{V_{\alpha, \{i\}}}(\xi^1) = 1 \neq 0 = \mathbbm{1}_{V_{\alpha, \{i\}}}(\xi^2);
    \end{equation*}
    \item[$(2.c)$] $\xi^1 \in \Sigma_A$ and $\xi^2 \in Y_A$. Take $\alpha = \kappa(\xi^1)_0 \dots \kappa(\xi^1)_{n-1}$, $n >  |\kappa(\xi^2)|+1$. We obtain
    \begin{equation*}
        \mathbbm{1}_{V_{\alpha, \varnothing}}(\xi^1) = 1 \neq 0 = \mathbbm{1}_{V_{\alpha, \varnothing}}(\xi^2).
    \end{equation*}
\end{itemize}
Then, the assertion (2) is proved. By the Stone--Weierstrass theorem, we conclude that $\mathscr{B}_A$ is dense in $C_0(X_A)$. 
\end{proof}

Given an irreducible matrix $A$, it holds by definition that its respective generalized countable Markov shift $X_A$ is the spectrum of the commutative C$^*$-algebra $\mathcal{D}_A$, and hence the Gelfand representation theorem for commutative C$^*$-algebras gives $\mathcal{D}_A \simeq C_0(X_A)$. The next result presents an explicit construction of the $*$-isomorphism between them, using the generalized cylinder sets, in analogy with the construction involving the standard cylinders for Cuntz--Krieger algebras (see, for instance, \cite{KessStadStrat2007}).

\begin{theorem}[Gelfand isomorphism]\label{thm:isomorphism_C_0_X_A_and_D_A} The map $\Gamma: C_0(X_A) \to \mathcal{D}_A$, where $\Gamma(f)$ is the diagonal operator given by 
    \begin{equation}\label{eq:inverse_Gelfand}
        \Gamma(f)\delta_\omega := f(\varphi_\omega)\delta_\omega, \quad \omega \in \Sigma_A
    \end{equation}
    is an isometric $*$-isomorphism. In particular, $\Gamma( \mathbbm{1}_{V_{\beta, F}}) = e_{\beta,F}$.
\end{theorem}

\begin{proof} Let $\Upsilon : C_0(X_A) \to \mathfrak{B}(\ell^2(\Sigma_A))$ be the map where
\begin{equation}\label{eq:auxiliary_map_Gelfand}
    \Upsilon(f)\delta_\omega := f(\varphi_\omega) \delta_\omega, 
\end{equation}
where we observe that it is well defined since, for every $v = \sum_{\omega \in H_v} \lambda_\omega \delta_\omega \in \ell^2(\Sigma_A)$, $H_v = \{x \in \Sigma_A : \langle v,\delta_x\rangle \neq 0\}$, we have
\begin{align*}
    \|\Upsilon(f)(v)\|^2_{\ell^2(\Sigma_A)} &= \left\langle \Upsilon(f)(v),\Upsilon(f)(v) \right\rangle = \sum_{\omega \in H_v} |\lambda_\omega f(\varphi_\omega)|^2 \leq \|f\|^2 \sum_{\omega \in H_v} |\lambda_\omega |^2 = \|f\|^2  \| v \|^2_{\ell^2(\Sigma_A)}.
\end{align*}
Note that it is straightforward that $\Upsilon$ is an injective $*$-homomorphism, and hence it is closed. Now, we show that $\Upsilon( \mathbbm{1}_{V_{\beta, F}}) = e_{\beta,F}$ for every $(\beta,F) \in I_A$.

 In fact,
\begin{align*}
    \Upsilon( \mathbbm{1}_{V_{\beta, F}})\delta_\omega = \mathbbm{1}_{V_{\beta, F}}(\varphi_\omega) \delta_\omega = \mathbbm{1}_{\Sigma_A \cap V_{\beta, F}}(\varphi_\omega) \delta_\omega, \quad \omega \in \Sigma_A.
\end{align*}
Now, note that
\begin{equation*}
    \Sigma_A \cap C_{\beta j^{-1}} = \bigsqcup_{k: A(j,k)=1} [\beta k] = [\beta] \cap \bigsqcup_{k: A(j,k)=1} \sigma^{-|\beta|}([k]) = [\beta] \cap \sigma^{-|\beta|}(\sigma([j])),
\end{equation*}
and hence,
\begin{equation}\label{eq:identity_V_gamma_F_on_Sigma_A}
    \Sigma_A \cap V_{\beta,F} = [\beta] \cap \bigcap_{j \in F}\sigma^{-|\beta|}(\sigma([j])) = \bigcap_{j \in F} \bigsqcup_{\substack{p \in \mathbb{N} \\ A(j,p) = 1 }} [\beta p]  =  \bigsqcup_{\substack{p \in \mathbb{N} \\ A(j,p) = 1, \\ \forall j \in F }} [\beta p],
\end{equation}
where in the case $\beta = e$ we set $[ep] = [p]$. Then,
\begin{align*}
    \mathbbm{1}_{\Sigma_A \cap V_{\beta, F}}(\varphi_\omega) \delta_\omega = \begin{cases}
        \delta_\omega, \quad \text{if } \omega \in [\beta p] \text{ for some } p  \text{ such that } A(j,p) = 1 \text{ for every } j \in F; \\
        0, \quad \text{otherwise},
    \end{cases}
\end{align*}
that is,
\begin{align*}
    \mathbbm{1}_{\Sigma_A \cap V_{\beta, F}}(\varphi_\omega) \delta_\omega = \begin{cases}
        \delta_\omega, \quad \text{if } \omega \in [\beta] \text{ and } A(j,\omega_{|\beta|}) = 1, \text{ for every } j \in F; \\
        0, \quad \text{otherwise}.
    \end{cases}
\end{align*}
Then, by Remark \ref{Rem:How_act_operator_e_gamma,F} we conclude that $\Upsilon(\mathbbm{1}_{V_{\beta, F}}) = e_{\beta,F}$. This immediately gives that $\widetilde{\Gamma} := \Upsilon\vert_{\mathscr{B}_A}: \mathscr{B}_A \to \mathcal{A}_{\mathcal{G}_A}$ is a $*$-isomorphism, whose extension to $C_0(X_A)$ is the $*$-homomorphism $\Upsilon$. On the other hand, Proposition \ref{prop:B_A_dense} gives that $\mathscr{B}_A$ is dense in $C_0(X_A)$, and by definition, $\mathcal{A}_{\mathcal{G}_A}$ is dense in $\mathcal{D}_A$. Therefore, $\Upsilon(C_0(X_A)) = \mathcal{D}_A$, i.e., by restriction to its image, $\Gamma := \Upsilon : C_0(X_A) \to \mathcal{D}_A$ is a $*$-isomorphism between two C$^*$-algebras, and then it is isometric.

\end{proof}

\begin{remark} In \cite{BoaCasGonWyk2023}, G. Boava, G. G. de Castro, D. Gon\c{c}alves, and D. W. van Wyk constructed a notion of C$^*$-algebra $\widetilde{\mathcal{O}}_X$ associated to a subshift. In particular, they proved in Proposition 3.13 a result similar to the theorem above for the Ott--Tomforde--Willis (OTW) subshifts $X^{OTW}$ \cite{OttTomfordeWillis2014}. Although there are some cases where $X^{OTW}$ and $X_A$ may coincide, this is not a general occurrence. In particular, if $X_A$ has two distinct elements $\xi^1,\xi^2 \in Y_A$ with the same stem and different roots, the respective OTW subshift contains only one element $x$, which is a finite sequence and coincides with the stem of $\xi^1$ and $\xi^2$. This occurs, for instance, for the pair renewal shift and the finite words ending in $1$, which appear with the two roots $\{1\}$ and $\{1,2\}$.
\end{remark}

Given a commutative C$^*$-algebra $\mathfrak{A}$ with Gelfand spectrum $\widehat{\mathfrak{A}}$, and given $a \in \mathfrak{A}$, we denote by $\widehat{a} \in C_0(\widehat{ \mathfrak{A}})$ the Gelfand transformation of $a$, that is, $\widehat{a}(\eta) = \eta(a)$, $\eta \in \widehat{\mathfrak{A}}$.

\begin{remark}\label{rmk:Gelfand_transform_D_A} Note that, for every $\omega \in \Sigma_A$ and $(\gamma,F) \in I_A$, we have by \eqref{eq:inclusion_Sigma_A_in_X_A} that, when $\gamma\neq e$,
\begin{equation*}
    \widehat{e_{\gamma, F}}(\varphi_{\omega}) = \varphi_{\omega}(e_{\gamma, F}) = \left\langle e_{\gamma, F}\delta_\omega, \delta_\omega\right\rangle  = \begin{cases}
        1, \, \text{if } \omega \in [\gamma] \cap \bigcap_{j \in F}\sigma^{-|\gamma|}(\sigma([j])),\\
        0, \, \text{otherwise,}
      \end{cases}
\end{equation*}
and in the case that $\gamma=e$,
\begin{equation*}
    \widehat{e_{e, F}}(\varphi_{\omega}) = \varphi_{\omega}(e_{e, F}) = \left\langle e_{e, F}\delta_\omega, \delta_\omega\right\rangle  = \begin{cases}
        1, \, \text{if } \omega \in \bigcap_{j \in F}\sigma([j]),\\
        0, \, \text{otherwise.}
      \end{cases}
\end{equation*}
Then, by \eqref{eq:identity_V_gamma_F_on_Sigma_A}, $\widehat{e_{\gamma, F}}$ and $\mathbbm{1}_{V_{\gamma,F}}$ are continuous functions on $X_A$ that coincide in the same dense subset $\Sigma_A$, and therefore
\begin{equation}\label{eq:evaluating_e_on_varphi_omega_via_characteristic_functions}
    \widehat{e_{\gamma, F}} = \mathbbm{1}_{V_{\gamma,F}}.
\end{equation}
In other words, $\Gamma^{-1}$ is in fact the Gelfand transformation.
\end{remark}

The identity \eqref{eq:evaluating_e_on_varphi_omega_via_characteristic_functions} is very useful and will be applied several times in this work.

\begin{corollary}\label{cor:products_D_A} Let $(\beta,F), (\gamma, G) \in I_A$. If $\gamma \in \llbracket \beta \rrbracket$, then
    \begin{align*}
        e_{\beta,F}e_{\gamma,G} &= 
        \begin{cases}
            e_{\beta,F\cup G}, \; \text{if } \beta = \gamma; \\
            e_{\beta,F}, \; \text{if } \beta \neq \gamma \text{ and } A(k,\beta_{|\gamma|}) = 1, \; k \in G;\\
            0,\; \text{otherwise.}
        \end{cases}
    \end{align*}
If $\beta \in \llbracket \gamma \rrbracket$, the identity above holds in an analogous way, by exchanging $\beta$ with $\gamma$ and $F$ with $G$. The product $e_{\beta,F}e_{\gamma,G}$ is zero if $\gamma \notin \llbracket \beta \rrbracket$ and $\beta \notin \llbracket \gamma \rrbracket$.
\end{corollary}

\begin{proof} It is a straightforward consequence of the proof of Lemma \ref{lemma:span_U_beta_F_star_algebra} and Theorem \ref{thm:isomorphism_C_0_X_A_and_D_A}.
\end{proof}

\section{The canonical UCP map on Exel--Laca algebras}\label{sec: Canonical_UCP}

The concept of a \emph{completely positive map} is stronger than that of a positive
map. This notion goes back to W.~F. Stinespring \cite{Stinespring1955} and was systematically 
developed by W.~B. Arveson in \cite{Arveson1969}. One direct
application of completely positive maps lies in entropy theory for operator algebras.
In this context, D. Voiculescu defined a noncommutative entropy for completely
positive maps on nuclear C$^*$-algebras in \cite{voiculescu1995dynamical}. Years
later, N.~P. Brown extended Voiculescu's definition to the larger class of exact
C$^*$-algebras in \cite{brown1999topological}. That same year, M. Choda showed in
\cite{Choda1999} that the noncommutative entropy of the canonical endomorphism of
Cuntz algebras (equivalently, Cuntz--Krieger algebras associated with a matrix all of
whose entries are $1$) coincides with the topological entropy of the corresponding
shift map. This result was later extended to the canonical map on general
Cuntz--Krieger algebras by P. Boca and P. Goldstein \cite{BocaGoldstein2000}, and
further to noncommutative pressure by D. Kerr and C. Pinzari \cite{KerrPinzari2002}.
It is precisely these results that make the canonical map on Exel--Laca algebras
interesting: an analogous entropy statement in the infinite-alphabet setting
presupposes that such a map exists, which is the question we settle here. In addition,
it was recently proved that noncommutative entropy and topological entropy coincide
for the renewal shift \cite{MichiNakaHisaYoshi2022}; see also \cite{Sun2025}.

\begin{definition}
    We say that an element of a C$^*$-algebra is \textit{positive} when it is self-adjoint and its spectrum is contained in the interval $[0,\infty)$. Let $\mathfrak{A}$, $\mathfrak{B}$ be C$^*$-algebras, and denote by $\mathbb{M}_n(\mathfrak{A})$ and $\mathbb{M}_n(\mathfrak{B})$ the spaces of $n\times n$ matrices over $\mathfrak{A}$ and $\mathfrak{B}$, respectively. A linear map $\varphi:\mathfrak{A}\to \mathfrak{B}$ is called \textit{positive} if $\varphi(a)$ is positive for every positive $a$, and is \textit{completely positive} if, for every $n\in\mathbb{N}$, the induced map $\varphi_n: \mathbb{M}_n(\mathfrak{A})\to \mathbb{M}_n(\mathfrak{B})$, defined by 
$$\varphi_n([a_{i,j}])=[\varphi(a_{i,j})],$$
is positive, i.e., it sends positive matrices to positive matrices (in the sense of positive elements of the C$^*$-algebras $\mathbb{M}_n(\mathfrak{A})$ and $\mathbb{M}_n(\mathfrak{B})$). Moreover, if $\varphi$ preserves the unit, we call $\varphi$ a \textit{unital completely positive} (UCP) map. Now, let $\mathcal{H}$ be a Hilbert space. The \textit{strong operator topology} on $\mathfrak{B} (\mathcal{H})$ is the topology of pointwise convergence; precisely, a sequence (or net) $\{T_n\}_{n}$ in $\mathfrak{B} (\mathcal{H})$ converges to $T \in \mathfrak{B} (\mathcal{H})$ in this topology when $T_n \xi \to T \xi$ for every $\xi\in\mathcal{H}$. These concepts can be found, for example, in \cite{brown2008cstar, paulsen2002}.
\end{definition}

In this section, our goal is to extend the canonical map on Cuntz--Krieger algebras $\mathcal{O}_A$ defined by $X \mapsto\ \sum_{i=1}^N S_i X S_i^*$, where the finite set of partial isometries $\{S_i\}_{i=1}^N$ generates $\mathcal{O}_A$, to the setting of Exel--Laca algebras. However, for the Exel--Laca algebras, the series defining this map does not always converge in norm. For example, in the renewal shift, $\sum_{i=1}^\infty \pi(S_i)\pi(S_i)^* = \sum_{i=1}^\infty T_iT_i^* $ converges to unity in the strong operator topology, while $\sum_{i=1}^\infty S_iS_i^*$ itself does not converge in norm in $\mathcal{O}_A$. To overcome this difficulty, we use the canonical representation $\pi$ (see Definition~\ref{Definition: representation of Exel--Laca algebras}); in this representation, the map $\theta: \mathfrak{B}(\ell^2(\Sigma_A))\ni Y \ \mapsto\ \sum_{i=1}^\infty T_i Y T_i^*\in \mathfrak{B}(\ell^2(\Sigma_A))$ is well defined because the series defining it converges in the strong operator topology. Moreover, $\theta$ is a unital completely positive map. When $\mathcal{O}_A$ is unital and the image of $\theta$ is contained in $\pi(\mathcal{O}_A)$ (see Theorem \ref{Theorem: alpha well-defined iff sigma extended} for equivalent conditions), we define the \emph{canonical map} on the Exel--Laca algebra. This construction generalizes the canonical map on Cuntz--Krieger algebras in the case of a finite alphabet, as studied in \cite{BocaGoldstein2000, Choda1999, CuntzKrieger1980, KerrPinzari2002}.

\begin{theorem} \label{thm:UCP_map_extended}
Let $A$ be a irreducible transition matrix and consider its Exel--Laca algebra via the canonical representation on $\mathfrak{B}(\ell^2(\Sigma_A))$ as in Definition \ref{Definition: representation of Exel--Laca algebras}. The map $\widetilde{\theta} : \mathfrak{B}(\ell^2(\Sigma_A)) \to \mathfrak{B}(\ell^2(\Sigma_A))$ given by
\begin{equation} \label{eq:canonical_UCP_map_strong_series}
    \widetilde{\theta}(Y) = \sum_{i\in\mathbb N} T_i Y T_i^*,
\end{equation}
where $T_i = \pi(S_i)$, is well defined by strong operator convergence, contractive, linear, normal,
completely positive, and unital, that is, $\widetilde{\theta}(1)=1$.
\end{theorem}

\begin{proof} \noindent\emph{Well-definedness, contractivity and linearity.} Since each $T_i$ is a partial isometry, we have $P_iT_i=T_i$. Fix $Y \in \mathfrak{B} (\ell^2(\Sigma_A))$ and $\xi \in \ell^2(\Sigma_A)$. For each $i \in \mathbb{N}$, one gets $T_iYT_i^*\xi=T_iYT_i^*P_i\xi$. Also, note that
\begin{equation}\label{eq:PTYT_equals_TYT}
P_i T_i Y T_i^* \xi = T_i T_i^* T_i Y T_i^* \xi = T_i (T_i^* T_i ) YT_i^* \xi = T_i Y T_i^* \xi,
\end{equation}
and hence $T_i Y T_i^* \xi \in P_i\ell^2(\Sigma_A)$. Since the projections $P_i$ are pairwise orthogonal, the vectors $\{T_iYT_i^*\xi\}_{i\in\mathbb N}$ are pairwise orthogonal. Moreover,
\begin{equation*}
\|T_i Y T_i^* \xi\| = \|T_i Y T_i^* P_i \xi\| \leq \|Y\| \|T_i^* P_i \xi\| \leq \|Y\| \|T_i^*\| \|P_i \xi\| = \|Y\| \|P_i \xi\|.
\end{equation*}
Hence, for every finite set $F \subset \mathbb{N}$,
\begin{equation}\label{eq:uniform_estimate_UCP_map_series}
\left\|\sum_{i\in F}T_iYT_i^*\xi\right\|^2 = \sum_{i\in F}\|T_iYT_i^*\xi\|^2 \leq \|Y\|^2\sum_{i\in F}\|P_i\xi\|^2 \leq \|Y\|^2\|\xi\|^2.
\end{equation}
Since $\sum_{i \in \mathbb{N}}P_i=1$ in the strong operator topology, we have $\sum_{i\in\mathbb{N}} \|P_i\xi\|^2=\|\xi\|^2<\infty$, and hence the partial sums $(S_N\xi)_{N \in \mathbb{N}}$ given by
\begin{equation*}
S_N\xi := \sum_{i=1}^N T_i Y T_i^* \xi
\end{equation*}
form a Cauchy sequence in $\ell^2(\Sigma_A)$. Indeed, for $M > N$, we have
\begin{equation*}
\|S_M\xi - S_N\xi\|^2 = \left\|\sum_{i=N+1}^M T_i Y T_i^*\xi\right\|^2 = \sum_{i=N+1}^M \|T_i Y T_i^*\xi\|^2 \leq \|Y\|^2 \sum_{i=N+1}^M\|P_i\xi\|^2 \to 0, \quad N,M \to \infty.
\end{equation*}
So $(S_N\xi)_{N \in \mathbb{N}}$ is a Cauchy sequence in $\ell^2(\Sigma_A)$, and therefore it converges for every $\xi \in \ell^2(\Sigma_A)$. Then, $\widetilde{\theta}$ as in \eqref{eq:canonical_UCP_map_strong_series} is well defined in the strong operator topology. Moreover, the uniform estimate \eqref{eq:uniform_estimate_UCP_map_series} gives $\|\widetilde{\theta}(Y)\xi\| \leq \|Y\|\|\xi\|$ for every $\xi \in \ell^2(\Sigma_A)$, i.e., $\|\widetilde{\theta}(Y)\| \leq \|Y\|$, that is, $\widetilde{\theta}$ is contractive. It is immediate that $\widetilde{\theta}$ is linear.

\medskip
\noindent\emph{Normality and complete positivity.} This is an application of the Kraus Decomposition Theorem \cite{Kraus1971} (see also Theorem 2.20 of \cite{Fagnola1999Algebraic} or Theorem A.32 of \cite{Umanita2005Thesis}): a linear map $T: \mathfrak{B}(\ell^2(\Sigma_A)) \to \mathfrak{B}(\ell^2(\Sigma_A))$ is normal and completely positive if and only if it can be represented in the form
\begin{equation*}
    T(a) = \sum_{i \in \mathbb{\mathbb{N}}} V_i a V_i^*,
\end{equation*}
where $\{V_i\}_{i \in \mathbb{N}}$ is a collection of elements of $\mathfrak{B}(\ell^2(\Sigma_A))$ such that the series above converges strongly for every $a \in \mathfrak{B}(\ell^2(\Sigma_A))$.

\medskip
\noindent\emph{Unitality.} It is straightforward, since
\begin{equation*}
\widetilde{\theta}(1) = \sum_{i\in\mathbb N}T_iT_i^* = \sum_{i\in\mathbb N}P_i = 1
\end{equation*}
strongly.
\end{proof}

\begin{corollary}\label{cor:theta_restriction} The restriction $\theta:= \widetilde{\theta}\vert_{\pi(\widetilde{\mathcal{O}}_A)}: \pi(\widetilde{\mathcal{O}}_A) \to \mathfrak{B}(\ell^2(\Sigma_A))$, where $\widetilde{\theta}$ is the map from Theorem \ref{thm:UCP_map_extended}, is a contractive ultra-weakly continuous (and hence norm continuous) unital completely positive map. The same holds if we restrict $\widetilde{\theta}$ to $\pi(\mathcal{O}_A)$, except that the restriction is not unital when $\mathcal{O}_A$ is not unital.  
\end{corollary}

In \cite{MichiNakaHisaYoshi2022}, a class of matrices is identified for which the associated Exel--Laca algebra is unital and the map $\theta$ is guaranteed to be UCP (i.e., unital completely positive), with image contained in $\pi(\mathcal{O}_A)$. The natural question is under what conditions these results can be extended; this question is closely related to the continuous extension of the shift map, which is the focus of the next section. To approach this problem, we introduce the following definition.

\begin{definition}
    Suppose that $\theta(\pi(\widetilde{\mathcal{O}}_A))\subseteq\pi(\widetilde{\mathcal{O}}_A)$ (respectively, $\theta(\pi(\mathcal{O}_A))\subseteq\pi(\mathcal{O}_A)$ when $\mathcal{O}_A$ is unital). Since $\theta(1)=1$, the map $\Theta:=\pi^{-1}\circ\theta\circ\pi$ is a UCP map of $\widetilde{\mathcal{O}}_A$ (respectively, of $\mathcal{O}_A$), which we call the \emph{canonical map} on the unital Exel--Laca algebra $\widetilde{\mathcal{O}}_A$ (respectively, on the Exel--Laca algebra $\mathcal{O}_A$).
\end{definition}

\begin{remark}
   In several cases, infinitely many terms of the series $\sum_{i \in \mathbb{N}} T_i \pi(X) T_i^*$, where $X \in \mathcal{O}_A$, are nonzero, and the series converges in the strong operator topology. However, this does not imply that $\lim_{n \to \infty} \sum_{i=1}^n S_i X S_i^*$ converges in the norm of $\mathcal{O}_A$; indeed, it may fail to converge. Nevertheless, under the assumptions of the previous definition, we can guarantee the existence of $\Theta(X) = \pi^{-1}(\theta(\pi(X)))\in\mathcal{O}_A$, since $\pi$ is a $*$-isomorphism onto its image.
\end{remark}

Since $\theta$ is linear and continuous, the condition $\theta(\pi(\widetilde{\mathcal{O}}_A)) \subseteq \pi(\widetilde{\mathcal{O}}_A)$ (respectively, $\theta(\pi(\mathcal{O}_A)) \subseteq \pi(\mathcal{O}_A)$ when $\mathcal{O}_A$ is unital) is equivalent to $\theta(\mathcal{F}_A) \subseteq \pi(\widetilde{\mathcal{O}}_A)$ (respectively, $\theta(\mathcal{F}_A)\subseteq\pi(\mathcal{O}_A)$); see Proposition \ref{Proposition: O_A is written with its generators}. The next lemma gives an explicit formula for $\widetilde{\theta}(Y)$, $Y \in \mathfrak{B}(\ell^2(\Sigma_A))$.

\begin{lemma}\label{lemma:explicit_theta_on_bounded_operators} Let $Y \in \mathfrak{B}(\ell^2(\Sigma_A))$. Then, for every $\omega \in \Sigma_A$,
    \begin{equation}\label{eq:theta_on_general_element}
        \widetilde{\theta}(Y) \delta_\omega = T_{\omega_0} Y \delta_{\sigma(\omega)}.
    \end{equation}
In particular, for $e_{\gamma, F, \beta} \in \mathcal{F}_A$, we have
    \begin{equation}\label{eq:theta_on_generators}
        \theta(e_{\gamma, F, \beta}) \delta_\omega = \sum_{i\in \mathbb{N}} e_{i\gamma, F, i\beta} \delta_\omega = e_{\omega_0 \gamma, F, \omega_0 \beta} \delta_\omega.
    \end{equation}
\end{lemma}

\begin{proof} It is straightforward from \eqref{eq:canonical_UCP_map_strong_series} that
\begin{align*}
    \widetilde{\theta}(Y) \delta_\omega = \sum_{i\in\mathbb N} T_i Y T_i^* \delta_\omega = T_{\omega_0} Y T_{\omega_0}^* \delta_\omega = T_{\omega_0} Y \delta_{\sigma(\omega)},
\end{align*}
since $T_i^*\delta_\omega=0$ for $i\neq\omega_0$. Now, if $Y = e_{\gamma, F, \beta} \in \mathcal{F}_A$, we similarly have
\begin{align*}
    \theta(e_{\gamma, F, \beta}) \delta_\omega = \sum_{i\in\mathbb N} T_i e_{\gamma, F, \beta} T_i^* \delta_\omega = \sum_{i\in\mathbb N} e_{i\gamma, F, i\beta} \delta_\omega,
\end{align*}
and the last identity of \eqref{eq:theta_on_generators} is a direct consequence of Remark \ref{remark:elements_of_F_A_acting_on_the_basis}.
\end{proof}

\begin{remark}\label{Remark: equivalent forms of gamma_0}
The map $\theta$ can also be written in the following equivalent forms, which do not refer to the canonical basis and use the strong operator convergence of the series involved.

If $\gamma$ or $\beta$ is different from the empty word $e$, we may assume without loss of generality that $\gamma \neq e$. Then, we can write
\begin{equation}\label{ecuacion de theta_0 para el endomorpismo caso gamma no vacia}
    \theta(e_{\gamma,F,\beta}) = \sum_{i\in \mathbb{N}} e_{i\gamma,F,i\beta}.
\end{equation}
On the other hand, for $\gamma =\beta = e$, a direct calculation gives
    \begin{equation*}
        e_{e,F,e}=e_{e,F} = \sum_{\substack{k: A(j,k) = 1, \\ \forall j \in F}} e_{k,\varnothing}.
    \end{equation*}
   Hence,
    \begin{equation}\label{ecuacion de alpha0 para el endomorpismo caso gamma vacia}
        \theta(e_{e, F}) = \sum_{i \in \mathbb{N}} \sum_{\substack{k: A(j,k) = 1, \\ \forall j \in F}} e_{ik, \varnothing} = \sum_{i \in \mathbb{N}} \sum_{k \in \mathbb{N}} \left(\prod_{j \in F} A(j,k)  \right)e_{ik, \varnothing}.
    \end{equation}
    Using Remark \ref{remark:elements_of_F_A_acting_on_the_basis}, one can verify that formulas \eqref{ecuacion de theta_0 para el endomorpismo caso gamma no vacia} and \eqref{ecuacion de alpha0 para el endomorpismo caso gamma vacia} coincide with \eqref{eq:theta_on_generators} when $\gamma \neq e$ and when $\gamma = \beta = e$, respectively, by evaluating them for any given $\delta_{\omega}$, $\omega \in \Sigma_A$.
\end{remark}

Observe that for any pair $( \gamma,F) \in I_A$, it may happen that $(i\gamma,F) \notin I_A$, for some $i\in\mathbb{N}$, and in this case $e_{i\gamma, F} = 0$ (see Remark \ref{rmk:trivial_terms_projections}); the same holds for elements of $\mathcal{F}_A$, that is, for any operator $e_{\gamma,F,\beta}\in\mathcal{F}_A$ it may happen that $i\gamma$ or $i\beta$ is not admissible, in which case $e_{i\gamma,F,i\beta}=0$.

\section{Extending the Shift Map Continuously} \label{sec:dual_endomorphism}

In this section, we present a result of B. K. Kwa\'sniewski and A. Lebedev \cite{KwaLeb2020} and their notion of duality between a dynamical system on the spectrum of a commutative C$^*$-algebra and a $*$-endomorphism of that algebra. Some of our main results are contained in this section. We show that the restriction of $\theta$ from the previous section to $\mathcal{D}_A$ is a $*$-homomorphism, which becomes a $*$-endomorphism on $\mathcal{D}_A$ precisely when the shift map admits a continuous extension to $X_A$; in that case, the partial map dual to the endomorphism is exactly this continuous extension. In general, the shift map is defined on the open subset of non-empty words and does not admit a continuous extension to the whole space. We also prove some important properties of the shift map when its continuous extension does exist; for instance, that in this case the set of empty-stem configurations is invariant under the shift action and that the transition matrix must be column-finite.

The aforementioned duality comes from a proposition that we present next. It connects endomorphisms of commutative Banach algebras with partially defined dynamics on the spectrum of the algebra, in the sense that the domain of the dynamical map is a clopen subset of the spectrum. This result was proved by B. K. Kwa\'sniewski and A. Lebedev. In the following proposition, we present a particular case of Proposition 1.1 in \cite{KwaLeb2020}, along with its proof, in the context of C$^*$-algebras.

\begin{proposition}\label{proposicion partial map en A}
    Let $\alpha:\mathcal{D}\to \mathcal{D}$ be a $*$-endomorphism of a commutative unital C$^*$-algebra $\mathcal{D}$. There is a uniquely determined continuous partial map $\sigma: \Delta \to X$ on the spectrum $X$ of the algebra, defined on a clopen subset $\Delta\subseteq X$, such that
    \begin{equation}\label{ecuacion partial map dual de alpha A}
        \widehat{\alpha(a)}(x)=
        \left\{ \begin{array}{lcc}
             \widehat{a}(\sigma(x)), & x\in\Delta, \\
             0, & x\notin\Delta,
             \end{array}
   \right. \quad  a\in \mathcal{D}.
    \end{equation}
    Moreover, $\Delta=X$ iff $\alpha$ preserves the unit of $\mathcal{D}$.
\end{proposition} 
\begin{proof}
    Uniqueness follows by supposing that there are two distinct maps $\sigma$ and $\psi$ satisfying \eqref{ecuacion partial map dual de alpha A}. Let $x\in\Delta$ be such that $\sigma(x)\neq\psi(x)$, and let $a\in\mathcal{D}$ be such that $\widehat{a}(\sigma(x))\neq\widehat{a}(\psi(x))$ (the functions $\widehat{a}$, $a\in\mathcal{D}$, separate the points of $X$). Then $$\widehat{\alpha(a)}(x)=\widehat{a}(\sigma(x))\neq\widehat{a}(\psi(x))=\widehat{\alpha(a)}(x),$$
which is a contradiction. Define the map $\sigma':X\to X\cup \{0\}$ by $\sigma'(x)=x\circ \alpha$. For every $x\in X$, $\sigma'(x)$ is a multiplicative linear functional on $\mathcal{D}$, which may be zero. We therefore define $\sigma=\sigma'|_{\Delta}$, where $\Delta=(\sigma')^{-1}(X)$.
We can write $\Delta=\{x\in X:\widehat{\alpha(1)}(x)=1\}=\{x\in X:\widehat{\alpha(1)}(x)\neq0\}$,
because $\widehat{\alpha(1)}$ is the characteristic function of the set $\Delta$; in other words, $\Delta$ is simultaneously the inverse image of a closed set and of an open set under a continuous function. We conclude that $\Delta$ is clopen. It remains to show that $\sigma$ is continuous. Since $\sigma$ is the restriction of the map $\sigma'(\cdot) = (\cdot) \circ \alpha$, and $\sigma'$ is continuous due to the continuity of $\alpha$ (as a $*$-endomorphism of a C$^*$-algebra), it follows that $\sigma$ is continuous.
\end{proof}

The result above establishes a notion of dynamical duality between the algebraic and the topological sides of the picture, namely the algebra and its spectrum, which motivates the following definition. 

\begin{definition}\label{definicion dynamics dual to the endomorphism} 
    We call the map $\sigma: \Delta \subseteq X \to X$ satisfying \eqref{ecuacion partial map dual de alpha A} the \textit{partial dynamics dual to the endomorphism} $\alpha$. When $\Delta = X$, we refer to $\sigma: X \to X$ as the \textit{dynamics dual to the endomorphism} $\alpha$.
\end{definition}

\[
\begin{array}{c@{\quad}c}
\begin{tikzcd}[column sep=4em, row sep=4em]
\mathcal{D} \arrow[r, "\alpha"] \arrow[d, swap, "\Gamma^{-1}"] \arrow[dr, bend left=30] \arrow[dr, bend right=30]
& \mathcal{D}  \arrow[d, "\Gamma^{-1}"] \\
C(X)  \arrow[r, ""]
& C(X)
\arrow[phantom, from=1-1, to=1-2, "\;\; a \longmapsto \alpha(a)"{above, yshift=1.2em}]
\arrow[phantom, from=2-1, to=2-2, 
       "\;\; \widehat{a} \longmapsto \widehat{a}\circ \sigma"{below, yshift=-1.2em}]
\end{tikzcd}
\quad\quad
&
\widehat{\alpha(a)} = \widehat{a} \circ \sigma
\end{array}
\]

We mention that \cite{KwaLeb2020} also studies the converse direction: under suitable conditions an endomorphism can be recovered from the dual dynamics, and conversely (see Proposition 2.9 and Corollary 2.10 in \cite{KwaLeb2020}).

\begin{definition}\label{eq:endomorphism_alpha} We define the linear map $\alpha: \mathcal{D}_A \to \mathfrak{B}(\ell^2(\Sigma_A))$ given by $\alpha := \theta\vert_{\mathcal{D}_A}$.
\end{definition}

\begin{theorem}\label{thm:alpha_star_hom} The map $\alpha$ is a $*$-homomorphism.
\end{theorem}

\begin{proof} By Theorem \ref{thm:isomorphism_C_0_X_A_and_D_A}, for every $X \in \mathcal{D}_A$, there exists a unique $f \in C_0(X_A)$ such that $\Gamma(f) = X$, and then identity \eqref{eq:theta_on_general_element} in Lemma \ref{lemma:explicit_theta_on_bounded_operators} gives
\begin{align} \label{eq:identity_alpha_on_elements_of_D_A}
    \alpha(X)\delta_\omega &= \theta(X) \delta_\omega = T_{\omega_0} X \delta_{\sigma(\omega)} = T_{\omega_0} \Gamma(f) \delta_{\sigma(\omega)} \\
    &= T_{\omega_0} f(\varphi_{\sigma(\omega)}) \delta_{\sigma(\omega)} = f(\varphi_{\sigma(\omega)}) T_{\omega_0} \delta_{\sigma(\omega)} = f(\varphi_{\sigma(\omega)}) \delta_{\omega}. \nonumber
\end{align}
Hence, for every $X,Y \in \mathcal{D}_A$, let $f,g \in C_0(X_A)$ the unique functions such that $\Gamma(f) = X$ and $\Gamma(g) = Y$. One gets
\begin{align*}
    \alpha(XY) \delta_\omega &= (fg)(\varphi_{\sigma(\omega)}) \delta_{\omega} = g(\varphi_{\sigma(\omega)}) f(\varphi_{\sigma(\omega)}) \delta_{\omega} \\
    &= g(\varphi_{\sigma(\omega)}) \alpha(X) \delta_{\omega} = \alpha(X) (g(\varphi_{\sigma(\omega)})\delta_{\omega}) = \alpha(X) \alpha(Y) \delta_{\omega},
\end{align*}
for every $\omega \in \Sigma_A$, and we conclude that $\alpha$ is multiplicative. Now, note that $\Gamma(\overline{f}) = X^*$, and then
\begin{equation*}
    \langle\alpha(X) \delta_\omega,\delta_\eta\rangle = f(\varphi_{\sigma(\omega)}) \langle \delta_\omega,\delta_\eta\rangle = \langle \delta_\omega,\overline{f(\varphi_{\sigma(\omega)})}\delta_\eta\rangle = \langle \delta_\omega,\alpha(X^*)\delta_\eta\rangle,
\end{equation*}
for every $\omega,\eta \in \Sigma_A$. Hence $\langle\alpha(X) u ,v\rangle = \langle u,\alpha(X^*)v\rangle$, for every $u,v \in \ell^2(\Sigma_A)$, and we have $\alpha(X)^* = \alpha(X^*)$. Therefore $\alpha$ is a $*$-homomorphism.
\end{proof}

\begin{remark}\label{rmk:diagonal_elements_mapped_to_diagonal_elements} It is straightforward to notice that the elements in $\alpha(\mathcal{D}_A)$ are diagonal operators of $\mathfrak{B}(\ell^2(\Sigma_A))$ with respect to the canonical basis.    
\end{remark}

\begin{proposition}\label{prop:alpha_acting_on_D_A}
    For any $a \in \mathcal{D}_A$, the equality 
    $$\left\langle \alpha(a) \delta_\omega, \delta_\omega\right\rangle=\left\langle a \delta_{\sigma(\omega)}, \delta_{\sigma(\omega)}\right\rangle$$ 
    holds for every $\delta_\omega$, $\omega \in \Sigma_A$, in the canonical basis of $\ell^2(\Sigma_A)$. Moreover, when $\alpha(a)\in\mathcal{D}_A$, we have
    $$\widehat{\alpha(a)}(\varphi_\omega)=\widehat{a}(\sigma(\varphi_{\omega}))$$
    for every $\varphi_\omega\in i_1(\Sigma_A)$.
\end{proposition}
\begin{proof} By Theorem \ref{thm:isomorphism_C_0_X_A_and_D_A}, there exists a unique $f \in C_0(X_A)$ satisfying $\Gamma(f) = a$. Identity \eqref{eq:identity_alpha_on_elements_of_D_A} in the proof of Theorem \ref{thm:alpha_star_hom} gives, for every $\omega \in \Sigma_A$, that
\begin{align*}
    \left\langle \alpha(a) \delta_\omega, \delta_\omega\right\rangle = \left\langle f(\varphi_{\sigma(\omega)}) \delta_\omega, \delta_\omega\right\rangle = f(\varphi_{\sigma(\omega)}) = \left\langle \Gamma(f) \delta_{\sigma(\omega)}, \delta_{\sigma(\omega)}\right\rangle = \left\langle a \delta_{\sigma(\omega)}, \delta_{\sigma(\omega)}\right\rangle.
\end{align*}
    For the second identity, using the definition of the inclusion $i_1$ of $\Sigma_A$ into $X_A$ (see Equation \eqref{eq:inclusion_Sigma_A_in_X_A}), one gets
    \begin{equation*}
        \widehat{\alpha(a)}(\varphi_\omega)=\varphi_\omega(\alpha(a))= \left\langle \alpha(a) \delta_\omega, \delta_\omega\right\rangle=\left\langle a \delta_{\sigma(\omega)}, \delta_{\sigma(\omega)}\right\rangle=\widehat{a}(\varphi_{\sigma(\omega)})=\widehat{a}(\sigma(\varphi_{\omega})).
    \end{equation*}
\end{proof}

When $\mathcal{O}_A$ is unital, the existence of a continuous extension of the shift map to $X_A$ implies that the transition matrix $A$ is column-finite (Theorem \ref{Theorem: If sigma is extended continuously, then A is column finite}); the next remark will be used to prove this fact. As a consequence, the extended shift map leaves the set of empty words invariant, i.e., $\sigma(E_A)\subseteq E_A$ (Proposition \ref{thm:E_A_shift_invariant}).

\begin{remark} \label{rmk:A_irred_X_A_compact_existence_of_a_symbol_with_2_ranges} Let $A$ be an irreducible matrix with $X_A$ compact. Then there exists a symbol $i \in \mathbb{N}$ such that the set $\{k \in \mathbb{N}: A(i,k) = 1\}$ has at least two distinct elements. Indeed, if every row of $A$ had at most one non-zero entry, then $A$ would be row-finite, hence $X_A=\Sigma_A$; but $\Sigma_A$ is not compact, since the cylinders $[i]$, $i\in\mathbb{N}$, form an infinite cover by pairwise disjoint non-empty open sets.
\end{remark}

\begin{lemma}\label{lemma:exploding_first_coordinate_implies_limit_points_in_E_A} Suppose that $X_A$ is compact. A sequence $(\xi^n)_{n \in \mathbb{N}}$ in $X_A\setminus E_A$ has all of its limit points in $E_A$ if and only if $x^n_0 := \kappa(\xi^n)_0 \to \infty$.
\end{lemma}

\begin{proof} Suppose that every limit point of $(\xi^n)_{n \in \mathbb{N}}$ lies in $E_A$. If $x_0^n\not\to\infty$, there is a subsequence with $x_0^{n_k}=p \in \mathbb{N}$ constant; by compactness, a further subsequence converges to some $\xi\in X_A$, and $\xi_p=1$ because $C_p$ is closed; hence $\xi\notin E_A$, a contradiction. Conversely, assume that $x^n_0 \to \infty$, that is, for every $p \in \mathbb{N}$, there exists $N = N_p$ such that $(\xi^n)_p = 0$ when $n > N$. Hence, every limit point $\xi$ of $(\xi^n)_{n \in \mathbb{N}}$ satisfies $\xi_p = 0$ for every $p \in \mathbb{N}$ and therefore it belongs to $E_A$.
\end{proof}

\begin{remark} We observe that the converse part of the proof above does not use compactness, although it may be vacuous if the sequence has no limit points.
\end{remark}

\begin{notation}
    For $i\in\mathbb{N}$, $C^i$ denotes the $i$-th column of $A$ and $C^i_j$ the $j$-th component of $C^i$. For $D\in\{0,1\}^\mathbb{N}$, we use the notation $R_D:=\{i\in\mathbb{N}:D_i=1\}$, where $D_i$ is the $i$-th coordinate of $D$.
\end{notation}

\begin{lemma}\label{lemma:convergence_of_C^n_equivalent_to_convergence_of_infinite_sequences_to_E_A} Suppose that $X_A$ is compact, and let
 $(i_n)_n$ be a sequence in $\mathbb{N}$ with $i_n\to\infty$. The sequence of columns $(C^{i_n})_n$ of $A$ converges to $C\in\{0,1\}^\mathbb{N}$ if and only if every sequence $(\xi^n)_n$ in $X_A\setminus E_A$ with $\kappa(\xi^n)_0=i_n$ for all $n$ converges to $(e,R_C)\in E_A$.
 \end{lemma}

 \begin{proof} Let $(\xi^n)_n$ be a sequence as in the statement above. By Lemma \ref{lemma:exploding_first_coordinate_implies_limit_points_in_E_A}, we know that the limit points of $(\xi^n)_n$ are in $E_A$. Since 
 $$R_{C^{i_n}}=\{j\in\mathbb{N}: A(j,i_n)=1\}=\{j\in\mathbb{N}:(\xi^n)_{j^{-1}}=1\}$$ and $C^{i_n}\to C$, we obtain that $R_C$ is the unique root of the limit points of $(\xi^n)_n$, namely $R_C=\{k\in\mathbb{N}:\lim_{n\to\infty}(\xi^n)_{k^{-1}}=1\}$. We conclude that $\xi^n\to (e,  R_C)$. Conversely, suppose that $C^{i_n}\nrightarrow C$. Hence, there exists a subsequence $(n_k)_k$ of the natural numbers such that $C^{i_{n_k}}\to C'$, for some $C'\neq C$. By the analysis above, $(\xi^n)_n$ would then have more than one limit point and hence would not converge, a contradiction. 
\end{proof}

\begin{remark}\label{rmk:convergence_of_C^n_equivalent_to_convergence_of_infinite_sequences_to_E_A_emptywords_case} For the case of sequences of empty words, $(e,R_{D^n})\to(e,R_{\overline{C}})$ in $X_A$ whenever $D^n\to\overline{C}$ in $\{0,1\}^{\mathbb{N}}$, since all coordinates of $(e,R_D)$ are determined by $D$ through rules (R2)--(R4).
\end{remark}

\begin{remark}\label{rmk:non-compact_case} If $X_A$ is not compact, Lemmas \ref{lemma:exploding_first_coordinate_implies_limit_points_in_E_A} and \ref{lemma:convergence_of_C^n_equivalent_to_convergence_of_infinite_sequences_to_E_A} hold for $\widetilde{X}_A$ and $\widetilde{E}_A = E_A \cup \{(e,\varnothing)\}$. In particular, Lemma \ref{lemma:convergence_of_C^n_equivalent_to_convergence_of_infinite_sequences_to_E_A} holds in any case with the hypothesis $C \neq 0$ instead of the compactness of $X_A$.   
\end{remark}

\begin{theorem}\label{Theorem: If sigma is extended continuously, then A is column finite}
      Let $A$ be an irreducible transition matrix for which the Exel--Laca algebra $\mathcal{O}_A$ is unital and $\sigma$ is continuously extended to $X_A$. Then $A$ is column-finite.
\end{theorem}
\begin{proof}
    Suppose that $A$ is not column-finite, that is, there exists $j\in\mathbb{N}$ such that $\{ p\in\mathbb{N}: A(p,j)=1 \}$ is infinite. By Remark \ref{rmk:A_irred_X_A_compact_existence_of_a_symbol_with_2_ranges}, there exists $i\in\mathbb{N}$ such that $A(i, a)=A(i, b)=1$ for some distinct $a,b\in\mathbb{N}$. Now, let $\gamma$ be a finite admissible word (possibly empty) such that $j\gamma i$ is admissible, and take $\omega\in[\gamma i a]$ and $\omega'\in[\gamma i b]$. Since the $j$-th column of $A$ is infinite, we can choose $p_n\to\infty$ with $A(p_n,j)=1$ for every $n$ and, by compactness of $\{0,1\}^{\mathbb{N}}$, such that $(C^{p_n})_n$ converges to some $C$; as $X_A$ is compact, $C\neq 0$. By Lemma \ref{lemma:convergence_of_C^n_equivalent_to_convergence_of_infinite_sequences_to_E_A}, both sequences $(p_nj\omega)_n$ and $(p_nj\omega')_n$ converge to the empty word $(e,R_C)$. Then,
    $$\sigma(e,R_C)=\lim_{n\to\infty}\sigma(p_n j \omega)=j\omega\neq j\omega'=\lim_{n\to\infty}\sigma(p_n j \omega')=\sigma(e,R_C),$$ a contradiction.
\end{proof}

\begin{remark} The same phenomenon occurs for other compactifications. For the blur shift spaces of T. Z. Almeida and M. Sobottka \cite{Almeida_Sobottka:21}, continuity of the shift at an empty word indexed by an infinite set $H$ forces $A|_H$ to be column-finite; when $H$ is the whole alphabet, this recovers the Ott--Tomforde--Willis setting \cite{OttTomfordeWillis2014}, where column-finiteness is again necessary.
\end{remark}

\begin{proposition}\label{thm:E_A_shift_invariant} Suppose that $X_A$ is compact and $A$ is irreducible. If the shift dynamics is extended continuously to the whole $X_A$, then $E_A$ is an invariant set in the sense that $\sigma(E_A) \subseteq E_A$.
\end{proposition} 

\begin{proof} Let $\xi\in E_A$ and let $(\xi^n)_n$ be a sequence in $\Sigma_A$ converging to $\xi$ (recall that $\Sigma_A$ is dense in $X_A$). By Lemma \ref{lemma:exploding_first_coordinate_implies_limit_points_in_E_A}, $\kappa(\xi^n)_0\to\infty$. Since $A$ is column-finite (Theorem \ref{Theorem: If sigma is extended continuously, then A is column finite}) and $A(\kappa(\xi^n)_0,\kappa(\xi^n)_1)=1$, no value of $\kappa(\xi^n)_1$ can be repeated infinitely often, so $\kappa(\sigma(\xi^n))_0=\kappa(\xi^n)_1\to\infty$. By continuity, $\sigma(\xi^n)\to\sigma(\xi)$, and by Lemma \ref{lemma:exploding_first_coordinate_implies_limit_points_in_E_A} again, $\sigma(\xi)\in E_A$.
\end{proof}

The next theorem shows that $\alpha$ is a $*$-endomorphism of $\mathcal{D}_A$, that is, $\alpha(\mathcal{D}_A)\subseteq\mathcal{D}_A$, precisely when $\sigma$ admits a continuous extension to $X_A$.

\begin{theorem}\label{Theorem: alpha well-defined iff sigma extended}
     Let $A$ be a column-finite irreducible transition matrix for which the Exel--Laca algebra $\mathcal{O}_A$ is unital. The following statements are equivalent:
    \begin{enumerate}
        \item $\alpha(\mathcal{G}_A) \subseteq \mathcal{D}_A$.
        \item $\alpha$ is a $*$-endomorphism of $\mathcal{D}_A$.
        \item $\sigma: \Sigma_A\sqcup F_A\to X_A$ extends continuously to all of $X_A$.
         \item The canonical map is well defined on the Exel--Laca algebra $\mathcal{O}_A$.
        \item For all sequences $(i_n)_n,(j_n)_n$ in $\mathbb{N}$ with $i_n\to\infty$ and $A(i_n,j_n)=1$ for every $n$, if $(C^{i_n})_n$ converges in $\{0,1\}^{\mathbb{N}}$, then so does $(C^{j_n})_n$. In this case, writing $C=\lim C^{i_n}$ and $\overline{C}=\lim C^{j_n}$, the limit $\overline{C}$ depends only on $C$, and the extended shift satisfies $\sigma(e,R_C)=(e,R_{\overline{C}})$.
    \end{enumerate}
Moreover, when these conditions hold, $\alpha$ is unital, and the continuous extension of $\sigma$ is the dynamics dual to $\alpha$, i.e.\ $\widehat{\alpha(a)}=\widehat{a}\circ\sigma$ for all $a\in\mathcal{D}_A$.
\end{theorem}

\begin{proof}
We prove $(1)\Rightarrow(2)\Rightarrow(3)\Rightarrow(1)$, then $(3)\Leftrightarrow(5)$, then $(4)\Rightarrow(1)$, and finally that (1), (2) and (3) together imply (4). This is illustrated in the diagram below.
\[
\begin{tikzpicture}[every node/.style={circle, draw, minimum size=0.6cm, inner sep=2pt}, node distance=1.8cm]
    \node (1) {1};
    \node (2) [right of=1] {2};
    \node (3) [right of=2] {3};
    \node (5) [right of=3] {5};
    \node (4) [below of=2] {4};
    
    \draw[->] (1) -- (2);
    \draw[->] (2) -- (3);
    \draw[->] (3) to[bend right=30] (1);
    \draw[->] (3) to[bend right=15] (5);
    \draw[->] (5) to[bend right=15] (3);  
    \draw[->] (4) to[bend left=30] (1);
    
    \draw[->] (2) -- (4);
    
    \coordinate (mid) at ($(2)!0.5!(4)$);
    \draw (1) -- (mid);
    
    \draw (3) -- (mid);
\end{tikzpicture}
\]

    $(1)\implies(2)$ It is straightforward because $\mathcal{D}_A$ is a C$^*$-algebra and $\mathcal{A}_{\mathcal{G}_A}$ is dense in $\mathcal{D}_A$, and $\alpha$ is continuous because it is a $*$-homomorphism between C$^*$-algebras (Theorem \ref{thm:alpha_star_hom}). 
    
    $(2)\implies(3)$ Assume that $\alpha$ is a $*$-endomorphism on $\mathcal{D}_A$. By Proposition \ref{prop:alpha_acting_on_D_A}, we have $\widehat{\alpha(a)}(x)=\widehat{a}(\sigma(x))$ for every $x\in\Sigma_A$ and for every $a\in \mathcal{D}_A$. By Proposition \ref{proposicion partial map en A} (recall that $\mathcal{D}_A$ is unital, since so is $\mathcal{O}_A$), there exists a unique map $\varphi : \Delta \subseteq X_A \to X_A$, $\Delta$ clopen, such that $\widehat{\alpha(a)}(x) = \widehat{a}(\varphi(x))$ if $x \in \Delta$, and zero otherwise. We claim that $\Sigma_A \subseteq \Delta$. In fact, suppose there exists $y \in \Sigma_A \cap \Delta^c$, $y = y_0y_1y_2 \dots$, and consider the function $\mathbbm{1}_{C_{y_1}} \in C(X_A)$. Then we have 
    \begin{equation*}
        \widehat{\alpha(e_{y_1,\varnothing})}(y) = \mathbbm{1}_{C_{y_1}}(\sigma(y)) = 1. 
    \end{equation*}
    However, $y \in \Delta^c$ and hence $\widehat{\alpha(e_{y_1,\varnothing})}(y) = 0$, a contradiction. Therefore $\Sigma_A \subseteq \Delta$ and the claim is proved.
    Now, since $\Sigma_A$ is dense and $\Delta$ is closed, we have $\Delta = X_A$, and then, for every $f \in C(X_A)$, one gets $f \circ \sigma(x) = f \circ \varphi(x)$ for all $x \in \Sigma_A$, and then we prove that $\sigma = \varphi$ on $\Sigma_A$. 
    In fact, since $\sigma(x), \varphi(x)\in X_A = \widehat{\mathcal{D}}_A$, one gets $\sigma(x)(a)=\widehat{a}(\sigma(x))=\widehat{a}(\varphi(x))=\varphi(x)(a)$ for every $a\in\mathcal{D}_A$. Hence $\varphi$ and $\sigma$ agree on $\Sigma_A$. Since both maps are continuous on the open set $\Sigma_A\sqcup F_A$, in which $\Sigma_A$ is dense, and $X_A$ is Hausdorff, they also agree on $F_A$. Therefore $\varphi$ is a continuous extension of $\sigma$ to $X_A$, and it is the unique one, by density of $\Sigma_A$.

    $(3)\implies(1)$ Let $e_{\gamma,F}$ be an element of $\mathcal{G}_A$. Since $\sigma:X_A\to X_A$ is continuous, we have $\widehat{e_{\gamma,F}}\circ \sigma\in C(X_A)$, so we may consider $b\in \mathcal{D}_A$ such that $\widehat{b}=\widehat{e_{\gamma,F}}\circ \sigma$. For any $\omega\in\Sigma_A$, we have by Theorem \ref{thm:isomorphism_C_0_X_A_and_D_A} and identity \eqref{eq:identity_alpha_on_elements_of_D_A} that
    \begin{align*}
        b \delta_\omega = \widehat{b}(\varphi_\omega) \delta_\omega = \widehat{e_{\gamma,F}}(\sigma(\varphi_\omega)) \delta_\omega = \widehat{e_{\gamma,F}}(\varphi_{\sigma(\omega)}) \delta_\omega = \alpha(e_{\gamma,F}) \delta_\omega.
    \end{align*}
    Therefore $\alpha(e_{\gamma,F}) = b \in \mathcal{D}_A$.

    $(1),(2),(3)\implies (4)$ Take $\theta$ from Corollary \ref{cor:theta_restriction}. We must show that $\theta(\pi(\mathcal{O}_A))\subseteq\pi(\mathcal{O}_A)$, for which it is enough to show that $\theta(\mathcal{F}_A) \subseteq \pi(\mathcal{O}_A)$. Since $\alpha(\mathcal{G}_A) = \theta(\mathcal{G}_A) \subseteq \mathcal{D}_A$, it remains to verify that $\theta(e_{\gamma,F,\beta})\in \pi(\mathcal{O}_A)$ when $\gamma\neq e$ or $\beta\neq e$. We may assume, without loss of generality, that $\gamma\neq e$ and by equation \eqref{ecuacion de theta_0 para el endomorpismo caso gamma no vacia}, the series
$$\theta(e_{\gamma,F,\beta}) = \sum_{i\in \mathbb{N}}  e_{i\gamma,F,i\beta}$$
    has finitely many non-zero terms because $A$ is column-finite. Hence, $\theta(e_{\gamma,F,\beta})\in \pi(\mathcal{O}_A)$, as desired. Finally, $\theta(1)=\sum_{i\in\mathbb{N}}T_iT_i^*=\sum_{i\in\mathbb{N}}P_i=1$ (strong operator topology), so $\theta$ is unital. Therefore, $\Theta=\pi^{-1}\circ \theta\circ \pi$ is the canonical map on $\mathcal{O}_A$.

    $(4)\implies (1)$ Since $\alpha(\mathcal{G}_A)=\theta(\mathcal{G}_A)\subseteq\pi(\mathcal{O}_A)$ is a set of diagonal operators and $\mathcal{D}_A$ contains all diagonal operators of $\pi(\mathcal{O}_A)$ (Proposition \ref{Proposition: D_A is MASA}), we conclude that $\alpha(\mathcal{G}_A)\subseteq\mathcal{D}_A$.

$(3)\implies(5)$  Suppose that $\sigma$ is extended continuously. Consider a sequence $(\omega^n)_n$ in $\Sigma_A$ such that $\omega^n\in[i_nj_n]$, for every $n$. Since $i_n\to\infty$ and $C^{i_n}\to C$, by Lemma \ref{lemma:convergence_of_C^n_equivalent_to_convergence_of_infinite_sequences_to_E_A} we have $\omega^n\to(e,R_C)\in E_A$; consequently, $j_n=\sigma(\omega^n)_0\to\infty$ (because $A$ is column-finite) and $\sigma(\omega^n)\to\sigma(e,R_{C})=(e,R_{\overline{C}})\in E_A$ (see Proposition \ref{thm:E_A_shift_invariant}), for some limit point $\overline{C}$ of the columns of $A$. Moreover, by Lemma \ref{lemma:convergence_of_C^n_equivalent_to_convergence_of_infinite_sequences_to_E_A}, $\overline{C}$ is the limit of the sequence $(C^{j_n})_n$.

Before the next implication, we highlight that, in condition (5), for a fixed sequence $(i_n)_n$ with $i_n\to\infty$, \textit{any} sequence $(j_n)_n$ satisfying $A(i_n,j_n)=1$ for every $n$ will have the same limit point $\overline{C}\in \{0,1\}^\mathbb{N}$.

$(5)\implies(3)$  Define $\underline{\sigma}:X_A\to X_A$ by $\underline{\sigma}=\sigma$ on $\Sigma_A\sqcup F_A$ and $\underline{\sigma}(e,R_C):=\lim_{n\to\infty}\sigma(\omega^n)$ for $(e,R_C)\in E_A$, where $(\omega^n)_n$ is any sequence in $\Sigma_A$ converging to $(e,R_C)$. This limit exists: by Lemmas \ref{lemma:exploding_first_coordinate_implies_limit_points_in_E_A} and \ref{lemma:convergence_of_C^n_equivalent_to_convergence_of_infinite_sequences_to_E_A}, $\omega^n_0\to\infty$ and $C^{\omega^n_0}\to C$; by (5), $C^{\omega^n_1}\to\overline{C}$ for some $\overline{C}$; since $A$ is column-finite, $\omega^n_1\to\infty$, and Lemma \ref{lemma:convergence_of_C^n_equivalent_to_convergence_of_infinite_sequences_to_E_A} gives $\sigma(\omega^n)\to(e,R_{\overline{C}})\in E_A$. We now verify that $\underline{\sigma}$ is well defined.

Take two sequences $(\eta^n)_n$ and $(\zeta^n)_n$ in $\Sigma_A$ converging to $(e, R_C)\in E_A$, that is, the corresponding sequences of columns $(C^{\eta^n_0})_n$ and $(C^{\zeta^n_0})_n$ of $A$ converge to the same element $C$. Hence, by hypothesis, the sequences $(C^{\eta^n_1})_n$ and $(C^{\zeta^n_1})_n$ converge to $C^\eta$ and $C^\zeta$, respectively. Now, consider the sequence $(\xi^n)_n$ defined by $\xi^{2n-1}=\eta^n$ and $\xi^{2n}=\zeta^n$, for every $n$. Note that $C^{\xi^n_0}\to C$ and therefore $C^{\xi^n_1}\to C^\xi$. It follows that $C^\xi=C^\eta=C^\zeta$ and this shows that $\underline{\sigma}$ is a function.
    
Since $\underline{\sigma}$ coincides with $\sigma$ on the open set $\Sigma_A\sqcup F_A$, where $\sigma$ is continuous, and $X_A$ is metrizable, it suffices to check sequential continuity at the points of $E_A$. Moreover, a sequence converges to a given point if and only if every subsequence has a further subsequence converging to it; hence it suffices to treat sequences contained in one of the four sets $\Sigma_A$, $X_A\setminus(\Sigma_A\cup E_A\cup F_A^1)$, $F_A^1$ and $E_A$, since every sequence in $X_A$ splits into at most four subsequences of these types.
    \begin{itemize}
        \item[(a)] $(\xi^n)_n\subseteq\Sigma_A$. It is straightforward by the arguments above.
        \item[(b)] $(\xi^n)_n\subseteq X_A\setminus(\Sigma_A\cup E_A\cup F_A^1)$, with $F_A^1=\{(\gamma,R)\in Y_A: |\gamma|=1\}$. In this case, we can identify $\xi^n=(\gamma^n, R_n)$ where $|\gamma^n|\geq 2$, for every $n$, $C^{\gamma^n_0}\to C$ and, by hypothesis, $C^{\gamma^n_1}\to \overline{C}$. Therefore, $\sigma(\xi^n)=(\sigma(\gamma^n), R_n)\to (e, R_{\overline{C}})$.
        \item[(c)] $(\xi^n)_n\subseteq F_A^1$. We write $\xi^n=(i_n, R_{D^n})$, where $D^n \in \{0,1\}^\mathbb{N}$, $n \in \mathbb{N}$, and $C^{i_n}\to C$. Since $\xi^n \in X_A$ for each $n$, so are the elements $(e, R_{D^n})$. Then, there exists a sequence $(a_{n,m})_m\subseteq\mathbb{N}$ such that $C^{a_{n,m}}\stackrel{m\to\infty}{\to} D^n$, for every $n$, and $A(i_n, a_{n,m})=1$, for every $n,m\in\mathbb{N}$. Fix any metric $d_{\{0,1\}^\mathbb{N}}$ that generates the product topology in $\{0,1\}^\mathbb{N}$, and consider the sequence $(j_n)_n\subseteq \mathbb{N}$, with $j_n=a_{n,m(n)}$, where $m(n)$ is large enough that $d_{\{0,1\}^\mathbb{N}}(C^{j_{n}}, D^n)\leq \frac{1}{n}$, and consequently 
        $$d_{\{0,1\}^\mathbb{N}}(D^n, \overline{C})\leq d_{\{0,1\}^\mathbb{N}}(D^n, C^{j_n})+d_{\{0,1\}^\mathbb{N}}(C^{j_n}, \overline{C}) \leq \frac{1}{n}+d_{\{0,1\}^\mathbb{N}}(C^{j_n}, \overline{C}) \to 0.$$
        Thus, $\sigma(\xi^n)=(\sigma(i_n), R_{D^n})=(e, R_{D^n})\to (e,R_{\overline{C}})$.

        \item[(d)] $(\xi^n)_n\subseteq E_A$. Write $\xi^n=(e, R_{D^n})$. For every $n \in \mathbb{N}$, there exists $(\omega^{n,m})_m\subseteq\Sigma_A$ such that $\omega^{n,m}\stackrel{m\to\infty}{\to}(e, R_{D^n})$. Since $X_A$ is metrizable, consider any metric $d_{X_A}$ that generates its topology. Observe that, for each $n$, we may take $m_1(n)$ large enough that $d_{X_A}(\omega^{n,m},(e, R_{D^n}))\leq \frac{1}{n}$ for every $m\geq m_1(n)$. Then, we have 
        \begin{align*}
            d_{X_A}(\omega^{n,m},(e, R_{C}))&\leq d_{X_A}(\omega^{n,m},(e, R_{D^n}))+d_{X_A}((e, R_{D^n}),(e, R_{C}))\\
            &\leq \frac{1}{n}+d_{X_A}((e, R_{D^n}),(e, R_{C}))\to 0.
        \end{align*}
        So, $C^{\omega^{n,m}_0}\stackrel{n\to\infty}{\to} C$ and, consequently, $C^{\omega^{n,m}_1}\stackrel{n\to\infty}{\to} \overline{C}$ for every $m\geq m_1(n)$. On the other hand, we know that $\underline{\sigma}(e, R_{D^n})=\lim_{m\to\infty}\sigma(\omega^{n,m})=(e, R_{\overline{D^{n}}})$ and, as a result, $C^{\omega^{n,m}_1}\stackrel{m\to\infty}{\to}\overline{D^{n}}$. For each $n$, let $m_2(n)$ be large enough that $d_{\{0,1\}^\mathbb{N}}(C^{\omega^{n,m}_1},\overline{D^{n}})\leq \frac{1}{n}$, for every $m\geq m_2(n)$. Now, choose the sequence $(\eta^n)_n\subseteq \Sigma_A$, with $\eta^n=\omega^{n,m(n)}$ and $m(n)=\max\{m_1(n),m_2(n)\}$. Hence,
        \begin{equation*}
            d_{\{0,1\}^\mathbb{N}}(\overline{D^{n}}, \overline{C})\leq d_{\{0,1\}^\mathbb{N}}(\overline{D^{n}}, C^{\eta^n_1})+d_{\{0,1\}^\mathbb{N}}(C^{\eta^n_1}, \overline{C})\leq \frac{1}{n}+d_{\{0,1\}^\mathbb{N}}(C^{\eta^n_1}, \overline{C})\to 0.
        \end{equation*}
        We conclude $\underline{\sigma}(\xi^n)=\underline{\sigma}(e, R_{D^n})=(e, R_{\overline{D^{n}}})\to (e,R_{\overline{C}})$.
    \end{itemize}

\end{proof}

\begin{corollary}\label{cor:A_column_finite_sufficient_to_check_alpha_0_on_empty_words} Suppose that $A$ is irreducible and column-finite. Then, $\alpha$ is a $*$-endomorphism on $\mathcal{D}_A$ if and only if $\alpha(e_{e,F}) \in \mathcal{D}_A$ for every non-empty $F \in \mathbb{P}_F(\mathbb{N})$.
\end{corollary}
 
\begin{proof} If $\alpha$ is a $*$-endomorphism on $\mathcal{D}_A$, then it is clear that $\alpha(e_{e, F}) \in \mathcal{D}_A$ for every finite $F \subset \mathbb{N}$, $F \neq \varnothing$. Conversely, for $(\gamma,F)\in I_A$ with $\gamma\neq e$, the sum in \eqref{ecuacion de theta_0 para el endomorpismo caso gamma no vacia} is finite because the column $\gamma_0$ of $A$ is finite, so $\alpha(e_{\gamma,F})\in\mathcal{D}_A$; hence $\alpha(\mathcal{G}_A)\subseteq\mathcal{D}_A$, and the implication $(1)\Rightarrow(2)$ of Theorem \ref{Theorem: alpha well-defined iff sigma extended} (whose proof does not use that $\mathcal{O}_A$ is unital) gives that $\alpha$ is a $*$-endomorphism of $\mathcal{D}_A$.
\end{proof}

\begin{example}
 Consider the Markov shift space determined by the transition matrix $A$ with $A(n+1, n) = A(1, n) = 1$ and $A(2, 3n+1) = A(2, 3n-1) = 1$, for all $n \in \mathbb{N}$. Observe the following sequences in $\Sigma_A$
\begin{equation*}
    \omega^n = (3n+2)(3n+1)\cdots 1 1^\infty \quad \text{and} \quad \eta^n = (3n+1)(3n)\cdots 1 1^\infty
\end{equation*}
satisfy $\omega^n, \eta^n \to (e,\{1,2\})$, with $1^\infty = 1 1 \cdots$.
If we suppose $\widehat{\alpha(e_{e,\{2\}})}\in C_0(X_A)$, by continuity we should have
\begin{equation*}
    \lim_{n\to\infty}\widehat{\alpha(e_{e,\{2\}})}(\varphi_{\omega^n})=\lim_{n\to\infty}\widehat{\alpha(e_{e,\{2\}})}(\varphi_{\eta^n}).
\end{equation*}
However, by \eqref{eq:theta_on_generators}, we have
\begin{equation*}
    \widehat{\alpha(e_{e,\{2\}})}(\varphi_{\omega^n})=\langle e_{\omega^n_0,\{2\}}\delta_{\omega^n},\delta_{\omega^n}\rangle=A(2,\omega^n_1)=A(2,3n+1)=1,\\
\end{equation*}
and analogously, $\widehat{\alpha(e_{e,\{2\}})}(\varphi_{\eta^n}) = A(2,3n) = 0$,
a contradiction, and therefore $\alpha(e_{e,\{2\}}) \notin \mathcal{D}_A$, although it is a bounded operator in $\ell^2(\Sigma_A)$. In terms of condition (5) of Theorem \ref{Theorem: alpha well-defined iff sigma extended}, the obstruction is the following: the columns $C^{3n+1}$ and $C^{3n+2}$ both converge to $C \in \{0,1\}^{\mathbb{N}}$ such that $R_C = \{1,2\}$, but their successors $C^{3n}$ and $C^{3n+1}$ (note that $A(3n+1,3n)=A(3n+2,3n+1)=1$) converge to the different limits $C'$, $R_{C'} = \{1\}$, and $C$, respectively. Hence, condition (5) does not hold. Observe that $A$ is irreducible and column-finite, and $X_A$ is compact (the first row is full), so all the hypotheses of Theorem \ref{Theorem: alpha well-defined iff sigma extended} are satisfied.

\[
A=\bordermatrix{ &  &  &  &  &  &  &  &  &  &  \cr
               & 1 & 1 & 1 & 1 & 1 & 1 & 1 & 1 & 1 & \cdots \cr
               & 1 & 1 & 0 & 1 & 1 & 0 & 1 & 1 & 0 & \cdots \cr
               & 0 & 1 & 0 & 0 & 0 & 0 & 0 & 0 & 0 &  \cdots \cr
               & 0 & 0 & 1 & 0 & 0 & 0 & 0 & 0 & 0 &  \cdots \cr
               & 0 & 0 & 0 & 1 & 0 & 0 & 0 & 0 & 0 &  \cdots \cr
               & 0 & 0 & 0 & 0 & 1 & 0 & 0 & 0 & 0 &  \cdots \cr
               & \vdots & \vdots & \vdots & \vdots & \vdots & \vdots & \vdots & \vdots & \vdots & \ddots}
\]    
\end{example}

\begin{remark}\label{Remark: alpha well-defined iff sigma extended, E_A is non-unital} When $A$ is column-finite and $\mathcal{O}_A$ is not unital, a result similar to Theorem \ref{Theorem: alpha well-defined iff sigma extended} holds for the unital versions $\widetilde{\mathcal{D}}_A\subseteq\widetilde{\mathcal{O}}_A$ (see Remark \ref{Remark:D_A_Tilde_Putting_Unity}). Recall that the unit of $\widetilde{\mathcal{O}}_A$ is $e_{e,\varnothing}=1$ and that $\theta(1)=1$ by Theorem \ref{thm:UCP_map_extended}; hence $\widetilde{\alpha}:=\theta|_{\widetilde{\mathcal{D}}_A}$ is a unital $*$-homomorphism, by the same argument as in Theorem \ref{thm:alpha_star_hom}. In this case, the following statements are equivalent:
    \begin{enumerate}
        \item $\theta(\widetilde{\mathcal{G}}_A)\subseteq\widetilde{\mathcal{D}}_A$.
        \item $\widetilde{\alpha}$ is a $*$-endomorphism of $\widetilde{\mathcal{D}}_A$.
        \item $\sigma: \Sigma_A\sqcup F_A\to \widetilde{X}_A$ extends continuously to the whole of $\widetilde{X}_A$.
        \item The canonical map on the unital Exel--Laca algebra $\widetilde{\mathcal{O}}_A$ exists.
        \item For all sequences $(i_n)_n,(j_n)_n$ in $\mathbb{N}$ with $i_n\to\infty$ and $A(i_n,j_n)=1$ for every $n$, if $(C^{i_n})_n$ converges, then so does $(C^{j_n})_n$; in this case, writing $C=\lim C^{i_n}$ and $\overline{C}=\lim C^{j_n}$, the extended shift satisfies $\sigma(e,R_C)=(e,R_{\overline{C}})$, where $\overline{C}$ may be the zero column, in which case $(e,R_{\overline{C}})=(e,\varnothing)$.
    \end{enumerate}

In Section 6 of \cite{KwaLeb2020}, it was shown that in the case of a compact (standard) Markov shift $\Sigma_A$, when we consider the C$^*$-algebra $C(\Sigma_A)$, the $*$-endomorphism $\Psi:C(\Sigma_A) \to C(\Sigma_A)$ whose corresponding dual map is the shift dynamics is given by the rule $\Psi(f) = f \circ \sigma$, for every complex-valued continuous function $f$ on $\Sigma_A$. In a similar fashion, the same happens here when $\alpha$ is a $*$-endomorphism of $\mathcal{D}_A$, or equivalently, via Theorem \ref{Theorem: alpha well-defined iff sigma extended}, when $\sigma$ can be continuously extended to the whole space $X_A$.
 This is realized via the Gelfand isomorphism $\Gamma$ from Theorem \ref{thm:isomorphism_C_0_X_A_and_D_A}. The maps $\alpha$ and $\Gamma$ induce a $*$-endomorphism $\Psi : C_0(X_A) \to C_0(X_A)$ which makes the diagram
\[
\begin{tikzcd}
\mathcal{D}_A \arrow[r, "\alpha"] 
& \mathcal{D}_A  \arrow[d, "\Gamma^{-1}"] \\
C_0(X_A) \arrow[u, "\Gamma"] \arrow[r, "\Psi"]
& C_0(X_A)
\end{tikzcd}
\]
commutes, that is, $\Psi = \Gamma^{-1} \circ \alpha \circ \Gamma$. This is the content of the next theorem.

\begin{theorem}\label{Theorem:shift_dual_is_theta} Let $A$ be a column-finite irreducible transition matrix for which the Exel--Laca algebra $\mathcal{O}_A$ is unital. Suppose that $\sigma$ is continuously extended to $X_A$, then $\Psi(f) = f \circ \sigma$ for every $f \in C(X_A)$.    
\end{theorem}

\begin{proof} The proof that $(3)$ implies $(1)$ in Theorem \ref{Theorem: alpha well-defined iff sigma extended} shows that $\widehat{\alpha(e_{\gamma, F})}=\widehat{e_{\gamma, F}}\circ\sigma$ for every $(\gamma, F) \in I_A$. Then,
\begin{align*}
    \Psi(\mathbbm{1}_{V_{\gamma,F}}) &= \Gamma^{-1} \circ \alpha \circ \Gamma (\mathbbm{1}_{V_{\gamma,F}}) = \Gamma^{-1} \circ \alpha (e_{\gamma, F}) = \widehat{\alpha(e_{\gamma, F})} = \widehat{e_{\gamma, F}}\circ\sigma = \mathbbm{1}_{V_{\gamma,F}}\circ\sigma = \Phi_\sigma(\mathbbm{1}_{V_{\gamma,F}}),
\end{align*}
where $\Phi_\sigma$ is the $*$-endomorphism $\Phi_\sigma:  C(X_A) \to  C(X_A)$, $\Phi_\sigma(f) :=  f\circ \sigma$, $f \in  C(X_A)$. Since $\Phi_\sigma$ and $\Psi$ coincide on the dense $*$-subalgebra $\spann \{\mathbbm{1}_{V_{\gamma,F}}: (\gamma,F) \in I_A\} \subseteq C(X_A)$, we conclude that they coincide, that is, $\Psi(f) = f \circ \sigma$ for every $f \in C(X_A)$. 
\end{proof}

\begin{remark}
In the case where $\mathcal{O}_A$ is not necessarily unital, suppose that the shift map $\sigma$ extends continuously to $\widetilde{X}_A$. Then, analogously to Theorem \ref{Theorem:shift_dual_is_theta}, the $*$-endomorphism
$\Psi=\Gamma^{-1} \circ \alpha \circ \Gamma$ of $C(\widetilde{X}_A)$, where $\alpha$ is the $*$-endomorphism of $\widetilde{\mathcal{D}}_A$ introduced in Remark \ref{Remark: alpha well-defined iff sigma extended, E_A is non-unital}, satisfies $\Psi(f) = f \circ \sigma$ for every $f \in C(\widetilde{X}_A)$.
\end{remark}

\section{Generalized countable Markov shifts with extendable shift maps}\label{sec:examples}

In this section, we present three classes of generalized countable Markov shifts, named \emph{single empty word class}, \emph{periodic renewal class} and \emph{permutation renewal class}, for which we prove explicitly that the dynamics can be continuously extended to the whole space. The three classes are governed by what the dynamics does to the finitely many empty words: the single empty-word class forces a fixed point, the periodic renewal class a single cycle, and the permutation renewal class an arbitrary permutation. They overlap but none contains another, as Figure \ref{fig:venn_renewal_classes} shows. We point out that, in our context, the notion of renewal class is broader than the usual one for standard Markov shifts \cite{Iommi2006, Sarig2001PT}. In addition, we present examples in the three classes, as well as counterexamples showing what may happen when some of the requirements for belonging to these families are removed.

\begin{definition}\label{def:Single_Empty_Word_Class} A generalized countable Markov shift $X_A$ belongs to the \emph{single empty-word class} when the following holds:
\begin{itemize}
    \item[$(i)$] $A$ is irreducible and column-finite;
    \item[$(ii)$] $X_A$ is compact;
    \item[$(iii)$] $E_A$ is a singleton.
\end{itemize}
\end{definition}

Examples of single empty-word generalized countable Markov shifts are the renewal shift (Example \ref{exa:renewal_shift}) and the lazy renewal shift (Example \ref{exa:lazy_renewal_shift}).

\begin{theorem} Let $X_A$ be a generalized countable Markov shift in the single empty-word class. The extension of $\sigma$ given by $\sigma(\xi^0) = \xi^0$, where $\xi^0$ is the empty word of $X_A$, is continuous. 
\end{theorem}

\begin{proof} Since $\sigma$ is continuous on the open set $\Sigma_A\sqcup F_A$ and $X_A$ is metrizable, it suffices to show that $\sigma(\xi^n)\to\xi^0$ whenever $\xi^n\to\xi^0$ with $\xi^n\neq\xi^0$ for all $n$. As $E_A=\{\xi^0\}$, we have $\xi^n\notin E_A$, so $\kappa(\xi^n)_0\to\infty$ by Lemma \ref{lemma:exploding_first_coordinate_implies_limit_points_in_E_A}. If $|\kappa(\xi^n)|=1$, then $\sigma(\xi^n)\in E_A$, i.e.\ $\sigma(\xi^n)=\xi^0$. For the remaining indices, $\kappa(\sigma(\xi^n))_0=\kappa(\xi^n)_1\to\infty$ because $A$ is column-finite, so every limit point of $(\sigma(\xi^n))_n$ lies in $E_A=\{\xi^0\}$ by Lemma \ref{lemma:exploding_first_coordinate_implies_limit_points_in_E_A}; by compactness, $\sigma(\xi^n)\to\xi^0$.
\end{proof}

\begin{remark} In \cite{MichiNakaHisaYoshi2022} there are conditions for the continuous extension of $\sigma$ to the whole of $X_A$. We compare those conditions with Definition \ref{def:Single_Empty_Word_Class}. First observe that, although they have as standing hypothesis that the shift is topologically mixing (or equivalently, aperiodic transition matrices, see Observation 7.2.2 in Kitchens' book \cite{Kitchens1998}), such a condition is only used for results that come after the extension of the shift map. So let us assume only that $A$ is irreducible in their work. Another standing hypothesis is that $\mathcal{O}_A$ be unital, which is equivalent to $(ii)$ in our definition. The column-finiteness property in $(i)$ is weaker than uniform column-finiteness (UCF) (\cite[Definition 3.1]{MichiNakaHisaYoshi2022}), and $(iii)$ is equivalent to the finitely summable property (FS) (\cite[Definition 3.3]{MichiNakaHisaYoshi2022}). Both extensions coincide, making the empty word a fixed point.
\end{remark}

\begin{figure}[h]
    \centering
    \scalebox{.7}{

\tikzset{every picture/.style={line width=0.75pt}} 

\begin{tikzpicture}[x=0.75pt,y=0.75pt,yscale=-1,xscale=1]

\draw  [fill={rgb, 255:red, 0; green, 0; blue, 0 }  ,fill opacity=1 ] (94.94,100.01) .. controls (94.88,94.07) and (99.65,89.21) .. (105.59,89.16) .. controls (111.53,89.1) and (116.38,93.87) .. (116.44,99.81) .. controls (116.49,105.75) and (111.72,110.6) .. (105.79,110.66) .. controls (99.85,110.71) and (94.99,105.94) .. (94.94,100.01) -- cycle ;
\draw  [fill={rgb, 255:red, 0; green, 0; blue, 0 }  ,fill opacity=1 ] (192.94,36.01) .. controls (192.88,30.07) and (197.65,25.21) .. (203.59,25.16) .. controls (209.53,25.1) and (214.38,29.87) .. (214.44,35.81) .. controls (214.49,41.75) and (209.72,46.6) .. (203.79,46.66) .. controls (197.85,46.71) and (192.99,41.94) .. (192.94,36.01) -- cycle ;
\draw [line width=2.25]    (194,30) .. controls (136.4,9.84) and (110.13,47.74) .. (105.99,84.57) ;
\draw [shift={(105.59,89.16)}, rotate = 273.61] [fill={rgb, 255:red, 0; green, 0; blue, 0 }  ][line width=0.08]  [draw opacity=0] (16.07,-7.72) -- (0,0) -- (16.07,7.72) -- (10.67,0) -- cycle    ;
\draw [line width=2.25]    (116.44,99.81) .. controls (194.41,100) and (114.33,191.24) .. (106.19,115.51) ;
\draw [shift={(105.79,110.66)}, rotate = 86.46] [fill={rgb, 255:red, 0; green, 0; blue, 0 }  ][line width=0.08]  [draw opacity=0] (16.07,-7.72) -- (0,0) -- (16.07,7.72) -- (10.67,0) -- cycle    ;

\draw (0,85) node [anchor=north west][inner sep=0.75pt]  [font=\huge]  {$( e,\{1\})$};
\draw (160,56) node [anchor=north west][inner sep=0.75pt]  [font=\huge]  {$( 1,\{1\})$};

\end{tikzpicture}
}
    \caption{The continuously extended dynamics on $X_A$ for the renewal shift on the empty word element.}
    \label{fig:dynamics_empty_renewal}
\end{figure}

We now present an example of a shift in the single empty-word class that does not satisfy the extension conditions of \cite{MichiNakaHisaYoshi2022}.

\begin{example}[Upper triangle shift]\label{ex:upper_triangle_shift} Let $A$ be the transition matrix given by $A(n+1,n) = A(n,n+i) = 1$, $n \in \mathbb{N}$ and $i\in\mathbb{N}_0$, and $A(i,j) = 0$ otherwise. The matrix is irreducible and column-finite, and $X_A$ is compact. Moreover, $E_A$ is a singleton, whose empty-stem configuration is $(e,\mathbb{N})$. However, $|\{i \in \mathbb{N}: A(i,j)=1\}| = j + 1$ for every $j \in \mathbb{N}$, and therefore $A$ is not uniformly column-finite, but $X_A$ belongs to the single empty-word class, and therefore $\sigma$ has a continuous extension to $X_A$.
\end{example}

The next example shows the importance of the column-finiteness hypothesis.

\begin{example}[Full shift] The transition matrix $A$ for the full shift is given by $A(i,j) = 1$, for every $i,j \in \mathbb{N}$. Observe that $X_A$ is compact because $1 = Q_1 \in \mathcal{D}_A$. Also, $E_A$ is a singleton, where the empty word is $\xi^0 = (e, \mathbb{N})$. However, $A$ is not column-finite. We show that the map $\sigma$ cannot be extended continuously to $X_A$. In fact, consider the sequences $\xi^n$ and $\eta^n$ in $\Sigma_A$, such that
\begin{equation*}
    \kappa(\xi^n) = n1^\infty \; \text{ and } \kappa(\eta^n) = n2^\infty.
\end{equation*}
By Lemma \ref{lemma:exploding_first_coordinate_implies_limit_points_in_E_A} and $|E_A| = 1$, we have that $\xi^n, \eta^n \to \xi^0$, but $\sigma(\xi^n) = 1^\infty$ and $\sigma(\eta^n) = 2^\infty$ for every $n \in \mathbb{N}$. Therefore, every extension of $\sigma$ to $X_A$ in the full shift is necessarily discontinuous. In \cite{OttTomfordeWillis2014}, the authors construct the full shift including finite words via a compactification of the alphabet. In this case, the full shift is homeomorphic to $X_A$, and the extension of the shift map to the whole space is discontinuous; it makes the empty word a fixed point.
\end{example}

Now we present an example that illustrates the necessity of the compactness hypothesis.

\begin{example} Let $A$ be the transition matrix defined by $A(1, 2n) = A(n + 1, n) = 1$ for all $n\in\mathbb{N}$, and zero otherwise. Since in this example there exists a sequence of columns of $A$ converging to the null column, Lemma 27 in \cite{BEFR2018} implies that $X_A$ is not compact. There exists only one empty word, namely $(e,\{1\})$. Consider the sequence $(\xi^n)_n$ such that
\begin{equation*}
    \kappa(\xi^n) = 2n(2n-1)\cdots 21(21)^\infty.
\end{equation*}
We find that $\xi^n_0=2n$ is even and $C^{2n}\to\{1\}$, so $\xi^n\to(e,\{1\})$; while $\sigma(\xi^n)_0=2n-1$ is odd and $C^{2n-1}\to 0$, so $(\sigma(\xi^n))_n$ has no limit point in $X_A$ (it would converge to $(e,\varnothing)$ in $\widetilde{X}_A$). Thus, it is impossible to define a continuous extension of $\sigma$ to $X_A$.
\end{example}

\begin{remark}\label{rmk: If A is renewal, sigma(E_A)=E_A}
     Let $A$ be an irreducible transition matrix for which the Exel--Laca algebra $\mathcal{O}_A$ is unital and $\sigma$ is continuously extended to $X_A$. If $A(n+1,n)=1$, for every $n\in\mathbb{N}$, then $\sigma(E_A)=E_A$. In fact, the inclusion $\sigma(E_A)\subseteq E_A$ is Proposition \ref{thm:E_A_shift_invariant}. For the reverse inclusion, let $(e,R_C)\in E_A$ and let $(i_n)_n$ be a sequence with $i_n\to\infty$ and $C^{i_n}\to C$. By compactness of $\{0,1\}^{\mathbb{N}}$, passing to a subsequence we may assume that $(C^{i_n+1})_n$ converges to some $C'$, and $C'\neq 0$ because $X_A$ is compact, so $(e,R_{C'})\in E_A$. Since $A(i_n+1,i_n)=1$ and $i_n+1\to\infty$, condition (5) of Theorem \ref{Theorem: alpha well-defined iff sigma extended} (note that $A$ is column-finite by Theorem \ref{Theorem: If sigma is extended continuously, then A is column finite}) gives $\sigma(e,R_{C'})=(e,R_C)$.
\end{remark}

\begin{definition}\label{def:Periodic_Renewal_Class} We say that a generalized countable Markov shift $X_A$ belongs to the \emph{periodic renewal class} if it satisfies the following properties:
\begin{itemize}
    \item[$(i)$] $A$ is irreducible;
    \item[$(ii)$] $X_A$ is compact;
    \item[$(iii)$] $A(i+1,i) = 1$ for every $i \in \mathbb{N}$;
    \item[$(iv)$] the set of infinite emitters is finite, where we define
    \begin{equation*}
        n_e := \max\{j \in \mathbb{N}: j \text{ is an infinite emitter}\};
    \end{equation*}
    \item[$(v)$] there exists an $n_e \times N$ $\{0,1\}$-matrix $M$ satisfying:
    \begin{itemize}
        \item[$(a)$] the columns of $M$ are pairwise distinct;
        \item[$(b)$] from the $I$-th column on, $I \in \mathbb{N}$, the first $n_e$ rows of $A$ consist of an infinite repetition of the matrix $M$. In other words, there exists $I \in \mathbb{N}$ such that, for every  $0 \leq p \leq N-1$, $s \in \mathbb{N}_0$, it follows that
        \begin{equation*}
            A(k,I + p + sN) = M(k, p  + 1),\; 1 \leq k \leq n_e.
        \end{equation*}
    \end{itemize}
    \item[$(vi)$] For every $i>n_e$, $A(i,k) = 1$ if and only if $k = i-1$.
\end{itemize}
\end{definition}

\begin{definition}\label{def:Permutation_Renewal_Class} Let $N\in\mathbb{N}$, and $\tau:\{1,2,\cdots, N\}\to\{1,2,\cdots, N\}$ be a permutation. Also, let $m\in\mathbb{N}$ be the smallest number such that $\tau^m$ is the identity function. We say that a generalized countable Markov shift $X_A$ belongs to the \emph{permutation renewal class} if it satisfies the following properties:
\begin{itemize}
    \item[$(i)$] $A$ is irreducible;
    \item[$(ii)$] $X_A$ is compact;
    \item[$(iii)$] $A(i+N,i) = 1$ for every $i \in \mathbb{N}$;
    \item[$(iv)$] the set of infinite emitters is finite, where we define
    \begin{equation*}
        n_e := \max\{j \in \mathbb{N}: j \text{ is an infinite emitter}\};
    \end{equation*}
    \item[$(v)$] there exists an $n_e \times N$ $\{0,1\}$-matrix $M$ satisfying:
    \begin{itemize}
        \item[$(a)$] the columns of $M$ are pairwise distinct;
        \item[$(b)$] Let $\overline{M}$
 be the $n_e \times mN$ $\{0,1\}$-matrix obtained by the horizontal concatenation of the matrices:
        $$\overline{M}=[\tau^{m-1}(M)\quad\cdots\quad\tau(M)\quad M],$$
        where $\tau^l(M)$ is the $n_e \times N$ $\{0,1\}$-matrix whose $k$-th column coincides with the $\tau^l(k)$-th column of $M$, $k\in\{1,\dots,N\}$.
        From an $I$-th column, $I \in \mathbb{N}$, the first $n_e$ rows of $A$ consist of an infinite repetition of the matrix $\overline{M}$. In other words, there exists $I \in \mathbb{N}$ such that, for every $0 \leq p \leq mN-1$ and $s \in \mathbb{N}_0$, it follows that
        \begin{equation*}
            A(k,I + p + smN) = \overline{M}(k, p+1),\; 1 \leq k \leq n_e;
        \end{equation*}
    \end{itemize}
    \item[$(vi)$] we impose the conditions:
        \begin{itemize}
            \item[$(a)$] if $N \geq n_e$, then $A(i+N,j) = 0$ for every $j \neq i$. In addition, $A(i,j) = 0$ on $\{(i,j) \in \mathbb{N} \times \mathbb{N}: n_e + 1 \leq i < N + 1, j \geq I\}$;
            \item[$(b)$] otherwise, when $N < n_e$, if $i + N \geq n_e + 1$, then $A(i+N,j) = 0$ for every $j \neq i$.
        \end{itemize}
    \end{itemize}
\end{definition}

The general form of the transition matrices of generalized countable Markov shifts in the periodic renewal class is shown in Figure \ref{fig:Periodic_Renewal_Matrix_Class}. Examples of generalized countable Markov shifts in this class are the renewal (Example \ref{exa:renewal_shift}) and pair renewal (Example \ref{exa:pair_renewal_shift}).

\begin{figure}[h]
    \centering
    \scalebox{.7}{

\tikzset{every picture/.style={line width=0.75pt}} 

\begin{tikzpicture}[x=0.75pt,y=0.75pt,yscale=-1,xscale=1]

\draw [line width=1.5]    (258.65,51.4) -- (258.95,74) ;
\draw [shift={(259,78)}, rotate = 269.25] [fill={rgb, 255:red, 0; green, 0; blue, 0 }  ][line width=0.08]  [draw opacity=0] (13.4,-6.43) -- (0,0) -- (13.4,6.44) -- (8.9,0) -- cycle    ;

\draw   (252,85) -- (356,85) -- (356,196) -- (252,196) -- cycle ;

\draw  [color={rgb, 255:red, 161; green, 163; blue, 167 }  ,draw opacity=1 ][fill={rgb, 255:red, 161; green, 163; blue, 167 }  ,fill opacity=1 ] (100,86) -- (250,86) -- (250,195) -- (100,195) -- cycle ;
\draw   (356,85) -- (460,85) -- (460,196) -- (356,196) -- cycle ;

\draw   (460,85) -- (564,85) -- (564,196) -- (460,196) -- cycle ;

\draw  [line width=1.5]  (150,203) -- (352,418) -- (100,418) -- (100,202) -- cycle ;
\draw  [line width=1.5]  (564,205) -- (563,423) -- (414,420) -- (211,205) -- cycle ;
\draw [line width=1.5]    (362.65,51.4) -- (362.95,74) ;
\draw [shift={(363,78)}, rotate = 269.25] [fill={rgb, 255:red, 0; green, 0; blue, 0 }  ][line width=0.08]  [draw opacity=0] (13.4,-6.43) -- (0,0) -- (13.4,6.44) -- (8.9,0) -- cycle    ;

\draw [line width=1.5]    (468.65,53.4) -- (468.95,76) ;
\draw [shift={(469,80)}, rotate = 269.25] [fill={rgb, 255:red, 0; green, 0; blue, 0 }  ][line width=0.08]  [draw opacity=0] (13.4,-6.43) -- (0,0) -- (13.4,6.44) -- (8.9,0) -- cycle    ;

\draw (21,79.4) node [anchor=north west][inner sep=0.75pt]  [font=\LARGE]  {$A\ =\ \begin{pmatrix}
  &  &  &  &  & \quad &  &  &  &  &  &  &  &  &  &  & \,     \hspace{5cm}  \\
1 &  &  &  &  &  &  &  &  &  &  &  &  &       \\
  & 1 &  &  &  &  &  &  &         \\
  &  & 1 &  &  &  &  &  &       \\
  &  &  & 1 &  &  &  &  &       \\
  &  &  &  & 1 &  &  &  &      \\
  &  &  &  &  & 1 &  &  &     \\
  &  &  &  &  &  & 1 &  &      \\ 
  &  &  &  &  &  &  & 1 & \\
  &  &  &  &  &  &  &  &  \\
  &  &  &  &  &  &  &  & \\
  &  &  &  &  &  &  &  & \\
  &  &  &  &  &  &  &  & \\
  &  &  &  &  &  &  &  & 
\end{pmatrix}$};
\draw (136,47.4) node [anchor=north west][inner sep=0.75pt]  [font=\LARGE,color={rgb, 255:red, 0; green, 0; blue, 0 }  ,opacity=1 ]  {$j< I$};
\draw (283,115.4) node [anchor=north west][inner sep=0.75pt]  [font=\Huge]  {$M$};
\draw (387,115.4) node [anchor=north west][inner sep=0.75pt]  [font=\Huge]  {$M$};
\draw (491,115.4) node [anchor=north west][inner sep=0.75pt]  [font=\Huge]  {$M$};
\draw (568,116.4) node [anchor=north west][inner sep=0.75pt]  [font=\huge]  {$\cdots $};
\draw (300,320) node [anchor=north west][inner sep=0.75pt]  [font=\Huge,rotate=-385.05]  {$\ddots $};
\draw (568,256.4) node [anchor=north west][inner sep=0.75pt]  [font=\huge]  {$\cdots $};
\draw (205,426.4) node [anchor=north west][inner sep=0.75pt]  [font=\Huge]  {$\vdots $};
\draw (464,428.4) node [anchor=north west][inner sep=0.75pt]  [font=\Huge]  {$\vdots $};
\draw (565,426.4) node [anchor=north west][inner sep=0.75pt]  [font=\Huge,rotate=-385.05]  {$\ddots $};
\draw (233,23.4) node [anchor=north west][inner sep=0.75pt]  [font=\Large]  {$j=I$};
\draw (337,23.4) node [anchor=north west][inner sep=0.75pt]  [font=\Large]  {$j=I+m$};
\draw (154,303.4) node [anchor=north west][inner sep=0.75pt]  [font=\Huge]  {$0$};
\draw (442,269.4) node [anchor=north west][inner sep=0.75pt]  [font=\Huge]  {$0$};
\draw (443,25.4) node [anchor=north west][inner sep=0.75pt]  [font=\Large]  {$j=I+2m$};

\end{tikzpicture}

}
    \caption{Typical matrix of a generalized countable Markov shift that belongs to the periodic renewal class.\label{fig:Periodic_Renewal_Matrix_Class}}
\end{figure}

\begin{proposition}\label{prop:PRC_properties} For matrices $A$ and $M$ as in Definition \ref{def:Periodic_Renewal_Class} or in Definition \ref{def:Permutation_Renewal_Class}, we always have the following:
\begin{enumerate}
    \item $M$ has no zero columns;
    \item $|E_A| = N \leq 2^{n_e}-1 <\infty$;
    \item $A$ is column-finite.
\end{enumerate}
\end{proposition}
\begin{proof} In fact, the compactness of $X_A$ implies that $\mathcal{D}_A$ is unital, and so is $\mathcal{O}_A$. By Proposition 8.5 of \cite{EL1999}, the zero column cannot be a limit point of the sequence of columns of $A$ (product topology in $\{0,1\}^\mathbb{N}$), and since $\lim_{k \to \infty}(A(n,k))_{k \in \mathbb{N}} = 0$ for every $n > n_e$, we necessarily have that $(1)$ holds. On the other hand, Proposition 46 of \cite{BEFR2018} gives a bijection between the elements of $E_A$ and the non-zero limit points of the columns of $A$, and since the columns of $M$ are pairwise distinct, by Definition \ref{def:Periodic_Renewal_Class} and Definition \ref{def:Permutation_Renewal_Class} there are exactly $N$ limit points of $(A(\cdot,k))_{k \in \mathbb{N}}$, namely the elements of $\{0,1\}^\mathbb{N}$ consisting of a column of $M$ in the rows $1 \leq k \leq n_e$ and zero in the remaining entries. We conclude that $|E_A| = N$. The bound $N\leq 2^{n_e}-1$ follows directly from (1) and $(v)(a)$: the $N$ columns of $M$ are pairwise distinct non-zero vectors of $\{0,1\}^{n_e}$. Then $(2)$ holds. Assertion $(3)$ is straightforward; observe that, for $j\geq I$, the $j$-th column of $A$ has at most $n_e$ non-zero entries among the first $n_e$ rows, at most finitely many among the rows $n_e<i\leq N$, and exactly one among the remaining rows, by $(vi)$.    
\end{proof}

\begin{remark}
    Our definitions of periodic and permutation renewal classes differ from other definitions of classes of shifts containing the renewal shift, as found in the literature on countable Markov shifts; see \cite{Iommi2006}. 
\end{remark}

For the next theorem, let $X_A$ be in the periodic or permutation renewal class, with the matrix $M$ as in Definition \ref{def:Periodic_Renewal_Class} or in Definition \ref{def:Permutation_Renewal_Class}, and let $R_k$ denote the set of indices corresponding to the entries equal to $1$ in the $k$-th column of $M$, $1 \leq k \leq N$. We commit an abuse of notation by also writing $R_k$ for the limit point of the columns of $A$ that coincides with the $k$-th column of $M$ in the entries from $1$ to $n_e$ and is zero in the remaining entries.

\begin{theorem}\label{thm:PRC_continuous_extension_dynamics} For a generalized countable Markov shift $X_A$ that belongs to the periodic or permutation renewal class, the shift map $\sigma$ admits a continuous extension to $X_A$. Moreover,
\begin{enumerate}
    \item if $X_A$ is in the periodic renewal class, then
        \begin{equation}\label{eq:PeriodicRC_continuous_extension_dynamics}
            \sigma(e,R_{1}) = (e,R_{N}) \text{ and } \sigma(e,R_{p}) = (e,R_{p-1}) \quad 1 < p \leq N;
        \end{equation}
     \item if $X_A$ is in the permutation renewal class, then
        \begin{equation}\label{eq:PermutationRC_continuous_extension_dynamics}
            \sigma(e,R_{p}) = (e,R_{\tau(p)}), \quad 1 \leq p \leq N.
        \end{equation}
\end{enumerate}
\end{theorem}

\begin{proof} Suppose first that $X_A$ belongs to the permutation renewal class. We verify condition (5) of Theorem \ref{Theorem: alpha well-defined iff sigma extended}, whose remaining hypotheses are satisfied by Proposition \ref{prop:PRC_properties}. Let $(i_n)_n$, $(j_n)_n$ be sequences with $i_n\to\infty$, $A(i_n,j_n)=1$ and $C^{i_n}\to C$. By Proposition \ref{prop:PRC_properties}, $R_C=R_p$ for some $1\leq p\leq N$, and, since the columns of $M$ are pairwise distinct, for $n$ large enough the first $n_e$ entries of $C^{i_n}$ coincide with the $p$-th column of $M$. Since $i_n\to\infty$, by $(vi)$ we have $j_n=i_n-N$ for $n$ large enough, and then, by $(v)$, the first $n_e$ entries of $C^{j_n}$ coincide with the $\tau(p)$-th column of $M$ (indeed, if column $i_n$ lies in the block $\tau^{l}(M)$ of $\overline{M}$, then column $i_n-N$ lies in the block $\tau^{l+1}(M)$, with $\tau^{m}=\mathrm{id}$). Hence $C^{j_n}\to R_{\tau(p)}$, which proves (5) and \eqref{eq:PermutationRC_continuous_extension_dynamics}. If $X_A$ belongs to the periodic renewal class, the proof is analogous to the one given above, by writing $j_n = i_n - 1$ instead of $j_n = i_n - N$, for large enough $n$, and noting that in this case we get $C^{j_n}\to R_{q}$, where $q = p -1$ for $1 < p \leq N$, and $q = N$ if $p = 1$. To illustrate the proof of \eqref{eq:PeriodicRC_continuous_extension_dynamics}, see Figure \ref{fig:Periodic_Renewal_sketch_of_proof}.

\end{proof}

\begin{figure}[h]
    \centering
    \input{figures/Perioic_Renewal_sketch_of_proof}
    \caption{Example of converging sequence of columns of a transition matrix $A$ belonging to the periodic renewal class.\label{fig:Periodic_Renewal_sketch_of_proof}}
\end{figure}

\begin{example}\label{ex:permutation_renewal_3_empty_words} Consider the permutation $\tau:\{1,2,3\}\to\{1,2,3\}$, defined by $\tau(1)=2$, $\tau(2)=1$ and $\tau(3)=3$. Let the Markov shift be given by the transition matrix $A$, defined by $A(1,n)=A(2,3n-2)=A(2,3n-1)=A(3,6n-4)=A(3,6n-2)=1$, together with the renewal condition $A(n+3,n)=1$, for every $n\in\mathbb{N}$. The matrices from Definition \ref{def:Permutation_Renewal_Class} are given by
\begin{equation*}
 M = \begin{pmatrix}
            1 & 1 & 1 \\
            1 & 1 & 0 \\
            1 & 0 & 0 \\
        \end{pmatrix}\quad\text{ and }\quad
    \overline{M} = \begin{pmatrix}
            1 & 1 & 1 & 1 & 1 & 1 \\
            1 & 1 & 0 & 1 & 1 & 0 \\
            0 & 1 & 0 & 1 & 0 & 0 \\
        \end{pmatrix}.
\end{equation*}
The set $E_A$ contains exactly three empty words, with roots $R_1=\{1,2,3\}$, $R_2=\{1,2\}$ and $R_3=\{1\}$, and the extension of $\sigma$ on $E_A$ acts by
$$\sigma(e,R_p)=(e,R_{\tau(p)}), \quad p\in\{1,2,3\}.$$ An illustration of the dynamics on $E_A$ can be found in Figure \ref{fig:permutation_renewal_3_empty_words}.
\end{example}

\begin{figure}[h!]
    \centering
    \scalebox{.82}{

\tikzset{every picture/.style={line width=0.75pt}} 

\begin{tikzpicture}[x=0.75pt,y=0.75pt,yscale=-1,xscale=1]

\draw  [fill={rgb, 255:red, 0; green, 0; blue, 0 }  ,fill opacity=1 ] (77.94,162.01) .. controls (77.88,156.07) and (82.65,151.21) .. (88.59,151.16) .. controls (94.53,151.1) and (99.38,155.87) .. (99.44,161.81) .. controls (99.49,167.75) and (94.72,172.6) .. (88.79,172.66) .. controls (82.85,172.71) and (77.99,167.94) .. (77.94,162.01) -- cycle ;
\draw  [fill={rgb, 255:red, 0; green, 0; blue, 0 }  ,fill opacity=1 ] (175.94,98.01) .. controls (175.88,92.07) and (180.65,87.21) .. (186.59,87.16) .. controls (192.53,87.1) and (197.38,91.87) .. (197.44,97.81) .. controls (197.49,103.75) and (192.72,108.6) .. (186.79,108.66) .. controls (180.85,108.71) and (175.99,103.94) .. (175.94,98.01) -- cycle ;
\draw [line width=2.25]    (177,92) .. controls (119.4,71.84) and (93.13,109.74) .. (88.99,146.57) ;
\draw [shift={(88.59,151.16)}, rotate = 273.61] [fill={rgb, 255:red, 0; green, 0; blue, 0 }  ][line width=0.08]  [draw opacity=0] (16.07,-7.72) -- (0,0) -- (16.07,7.72) -- (10.67,0) -- cycle    ;
\draw [line width=2.25]    (99.44,161.81) .. controls (177.41,162) and (97.33,253.24) .. (89.19,177.51) ;
\draw [shift={(88.79,172.66)}, rotate = 86.46] [fill={rgb, 255:red, 0; green, 0; blue, 0 }  ][line width=0.08]  [draw opacity=0] (16.07,-7.72) -- (0,0) -- (16.07,7.72) -- (10.67,0) -- cycle    ;
\draw  [fill={rgb, 255:red, 0; green, 0; blue, 0 }  ,fill opacity=1 ] (385.94,158.01) .. controls (385.88,152.07) and (390.65,147.21) .. (396.59,147.16) .. controls (402.53,147.1) and (407.38,151.87) .. (407.44,157.81) .. controls (407.49,163.75) and (402.72,168.6) .. (396.79,168.66) .. controls (390.85,168.71) and (385.99,163.94) .. (385.94,158.01) -- cycle ;
\draw  [fill={rgb, 255:red, 0; green, 0; blue, 0 }  ,fill opacity=1 ] (536.94,158.01) .. controls (536.88,152.07) and (541.65,147.21) .. (547.59,147.16) .. controls (553.53,147.1) and (558.38,151.87) .. (558.44,157.81) .. controls (558.49,163.75) and (553.72,168.6) .. (547.79,168.66) .. controls (541.85,168.71) and (536.99,163.94) .. (536.94,158.01) -- cycle ;
\draw [line width=2.25]    (541.33,149) .. controls (499.84,121.98) and (445.31,127.55) .. (409.19,147.75) ;
\draw [shift={(405.33,150)}, rotate = 328.57] [fill={rgb, 255:red, 0; green, 0; blue, 0 }  ][line width=0.08]  [draw opacity=0] (16.07,-7.72) -- (0,0) -- (16.07,7.72) -- (10.67,0) -- cycle    ;
\draw [line width=2.25]    (533.17,165.49) .. controls (504.64,173.53) and (452.46,184.31) .. (407.44,157.81) ;
\draw [shift={(538.33,164)}, rotate = 163.5] [fill={rgb, 255:red, 0; green, 0; blue, 0 }  ][line width=0.08]  [draw opacity=0] (16.07,-7.72) -- (0,0) -- (16.07,7.72) -- (10.67,0) -- cycle    ;
\draw  [fill={rgb, 255:red, 0; green, 0; blue, 0 }  ,fill opacity=1 ] (356.94,83.01) .. controls (356.88,77.07) and (361.65,72.21) .. (367.59,72.16) .. controls (373.53,72.1) and (378.38,76.87) .. (378.44,82.81) .. controls (378.49,88.75) and (373.72,93.6) .. (367.79,93.66) .. controls (361.85,93.71) and (356.99,88.94) .. (356.94,83.01) -- cycle ;
\draw  [fill={rgb, 255:red, 0; green, 0; blue, 0 }  ,fill opacity=1 ] (306.94,99.01) .. controls (306.88,93.07) and (311.65,88.21) .. (317.59,88.16) .. controls (323.53,88.1) and (328.38,92.87) .. (328.44,98.81) .. controls (328.49,104.75) and (323.72,109.6) .. (317.79,109.66) .. controls (311.85,109.71) and (306.99,104.94) .. (306.94,99.01) -- cycle ;
\draw [line width=2.25]    (380.72,157.67) .. controls (358.78,155.76) and (331.24,146.97) .. (317.79,109.66) ;
\draw [shift={(385.94,158.01)}, rotate = 182.55] [fill={rgb, 255:red, 0; green, 0; blue, 0 }  ][line width=0.08]  [draw opacity=0] (16.07,-7.72) -- (0,0) -- (16.07,7.72) -- (10.67,0) -- cycle    ;
\draw [line width=2.25]    (386.22,145.96) .. controls (376.84,134.36) and (367.28,126.27) .. (366.79,92.66) ;
\draw [shift={(389.33,150)}, rotate = 234.08] [fill={rgb, 255:red, 0; green, 0; blue, 0 }  ][line width=0.08]  [draw opacity=0] (16.07,-7.72) -- (0,0) -- (16.07,7.72) -- (10.67,0) -- cycle    ;
\draw  [fill={rgb, 255:red, 0; green, 0; blue, 0 }  ,fill opacity=1 ] (564.94,78.01) .. controls (564.88,72.07) and (569.65,67.21) .. (575.59,67.16) .. controls (581.53,67.1) and (586.38,71.87) .. (586.44,77.81) .. controls (586.49,83.75) and (581.72,88.6) .. (575.79,88.66) .. controls (569.85,88.71) and (564.99,83.94) .. (564.94,78.01) -- cycle ;
\draw  [fill={rgb, 255:red, 0; green, 0; blue, 0 }  ,fill opacity=1 ] (605.94,99.01) .. controls (605.88,93.07) and (610.65,88.21) .. (616.59,88.16) .. controls (622.53,88.1) and (627.38,92.87) .. (627.44,98.81) .. controls (627.49,104.75) and (622.72,109.6) .. (616.79,109.66) .. controls (610.85,109.71) and (605.99,104.94) .. (605.94,99.01) -- cycle ;
\draw  [fill={rgb, 255:red, 0; green, 0; blue, 0 }  ,fill opacity=1 ] (625.94,139.01) .. controls (625.88,133.07) and (630.65,128.21) .. (636.59,128.16) .. controls (642.53,128.1) and (647.38,132.87) .. (647.44,138.81) .. controls (647.49,144.75) and (642.72,149.6) .. (636.79,149.66) .. controls (630.85,149.71) and (625.99,144.94) .. (625.94,139.01) -- cycle ;
\draw [line width=2.25]    (573.79,88.66) .. controls (572.43,108.61) and (567.09,123.9) .. (555.17,143.91) ;
\draw [shift={(552.59,148.16)}, rotate = 301.81] [fill={rgb, 255:red, 0; green, 0; blue, 0 }  ][line width=0.08]  [draw opacity=0] (16.07,-7.72) -- (0,0) -- (16.07,7.72) -- (10.67,0) -- cycle    ;
\draw [line width=2.25]    (616.69,98.91) .. controls (608.18,119.7) and (585.07,139.68) .. (562.39,155.15) ;
\draw [shift={(558.44,157.81)}, rotate = 326.51] [fill={rgb, 255:red, 0; green, 0; blue, 0 }  ][line width=0.08]  [draw opacity=0] (16.07,-7.72) -- (0,0) -- (16.07,7.72) -- (10.67,0) -- cycle    ;
\draw [line width=2.25]    (636.69,138.91) .. controls (628,160.14) and (603.16,174) .. (561.89,164.16) ;
\draw [shift={(557.33,163)}, rotate = 15.26] [fill={rgb, 255:red, 0; green, 0; blue, 0 }  ][line width=0.08]  [draw opacity=0] (16.07,-7.72) -- (0,0) -- (16.07,7.72) -- (10.67,0) -- cycle    ;

\draw (12,149.4) node [anchor=north west][inner sep=0.75pt]  [font=\large]  {$( e,\{1\})$};
\draw (156,116.4) node [anchor=north west][inner sep=0.75pt]  [font=\large]  {$( 1,\{1\})$};
\draw (355,180.4) node [anchor=north west][inner sep=0.75pt]  [font=\large]  {$( e,\{1,2\})$};
\draw (503,180.4) node [anchor=north west][inner sep=0.75pt]  [font=\large]  {$( e,\{1,2,3\})$};
\draw (254,61.07) node [anchor=north west][inner sep=0.75pt]  [font=\large]  {$( 1,\{1,2\})$};
\draw (349,42.07) node [anchor=north west][inner sep=0.75pt]  [font=\large]  {$( 2,\{1,2\})$};
\draw (496,38.07) node [anchor=north west][inner sep=0.75pt]  [font=\large]  {$( 1,\{1,2,3\})$};
\draw (609,61.07) node [anchor=north west][inner sep=0.75pt]  [font=\large]  {$( 2,\{1,2,3\})$};
\draw (648,112.07) node [anchor=north west][inner sep=0.75pt]  [font=\large]  {$( 3,\{1,2,3\})$};

\end{tikzpicture}
}
    \caption{The dynamics on $E_A$ in Example \ref{ex:permutation_renewal_3_empty_words}. The elements presented above are written as an ordered pair stem-root.}
    \label{fig:permutation_renewal_3_empty_words}
\end{figure}

A generalization of the pair renewal shift, see Example \ref{exa:pair_renewal_shift}, for any finite quantity of empty words is presented next.

\begin{example}[$N$-renewal shift] Fix $N \in \mathbb{N}$, $N \geq 2$, and consider the transition matrix given by $A(1,n)=A(n+1,n)=A(k,nN+k)=1$, for every $1\leq k\leq N$ and $n\geq 1$, and zero otherwise. In this example, $X_A$ belongs to the periodic renewal class. It is immediate that $A$ is irreducible. One can check that $X_A$ is compact and that item $(iii)$ of Definition \ref{def:Periodic_Renewal_Class} is satisfied. In addition, there are exactly $N$ infinite emitters, namely the elements of $\{1,2,\dots,N\}$, and $N$ empty-word configurations, namely $(e,\{1,n\})$, $1 \leq n \leq N$. The matrix $M$ from Definition~\ref{def:Periodic_Renewal_Class} $(v)$ is given by
\begin{equation*}
    M  = \bordermatrix{  & 1 & 2 & 3 & \cdots & N \cr
              1 & 1 & 1 & 1 & \cdots & 1    \cr
              2 & 0 & 1 & 0 & \cdots & 0    \cr
              3 & 0 & 0 & 1 & \cdots & 0    \cr
              4 & 0 & 0 & 0 & \cdots & 0    \cr
              \vdots & \vdots & \vdots & \vdots & \ddots & \vdots  \cr
              N-1 & 0 & 0 & 0 & \cdots & 0  \cr
              N & 0 & 0 & 0 & \cdots & 1   
        }
\end{equation*}
for $I = N+1$. By Theorem \ref{thm:PRC_continuous_extension_dynamics}, the shift dynamics in the $N$-renewal shift admits a continuous extension, which is shown in Figure \ref{fig:dynamics_N_renewal}.

\begin{figure}[h!]
    \centering
    \input{figures/dynamics_N_renewal_shift}
    \caption{The continuously extended dynamics on $X_A$ for the $N$-renewal shift for the empty word elements. The elements presented above are written as an ordered pair stem-root.}
    \label{fig:dynamics_N_renewal}
\end{figure}
\end{example}

We observe that the same proof in Theorem \ref{thm:PRC_continuous_extension_dynamics} holds when we add a finite number of edges in the symbolic graph whose source symbols are strictly greater than $n_e$, i.e., if we include a finite number of 1's in the zero regions presented in Figure \ref{fig:Periodic_Renewal_Matrix_Class}. Example \ref{exa:too_much_ones_after_n_e_discontinuous} shows that adding infinitely many extra edges in the aforementioned region may render any extension of $\sigma$ discontinuous.

\begin{example}\label{exa:too_much_ones_after_n_e_discontinuous} Let the Markov shift space have transition matrix $A$ defined by $A(1, n) = A(2, 2n) = A(n + 1, n)=A(4n,2n) = 1$ for all $n\in\mathbb{N}$, and zero otherwise. The set of empty words is given by $E_A=\{(e,\{1\}), (e,\{1,2\})\}$. Consider the sequences $\xi^n$ and $\eta^n$ in $\Sigma_A$, such that
\begin{equation*}
    \kappa(\xi^n) = 2n(2n-1)\cdots 1^\infty \; \text{ and } \kappa(\eta^n) = 4n(2n)(2n-1)\cdots 1^\infty.
\end{equation*}
We have that $\xi^n,\eta^n\to(e,\{1,2\})$, but $\sigma(\xi^n) \to(e,\{1\})$ and $\sigma(\eta^n)=\xi^n\to(e,\{1,2\})$. Hence, it is not possible to obtain a continuous extension of $\sigma$ to $X_A$.
\end{example}

\begin{remark} Single empty-word, periodic renewal, and permutation renewal classes are not disjoint, nor are they proper subsets of each other. An example that belongs to all these classes is the renewal shift. Also, the lazy renewal shift belongs to the single empty-word class only, because its main diagonal is filled with 1's. On the other hand, the $N$-renewal shift, for $N\geq 2$, belongs exclusively to the periodic renewal class: it has $N \geq 2$ empty words, and so it is not in the single empty-word class; and it is not in the permutation renewal class because the matrix $M$ of Definition \ref{def:Permutation_Renewal_Class} would have to have $N$ columns, and then condition $(iii)$ of that definition would impose $A(i+N,i) = 1$ for every $i$, which is false. On the other hand, the matrix in Example \ref{ex:permutation_renewal_3_empty_words} belongs only to the permutation renewal class: it is not in the single empty-word class because it has three empty words, and it is not in the periodic renewal class because $A(4,3)=0$, so condition $(iii)$ of Definition \ref{def:Periodic_Renewal_Class} fails. This is illustrated in Figure \ref{fig:venn_renewal_classes}.
\end{remark}

\begin{figure}
    \centering
    \tikzset{every picture/.style={line width=0.75pt}} 

\begin{tikzpicture}[x=0.75pt,y=0.75pt,yscale=-1,xscale=1]

\draw  [color={rgb, 255:red, 0; green, 0; blue, 0 }  ,draw opacity=1 ][line width=1.5]  (38.33,129.67) .. controls (38.33,80.15) and (78.48,40) .. (128,40) .. controls (177.52,40) and (217.67,80.15) .. (217.67,129.67) .. controls (217.67,179.19) and (177.52,219.33) .. (128,219.33) .. controls (78.48,219.33) and (38.33,179.19) .. (38.33,129.67) -- cycle ;
\draw  [color={rgb, 255:red, 0; green, 0; blue, 0 }  ,draw opacity=1 ][line width=1.5]  (143.33,130.67) .. controls (143.33,81.15) and (183.48,41) .. (233,41) .. controls (282.52,41) and (322.67,81.15) .. (322.67,130.67) .. controls (322.67,180.19) and (282.52,220.33) .. (233,220.33) .. controls (183.48,220.33) and (143.33,180.19) .. (143.33,130.67) -- cycle ;
\draw  [color={rgb, 255:red, 0; green, 0; blue, 0 }  ,draw opacity=1 ][line width=1.5]  (91.33,217.67) .. controls (91.33,168.15) and (131.48,128) .. (181,128) .. controls (230.52,128) and (270.67,168.15) .. (270.67,217.67) .. controls (270.67,267.19) and (230.52,307.33) .. (181,307.33) .. controls (131.48,307.33) and (91.33,267.19) .. (91.33,217.67) -- cycle ;

\draw (11,14) node [anchor=north west][inner sep=0.75pt]  [font=\footnotesize,color={rgb, 255:red, 0; green, 0; blue, 0 }  ,opacity=1 ] [align=left] {{\large Single-empty word}};
\draw (208,15) node [anchor=north west][inner sep=0.75pt]  [font=\large,color={rgb, 255:red, 0; green, 0; blue, 0 }  ,opacity=1 ] [align=left] {Periodic renewal};
\draw (108,315) node [anchor=north west][inner sep=0.75pt]  [font=\large,color={rgb, 255:red, 0; green, 0; blue, 0 }  ,opacity=1 ] [align=left] {Permutation renewal};
\draw (46,112.33) node [anchor=north west][inner sep=0.75pt]  [font=\normalsize,color={rgb, 255:red, 0; green, 0; blue, 0 }  ,opacity=1 ] [align=left] {Lazy renewal};
\draw (226,86.33) node [anchor=north west][inner sep=0.75pt]  [font=\normalsize,color={rgb, 255:red, 0; green, 0; blue, 0 }  ,opacity=1 ] [align=left] {$\displaystyle N$-renewal};
\draw (145,241.33) node [anchor=north west][inner sep=0.75pt]  [font=\normalsize,color={rgb, 255:red, 0; green, 0; blue, 0 }  ,opacity=1 ] [align=left] {Example \ref{ex:permutation_renewal_3_empty_words}};
\draw (152,144.33) node [anchor=north west][inner sep=0.75pt]  [font=\normalsize,color={rgb, 255:red, 0; green, 0; blue, 0 }  ,opacity=1 ] [align=left] {Renewal};
\draw (236,112.4) node [anchor=north west][inner sep=0.75pt]    {$N\ \geq 2$};

\end{tikzpicture}
    \caption{Relations between the single empty-word, periodic renewal, and permutation renewal classes.}
    \label{fig:venn_renewal_classes}
\end{figure}

\begin{example}\label{exa:restictions_from_widetilde_X_A_to_X_A}
    Consider the matrices $A_1$ and $A_2$ defined by $A_1(1,2n-1)=A_2(1,2n-1)=A_2(2,1)=1$, together with the renewal conditions $A_1(n+1,n)=1$ and $A_2(n+2,n)=1$, for every $n\in\mathbb{N}$.
    \[
A_1=\begin{pmatrix}
            1 & 0 & 1 & 0 & 1 & 0 & \cdots \\
            1 & 0 & 0 & 0 & 0 & 0 & \cdots  \\
            0 & 1 & 0 & 0 & 0 & 0 & \cdots \\
            0 & 0 & 1 & 0 & 0 & 0 & \cdots \\
            0 & 0 & 0 & 1 & 0 & 0 & \cdots \\
            0 & 0 & 0 & 0 & 1 & 0 & \cdots \\
            \vdots & \vdots & \vdots & \vdots & \vdots & \vdots & \ddots \\
        \end{pmatrix} \quad A_2=\begin{pmatrix}
            1 & 0 & 1 & 0 & 1 & 0 & \cdots \\
            1 & 0 & 0 & 0 & 0 & 0 & \cdots  \\
            1 & 0 & 0 & 0 & 0 & 0 & \cdots \\
            0 & 1 & 0 & 0 & 0 & 0 & \cdots \\
            0 & 0 & 1 & 0 & 0 & 0 & \cdots \\
            0 & 0 & 0 & 1 & 0 & 0 & \cdots \\
            \vdots & \vdots & \vdots & \vdots & \vdots & \vdots & \ddots \\
        \end{pmatrix}
\]
Since the zero column is an accumulation point of the columns of both matrices, the Exel--Laca algebras $\mathcal{O}_{A_1}$ and $\mathcal{O}_{A_2}$ are non-unital. In both examples, $\widetilde{X}_{A_i}$ has exactly two empty words, namely $(e,\{1\})$ and $(e,\varnothing)$, and the shift maps extend continuously to the whole of $\widetilde{X}_{A_i}$, $i \in \{1,2\}$ (see Remark \ref{Remark: alpha well-defined iff sigma extended, E_A is non-unital}), with the following dynamics:
$$\sigma_1(e,\{1\})=(e,\varnothing),\quad \sigma_1(e,\varnothing)=(e,\{1\})$$ and $$\sigma_2(e,\{1\})=(e,\{1\}),\quad \sigma_2(e,\varnothing)=(e,\varnothing),$$
where $\sigma_i$ denotes the shift map on $\widetilde{X}_{A_i}$. In the first example, it is impossible to restrict $\sigma_1$ to obtain an extension on $X_{A_1}$, since $\sigma_1(X_{A_1})\nsubseteq X_{A_1}$. In contrast, in the second example such a restriction is possible.
     \end{example}

To conclude this section, we present two examples for which $\sigma(E_A) \subsetneq E_A$, that is, for which the inclusion is strict. The first has finitely many empty words, while the second has countably infinitely many.

\begin{example}\label{Example: finitely many empty words and strict invariance on E_A}
    Let $P=\{p_0,p_1,\cdots\}$ be the set consisting of $0$, $1$ and the prime numbers, arranged in increasing order. Consider the generalized countable Markov shift with transition matrix $A$ defined by $A(1,n)=A(2^2,2^{p_{n+1}})=1$ and, if $2^{p_n}\leq i<2^{p_{n+1}}$, then $A(i,2^{p_{n-1}})=1$, for every $n\in\mathbb{N}$, and zero otherwise. The columns of $A$ have exactly two limit points, $C$ with $R_C = \{1\}$ (along the columns $j\notin\{2^{p_k}:k\in\mathbb{N}\}$, which contain only the entry of the first row) and $\overline{C}$ with $R_{\overline{C}} = \{1,2^2\}$ (along the columns $2^{p_k}$, $k\to\infty$); hence $E_A=\{(e,\{1\}),(e,\{1,2^2\})\}$. Note also that $A$ is irreducible and column-finite, and $X_A$ is compact (the first row is full). Now let $(i_n)_n,(j_n)_n$ be such that $i_n\to\infty$ and $A(i_n,j_n)=1$. For $i_n> 2^{p_2}=4$ there is a unique $k(n)$ with $2^{p_{k(n)}}\leq i_n<2^{p_{k(n)+1}}$, the only non-zero entry of row $i_n$ is at $j_n=2^{p_{k(n)-1}}$, and $k(n)\to\infty$. Therefore $C^{j_n}\to\overline{C}$, regardless of the limit of $(C^{i_n})_n$. By Theorem \ref{Theorem: alpha well-defined iff sigma extended}, $\sigma$ extends continuously and $\sigma(e,\{1\})=\sigma(e,\{1,2^2\})=(e,\{1,2^2\})$, so that $\sigma(E_A)\subsetneq E_A$.
\end{example}
\[
A=\begin{pNiceArray}{cccccccccccccccccccccccc}[first-row,first-col,margin]
& 2^0 & 2^1 & & 2^2 & & & & 2^3 & &  \\
2^0 & 1 & 1 & 1 & 1 & 1 & 1 & 1 & 1 & 1 & \cdots\\
2^1 & 1 & 0 & 0 & 0 & 0 & 0 & 0 & 0 & 0  & \cdots\\
 & 1 & 0 & 0 & 0 & 0 & 0 & 0 & 0 & 0  & \cdots\\
2^2 & 0 & 1 & 0 & 1 & 0 & 0 & 0 & 1 & 0 & \cdots\\
 & 0 & 1 & 0 & 0 & 0 & 0 & 0 & 0 & 0 &  \cdots\\ 
 & 0 & 1 & 0 & 0 & 0 & 0 & 0 & 0 & 0 &  \cdots\\ 
 & 0 & 1 & 0 & 0 & 0 & 0 & 0 & 0 & 0 & \cdots\\ 
2^3 & 0 & 0 & 0 & 1 & 0 & 0 & 0 & 0 & 0 & \cdots\\ 
& 0 & 0 & 0 & 1 & 0 & 0 & 0 & 0 & 0 &\cdots\\ 

& \vdots & \vdots & \vdots &  \vdots & \vdots & \vdots & \vdots & \vdots & \vdots &   \ddots \\
2^5 & 0 & 0 & 0 & 0 & 0 & 0 & 0 & 1 & 0 &\cdots\\
& 0 & 0 & 0 & 0 & 0 & 0 & 0 & 1 & 0 &\cdots\\ 

& \vdots & \vdots & \vdots &  \vdots & \vdots & \vdots & \vdots & \vdots & \vdots &   \ddots
\end{pNiceArray}
\]

\begin{figure}[h!]
    \centering
    \tikzset{every picture/.style={line width=0.75pt}} 

\begin{tikzpicture}[x=0.75pt,y=0.75pt,yscale=-1,xscale=1]

\draw  [fill={rgb, 255:red, 0; green, 0; blue, 0 }  ,fill opacity=1 ] (58.94,133.01) .. controls (58.88,127.07) and (63.65,122.21) .. (69.59,122.16) .. controls (75.53,122.1) and (80.38,126.87) .. (80.44,132.81) .. controls (80.49,138.75) and (75.72,143.6) .. (69.79,143.66) .. controls (63.85,143.71) and (58.99,138.94) .. (58.94,133.01) -- cycle ;
\draw  [fill={rgb, 255:red, 0; green, 0; blue, 0 }  ,fill opacity=1 ] (209.94,133.01) .. controls (209.88,127.07) and (214.65,122.21) .. (220.59,122.16) .. controls (226.53,122.1) and (231.38,126.87) .. (231.44,132.81) .. controls (231.49,138.75) and (226.72,143.6) .. (220.79,143.66) .. controls (214.85,143.71) and (209.99,138.94) .. (209.94,133.01) -- cycle ;
\draw [line width=2.25]    (205.06,139.11) .. controls (184.35,147.73) and (140.6,161.66) .. (80.44,134.81) ;
\draw [shift={(209.94,137.01)}, rotate = 156.18] [fill={rgb, 255:red, 0; green, 0; blue, 0 }  ][line width=0.08]  [draw opacity=0] (16.07,-7.72) -- (0,0) -- (16.07,7.72) -- (10.67,0) -- cycle    ;
\draw  [fill={rgb, 255:red, 0; green, 0; blue, 0 }  ,fill opacity=1 ] (29.94,58.01) .. controls (29.88,52.07) and (34.65,47.21) .. (40.59,47.16) .. controls (46.53,47.1) and (51.38,51.87) .. (51.44,57.81) .. controls (51.49,63.75) and (46.72,68.6) .. (40.79,68.66) .. controls (34.85,68.71) and (29.99,63.94) .. (29.94,58.01) -- cycle ;
\draw [line width=2.25]    (59.22,120.96) .. controls (49.84,109.36) and (40.28,101.27) .. (39.79,67.66) ;
\draw [shift={(62.33,125)}, rotate = 234.08] [fill={rgb, 255:red, 0; green, 0; blue, 0 }  ][line width=0.08]  [draw opacity=0] (16.07,-7.72) -- (0,0) -- (16.07,7.72) -- (10.67,0) -- cycle    ;
\draw  [fill={rgb, 255:red, 0; green, 0; blue, 0 }  ,fill opacity=1 ] (167.94,58.01) .. controls (167.88,52.07) and (172.65,47.21) .. (178.59,47.16) .. controls (184.53,47.1) and (189.38,51.87) .. (189.44,57.81) .. controls (189.49,63.75) and (184.72,68.6) .. (178.79,68.66) .. controls (172.85,68.71) and (167.99,63.94) .. (167.94,58.01) -- cycle ;
\draw  [fill={rgb, 255:red, 0; green, 0; blue, 0 }  ,fill opacity=1 ] (239.94,52.01) .. controls (239.88,46.07) and (244.65,41.21) .. (250.59,41.16) .. controls (256.53,41.1) and (261.38,45.87) .. (261.44,51.81) .. controls (261.49,57.75) and (256.72,62.6) .. (250.79,62.66) .. controls (244.85,62.71) and (239.99,57.94) .. (239.94,52.01) -- cycle ;
\draw [line width=2.25]    (178.79,68.66) .. controls (185.38,91.42) and (194.55,103.39) .. (210.42,120.38) ;
\draw [shift={(213.83,124)}, rotate = 226.55] [fill={rgb, 255:red, 0; green, 0; blue, 0 }  ][line width=0.08]  [draw opacity=0] (16.07,-7.72) -- (0,0) -- (16.07,7.72) -- (10.67,0) -- cycle    ;
\draw [line width=2.25]    (250.79,62.66) .. controls (248.04,85.3) and (237.47,100.72) .. (227.15,118.85) ;
\draw [shift={(224.83,123)}, rotate = 298.81] [fill={rgb, 255:red, 0; green, 0; blue, 0 }  ][line width=0.08]  [draw opacity=0] (16.07,-7.72) -- (0,0) -- (16.07,7.72) -- (10.67,0) -- cycle    ;
\draw [line width=2.25]    (229.83,128) .. controls (308.23,87.82) and (296.36,176.33) .. (234.2,142.24) ;
\draw [shift={(230.33,140)}, rotate = 31.41] [fill={rgb, 255:red, 0; green, 0; blue, 0 }  ][line width=0.08]  [draw opacity=0] (16.07,-7.72) -- (0,0) -- (16.07,7.72) -- (10.67,0) -- cycle    ;

\draw (37,155.4) node [anchor=north west][inner sep=0.75pt]  [font=\large]  {$( e,\{1\})$};
\draw (176,159.4) node [anchor=north west][inner sep=0.75pt]  [font=\large]  {$( e,\{1,4\})$};
\draw (12,17.07) node [anchor=north west][inner sep=0.75pt]  [font=\large]  {$( 1,\{1\})$};
\draw (131,19.07) node [anchor=north west][inner sep=0.75pt]  [font=\large]  {$( 1,\{1,4\})$};
\draw (225,14.07) node [anchor=north west][inner sep=0.75pt]  [font=\large]  {$( 4,\{1,4\})$};

\end{tikzpicture}
    \caption{The dynamics on $E_A$ from Example \ref{Example: finitely many empty words and strict invariance on E_A}.}
    \label{fig:invariant_proper_empty_words}
\end{figure}

\begin{example}\label{Example: infinitely many empty words and strict invariance on E_A}
    Using the same set $P$ as in the previous example, consider the matrix $A$ defined by the following conditions:
    \begin{itemize}
        \item $A(1,n)=1$ for every $n\in\mathbb{N}$;
        \item if $2^{p_n}\leq i<2^{p_{n+1}}$, then $A(i,2^{p_{n-1}})=1$, for every $n\in\mathbb{N}$;
        \item for each $n\in\mathbb{N}$, $A(2^{p_n},p_n^{p_m})=1$, for every $m\geq n>2$;
    \end{itemize}
    and $A(i,j)=0$ otherwise. Note that $A$ is irreducible and column-finite, and $X_A$ is compact, so Theorem \ref{Theorem: alpha well-defined iff sigma extended} applies. By a similar analysis as in the previous example, the set of empty words is given by $E_A=\{(e,\{1\})\}\cup\{(e,\{1,2^{p_{n}}\}):n>2 \}$. Also, every allowed successor $j_n$ of a symbol $i_n\to\infty$ is either $2^{p_{k-1}}$ (whose column converges to $C$, $R_C = \{1\}$) or $p_k^{p_m}$ with $k\to\infty$ (whose column has non-zero entries only in the rows $1$ and $2^{p_k}\to\infty$, hence also converges to $C$); so $C^{j_n}\to C$ in all cases. Therefore, its dynamics satisfies
    $$\sigma(e,\{1\})=\sigma(e,\{1,2^{p_{n}}\})=(e,\{1\}),\, \text{ for every }n>2,$$ and $\sigma(E_A) \subsetneq E_A$.
\end{example}

\section{When the extension of the shift map is a local homeomorphism}\label{sec:when_is_ homeo}

In this short section, we discuss the problem of when the shift map is a local homeomorphism. This property is particularly important for groupoid C$^*$-algebras and, in particular, for Exel--Laca algebras; see \cite{BEFR2018, Deaconu:95, EL1999, Renault1980, Renault1999}. Since the generalized shift map is a homeomorphism when restricted to each generalized cylinder, it follows that $\sigma: X_A \setminus E_A \to X_A$ is a local homeomorphism. However, as we show next, this property does not necessarily persist for the extension of $\sigma$ to the whole space $X_A$.

\begin{theorem}\label{thm:sufficient_matrix_condition_X_A_isnt_homeo_local}
    Suppose that $X_A$ is compact and $\sigma$ admits a continuous extension. Let $(i_n)_n,(j_n)_n$ and $(\ell_n)_n$ be sequences in $\mathbb{N}$ such that $i_n\neq j_n$, $i_n,j_n\to\infty$, and $A(i_n,\ell_n)=A(j_n,\ell_n)=1$, for every $n\in\mathbb{N}$. If $C^{i_n}$ and $C^{j_n}$ converge to the same $C\in\{0,1\}^{\mathbb{N}}$, then $\sigma$ is not a local homeomorphism.
\end{theorem}

\begin{proof}
    Suppose, by contradiction, $\sigma$ is a local homeomorphism. Then there exists $W$ an open set containing the empty word $(e, R_C)$ such that $\sigma\vert_W: W \to \sigma(W)$ is a homeomorphism.
    Let $(\omega^n)_n$ be a sequence in $\Sigma_A$ such that $\omega^n\in[\ell_n]$. Since $C^{i_n},C^{j_n}\to C$, the sequences $i_n\omega^n, j_n\omega^n$ both converge to $(e, R_C)$. Hence, there exists $N\in\mathbb{N}$ such that $i_n\omega^n, j_n\omega^n\in W$, for every $n\geq N$. On the other hand, $\sigma(i_n\omega^n)=\omega^n= \sigma(j_n\omega^n)$. Thus, $\sigma|_W$ is not injective, which contradicts the fact that $\sigma|_W$ is a homeomorphism.
\end{proof}

\begin{corollary}\label{cor:sufficient_condition_when_sigma_is_not_homeo} Suppose that $X_A$ is compact and $\sigma$ admits a continuous extension. If $\sigma$ is a local homeomorphism, then $R$ is finite for every $(e,R) \in E_A$.
\end{corollary}
\begin{proof}
    Suppose that there exists $(e,R) \in E_A$ such that $R$ is infinite. Write $R=\{k_m:m\in\mathbb{N}\}$, where $k_m\to\infty$. Let $(k_{m_n})_n$ be a subsequence such that $(C^{k_{m_n}})_n$ converges. Set $(i_n)_n=(k_{m_n})_n$ and $(j_n)_n=(k_{m_{n+1}})_n$. Note that for each $n\in\mathbb{N}$ there exists $\ell_n$ such that $A(i_n,\ell_n)=A(j_n,\ell_n)=1$. Indeed, since $R$ is the root of an empty word, with $R = R_C$ for some $C \in \{0,1\}^\mathbb{N}$, $C=\lim_t C^{a_t}$ for some sequence $a_t\to\infty$, so for $t$ large enough the column $C^{a_t}$ has the entries $i_n$ and $j_n$ equal to $1$, and we may take $\ell_n=a_t$. By Theorem \ref{thm:sufficient_matrix_condition_X_A_isnt_homeo_local}, $\sigma$ is not a local homeomorphism.
\end{proof}

\begin{remark} The same results of Theorem \ref{thm:sufficient_matrix_condition_X_A_isnt_homeo_local} and Corollary \ref{cor:sufficient_condition_when_sigma_is_not_homeo} hold for $\widetilde{X}_A$.    
\end{remark}

\begin{example}
    Consider the upper triangle shift (see Example \ref{ex:upper_triangle_shift}), which contains the unique empty word $(e,\mathbb{N})$. Since this word has an infinite root, it follows from Corollary \ref{cor:sufficient_condition_when_sigma_is_not_homeo} that $\sigma$ is not a local homeomorphism.
\end{example}

The following examples show that the existence of an empty word with an infinite root is not a necessary condition for the shift map to fail to be a local homeomorphism.

\begin{example}
Since the lazy renewal shift (see Example \ref{exa:lazy_renewal_shift}) belongs to the single empty-word class, the shift map extends continuously to the whole of $X_A$, and $X_A$ contains the unique empty word $(e,\{1\})$. By definition, $A(n+1,n)=A(n,n)=1$, and by Theorem \ref{thm:sufficient_matrix_condition_X_A_isnt_homeo_local} (with $i_n=n$, $j_n=n+1$, $\ell_n=n$ and $C=\{1\}$), $\sigma$ is not a local homeomorphism. By similar arguments, one can show that in Example \ref{Example: infinitely many empty words and strict invariance on E_A}, for which we have infinitely many empty words, the shift map is not a local homeomorphism. Indeed, by definition, we have $A(2^{p_{n+1}}+1,2^{p_{n}})=A(2^{p_{n+1}}+2,2^{p_{n}})=1$ and $C^{2^{p_{n+1}}+1}=C^{2^{p_{n+1}}+2}=C$ for every $n>1$, where $R_C=\{1\}$. Again, the matrix satisfies the conditions of Theorem \ref{thm:sufficient_matrix_condition_X_A_isnt_homeo_local}.
\end{example}

The upper triangle and lazy renewal shifts show that membership in the single empty-word class does not guarantee that the extended shift is a local homeomorphism. The next proposition shows that, for the periodic and permutation renewal classes, this is always the case. Since the renewal shift lies in all three classes, membership in the single empty-word class alone does not determine whether the extended shift map is a local homeomorphism.

\begin{proposition}\label{Propo: Permutation renewal shift is a class of local homeomorphisms}
    Let $X_A$ belong to the periodic or permutation renewal class. Then the shift map $\sigma$, continuously extended to the whole space $X_A$, is a local homeomorphism.
\end{proposition}
\begin{proof}
    Suppose that $X_A$ belongs to the permutation renewal class. It suffices to prove that $\sigma$ is a local homeomorphism at the points of $E_A$. Since $E_A$ is finite, for each $(e,R) \in E_A$, there exist $F,G \subset \mathbb{N}$, both finite, such that $W \cap E_A = \{(e,R)\}$, where
    \begin{equation*}
        W = \left(\bigcap_{i \in F} C_{i^{-1}} \right) \cap \left(\bigcap_{j \in G} C_{j^{-1}}^c \right).
    \end{equation*}
    Observe that $W$ is clopen and hence compact. Set $V= W\cap \bigcap_{i=1}^{\ell+N}C_i^c$, where $\ell=\max\{N,n_e\}$, $N$ and $n_e$ as in Definition \ref{def:Permutation_Renewal_Class}, and note that $V$ is open, compact and Hausdorff. Then, it is sufficient to prove that $\sigma\vert_V:V \to \sigma(V)$ is an injection. Let $\xi\in V\setminus\{(e,R)\}$. Then $\kappa(\xi)_0>\ell+N\geq n_e$, so $\kappa(\xi)_0$ is not an infinite emitter; since a finite stem ends with an infinite emitter, $|\kappa(\xi)|\geq 2$, and by $(vi)$ the only successor of $\kappa(\xi)_0$ is $\kappa(\xi)_1=\kappa(\xi)_0-N$. Hence $\sigma(\xi)\notin E_A$, which shows that $\sigma(\xi)\neq\sigma(e,R)\in E_A$, and $\kappa(\xi)_0=\kappa(\sigma(\xi))_0+N$ is determined by $\sigma(\xi)$; as $\xi$ is determined by $\sigma(\xi)$ and $\kappa(\xi)_0$, the restriction $\sigma|_V$ is injective. The proof is similar when $X_A$ belongs to the periodic renewal class.
\end{proof}
\section{Extendable shift maps with infinite (uncountable) empty words}\label{sec:Example_infinit_empty_words}

In this section, we present an example where the set of empty words $E_A$ is in bijection with the Cantor set $\{0,1\}^{\mathbb{N}}$ and is therefore uncountable. Nevertheless, the shift map still admits a continuous extension to the entire space $X_A$. In addition, the dynamics induced on $E_A$ is conjugate, via the homeomorphism $\Phi$ between $E_A$ and the Cantor set of Proposition \ref{pro:bijection_between_E_A_and_Cantor_set}, to a ``subtract one'' rule on $\{0,1\}^{\mathbb{N}}$. More precisely, for $x=(x_n)_n\in\{0,1\}^{\mathbb{N}}$, we have

$$\Phi\circ\sigma\vert_{E_A}\circ\Phi^{-1}(x)=\begin{cases}
(\underbrace{1,\ldots,1}_{k-1},0,x_{k+1},x_{k+2},\ldots),
& \text{if } x_1=\cdots=x_{k-1}=0,\ x_k=1,\\[2mm]
(1,1,1,\ldots),
& \text{if } x=(0,0,0,\ldots).
\end{cases}$$
Thus, under the identification of $\{0,1\}^{\mathbb{N}}$ with the $2$-adic integers, $x\mapsto\sum_{n\geq 1}x_n2^{n-1}$, this map corresponds precisely to subtraction by one, that is, to the inverse of the dyadic odometer.  
\[
\begin{tikzcd}[row sep=1.5cm, column sep=2cm]
E_A \arrow[r, "\sigma\vert_{E_A}"] 
 \arrow[d, swap, "\Phi"] & E_A  \arrow[d, "\Phi"] \\
\{0,1\}^{\mathbb{N}}  \arrow[r, swap, "\Phi\circ\sigma\vert_{E_A}\circ\Phi^{-1}"]
& \{0,1\}^{\mathbb{N}}
\end{tikzcd}
\]

The ``subtract one'' map is the inverse of the dyadic odometer \cite{Silva2007}. Odometers and their generalizations, such as the Vershik map \cite{Vershik1981, Vershik1992}, are objects of great interest in dynamical systems. For instance, Vershik maps act on the path space of a standard Bratteli diagram \cite{Bratteli1972}. In this framework, the incidence matrices and path orderings of the diagram encode topological and ergodic properties, since every Cantor minimal system can be realized as a standard Bratteli--Vershik system \cite{HerPutSka1992}. Furthermore, these adic transformations have been extended to Borel dynamics \cite{BDK2006} and non-compact settings via generalized Bratteli diagrams with infinite vertex sets \cite{BJKK2025, BezuglyiJorgensenKarpelSanadhya2026}.

Let $\underline{D}\in \{0,1\}^{\mathbb{N}\times \mathbb{N}}$ denote the matrix filled with $1$’s in the main subdiagonal, that is, $\underline{D}(n+1,n)=1$ for every $n\in\mathbb{N}$, and zero otherwise. For $L\in\mathbb{N}$, consider the horizontal strip matrix $S_L\in\{0,1\}^{\mathbb{N}\times \mathbb{N}}$ defined by  
$$
S_L(i,i+(q-1)L-1)=1, \quad \text{for every } q>1 \text{ and } L\leq i<2L,
$$  
and zero otherwise. Note that the matrix  
$
A_m=\underline{D}+\sum_{\ell=0}^m S_{2^{\ell}}
$  
belongs to the periodic renewal class. In the next lemma, $C^j_i=A_m(i,j)$.

\begin{lemma}\label{Lemma: properties of matrix A_N infinite E_A}
    Given $m>1$, $2^n\leq i< 2^{n+1}$ for some $1\leq n\leq m$, and $j\geq i-1$, the following statements hold:
    \begin{enumerate}
        \item $A_m(i,j)=1$ if and only if $j=q\cdot 2^n + (i\ \text{mod}\ 2^n) -1$ for some $q\geq1$.
        \item $A_m(i,j)=1$ if and only if $(i-1)\ \text{mod}\ 2^n = j\ \text{mod}\ 2^n$.
        \item If $A_m(i,j)=1$, then $C_k^{i-1}=C_k^{j}$ for every $0<k<2^{n+1}$.
        \item Suppose that $n>1$ and $i_1,i_2\geq 2^{n+1}-1$. If $C_k^{i_1}=C_k^{i_2}$ for every $0<k<2^{n+1}$, then $C_k^{i_1-1}=C_k^{i_2-1}$ for every $0<k<2^n$.
    \end{enumerate}
\end{lemma}

\begin{proof}
    The first two assertions follow directly from the construction of $A_m$. More explicitly, we have the following equivalences:
    \begin{itemize}
        \item[i.] $A_m(i,j)=1$,
        \item[ii.] $j=q\cdot 2^n + (i\ \text{mod}\ 2^n) -1$ for some $q\geq1$, and
        \item[iii.] $j\ \text{mod}\ 2^n = [q\cdot 2^n + (i\ \text{mod}\ 2^n) -1]\ \text{mod}\ 2^n = (i-1)\ \text{mod}\ 2^n$.
    \end{itemize}
    Hence, (1) and (2) are proved. Now assume that $A_m(i,j)=1$. 
    \begin{itemize}
        \item[$(a)$] For $k=1$, we have $A_m(1,i-1)=1=A_m(1,j)$.
        \item[$(b)$] For any $1<k<2^{n}\leq i$, let $1\leq n'<n$ be such that $2^{n'}\leq k<2^{n'+1}$.  From the equality $(i-1) \ \text{mod}\ 2^{n'} = j \ \text{mod}\ 2^{n'}$ obtained from (2), we deduce that
        \begin{equation*}
            (k-1) \ \text{mod}\ 2^{n'} = (i-1) \ \text{mod}\ 2^{n'} \iff (k-1) \ \text{mod}\ 2^{n'} = j \ \text{mod}\ 2^{n'},
        \end{equation*}
        and again by (2), this is equivalent to $A_m(k,i-1)=A_m(k,j)$.
        \item[$(c)$] For $2^n\leq k< 2^{n+1}$, by the definition of $A_m$ and (2), it is straightforward that $A_m(k,i-1)=A_m(k,j)$. In fact, for $k>i$, we have $A_m(k,i-1)=0$, because $i-1<k-1$, and $A_m(k,j)=0$, because $j\equiv i-1\not\equiv k-1\pmod{2^n}$; for $k\leq i$, both $A_m(k,i-1)$ and $A_m(k,j)$ equal $1$ exactly when $k=i$.
    \end{itemize}
     This completes the proof of (3).  For the final statement, take $2^n\leq i<2^{n+1}$ such that $A_m(i,i_1)=A_m(i,i_2)=1$ (note that, in this situation, $i_1,i_2\geq i-1$). Such a $i$ does exist since, for each block of rows $[2^n,2^{n+1})$, a column $i'\geq 2^{n+1}-1$ has exactly one non-zero entry (at the row $2^n+((i'+1)\bmod 2^n)$).
     \begin{itemize}
        \item[$(a)$] If $2^{n}\leq i-1<2^{n+1}$, then, since $(i_1-1) \ \text{mod}\ 2^{n} = (i-2) \ \text{mod}\ 2^{n}=(i_2-1) \ \text{mod}\ 2^{n},$ by (2) and (3) we get $C^{i_1-1}_k=C^{i-2}_k=C^{i_2-1}_k$ for $0<k<2^{n+1}$.
        \item[$(b)$] If $2^{n-1}\leq i-1<2^{n}$ (i.e., $i-1=2^n-1$), then, since
        \begin{equation*}
            (i_1-1) \ \text{mod}\ 2^{n-1} = (i-2) \ \text{mod}\ 2^{n-1}=(i_2-1) \ \text{mod}\ 2^{n-1},
        \end{equation*}
        by (2) and (3) we get $C^{i_1-1}_k=C^{i-2}_k=C^{i_2-1}_k$ for $0<k<2^{n}$.
    \end{itemize}
     In both cases, we obtain $C_k^{i_1-1}=C_k^{i_2-1}$ for every $0<k<2^n$, proving (4).
\end{proof}

From this point until the end of the section, we only consider the transition matrix $A:=\underline{D}+\sum_{k=0}^\infty S_{2^k}$. $A$ is irreducible (even aperiodic) and column-finite, and $X_A$ is compact. Note that every symbol is an infinite emitter, so $A$ is far from the classes of Section \ref{sec:examples}. Next, we prove that, for this matrix, the shift map admits a continuous extension to $X_A$. Later, we show that $X_A$ contains an uncountable number of empty words, each having an infinite root.

\begin{remark}\label{rmk:columns_of_A_and_A_m_coincide}
For $m\in\mathbb{N}_0$, note that the matrices $A$ and $A_m$ coincide in their first $2^{m+1}-1$ rows, i.e., $A(i,j)=A_m(i,j)$ for every $1\leq i<2^{m+1}$ and $j\in\mathbb{N}$. Hence, the results of Lemma~\ref{Lemma: properties of matrix A_N infinite E_A} also hold for the columns of the matrix $A$.
\end{remark}
\vspace{-1cm}
\[
A=\begin{pNiceArray}{cccccccccccccccccccccccc}[first-row,first-col,margin]
& & & & & & & & & & & & & & & & & & & & & & &  \\
2^0 & 1 & 1 & 1 & 1 & 1 & 1 & 1 & 1 & 1 & 1 & 1 & 1 & 1 & 1 & 1 & 1 & 1 & 1 & 1 & 1 & 1 & 1 & \cdots\\
2^1 & 1 & 0 &\Block[draw]{2-2}{} 1 & 0 & \Block[draw]{2-2}{}1 & 0 & \Block[draw]{2-2}{}1 & 0 & \Block[draw]{2-2}{}1 & 0 & \Block[draw]{2-2}{}1 & 0 & \Block[draw]{2-2}{}1 & 0 & \Block[draw]{2-2}{}1 & 0 & \Block[draw]{2-2}{}1 & 0 & \Block[draw]{2-2}{}1 & 0 & \Block[draw]{2-2}{}1 & 0 & \cdots\\
 & 0 & 1 & 0 & 1 & 0 & 1 & 0 & 1 & 0 & 1 & 0 & 1 & 0 & 1 & 0  & 1 & 0 & 1 & 0 & 1 & 0 & 1 & \cdots\\
2^2 & 0 & 0 & 1 & 0 & 0 & 0 & \Block[draw]{4-4}{}1 & 0 & 0 & 0 & \Block[draw]{4-4}{} 1 & 0 & 0 & 0 & \Block[draw]{4-4}{} 1 & 0 & 0 & 0 & \Block[draw]{4-4}{} 1 & 0 & 0 & 0 & \cdots\\
 & 0 & 0 & 0 & 1 & 0 & 0 & 0 & 1 & 0 & 0 & 0 & 1 & 0 & 0 & 0 & 1 & 0 & 0 & 0 & 1 & 0 & 0  & \cdots\\ 
 & 0 & 0 & 0 & 0 & 1 & 0 & 0 & 0 & 1 & 0 & 0 & 0 & 1 & 0 & 0 & 0 & 1 & 0 & 0 & 0 & 1 & 0 & \cdots\\ 
 & 0 & 0 & 0 & 0 & 0 & 1 & 0 & 0 & 0 & 1 & 0 & 0 & 0 & 1 & 0 & 0 & 0 & 1 & 0 & 0 & 0 & 1 & \cdots\\ 
2^3 & 0 & 0 & 0 & 0 & 0 & 0 & 1 & 0 & 0 & 0 & 0 & 0 & 0 & 0 & \Block[draw]{8-8}{}1 & 0 & 0 & 0 & 0 & 0 & 0 & 0 & \cdots\\ 
& 0 & 0 & 0 & 0 & 0 & 0 & 0 & 1 & 0 & 0 & 0 & 0 & 0 & 0 & 0 & 1 & 0 & 0 & 0 & 0 & 0 & 0 & \cdots\\ 
& 0 & 0 & 0 & 0 & 0 & 0 & 0 & 0 & 1 & 0 & 0 & 0 & 0 & 0 & 0 & 0 & 1 & 0 & 0 & 0 & 0 & 0 & \cdots\\ 
& 0 & 0 & 0 & 0 & 0 & 0 & 0 & 0 & 0 & 1 & 0 & 0 & 0 & 0 & 0 & 0 & 0 & 1 & 0 & 0 & 0 & 0 & \cdots\\ 
& 0 & 0 & 0 & 0 & 0 & 0 & 0 & 0 & 0 & 0 & 1 & 0 & 0 & 0 & 0 & 0 & 0 & 0 & 1 & 0 & 0 & 0  & \cdots\\
& 0 & 0 & 0 & 0 & 0 & 0 & 0 & 0 & 0 & 0 & 0 & 1 & 0 & 0 & 0 & 0 & 0 & 0 & 0 & 1 & 0 & 0 & \cdots\\ 
& 0 & 0 & 0 & 0 & 0 & 0 & 0 & 0 & 0 & 0 & 0 & 0 & 1 & 0 & 0 & 0 & 0 & 0 & 0 & 0 & 1 & 0  & \cdots\\ 
& 0 & 0 & 0 & 0 & 0 & 0 & 0 & 0 & 0 & 0 & 0 & 0 & 0 & 1 & 0 & 0 & 0 & 0 & 0 & 0 & 0 & 1  & \cdots\\ 
& \vdots & \vdots & \vdots & \vdots  & \vdots & \vdots & \vdots & \vdots & \vdots & \vdots & \vdots & \vdots  & \vdots & \vdots & \vdots & \vdots & \vdots & \vdots & \vdots & \vdots  & \vdots & \vdots & \ddots 
\end{pNiceArray}
\]

\begin{proposition}\label{propo: shift on X_A^N extends continuously}
    The generalized shift map on $X_A$ admits a continuous extension.
\end{proposition}

\begin{proof}
    Let $(i_n)_n$ and $(j_n)_n$ be sequences in $\mathbb{N}$ such that $i_n\to\infty$ and $A(i_n,j_n)=1$ for every $n$. Suppose that $C^{i_n}\to C\in \{0,1\}^\mathbb{N}$. We show that each coordinate of the sequence $(C^{j_n})_n$ is eventually constant. Given $N\in\mathbb{N}$, $N>1$, choose $n_0\in\mathbb{N}$ such that, for every $n\geq n_0$,
    $$2^{N+1}\leq i_n \quad\text{and}\quad C^{i_n}_k=C_k \text{ for every }0<k<2^{N+1}.$$ Hence, by Lemma \ref{Lemma: properties of matrix A_N infinite E_A} (4), $C_k^{i_{n_0}-1}=C_k^{i_n-1}  \text{ for every } 0<k<2^{N}$. Since $A(i_n,j_n)=1$ and $i_n\geq 2^{N+1}$, Lemma \ref{Lemma: properties of matrix A_N infinite E_A} (3) (applied to the block $[2^{t},2^{t+1})$ containing $i_n$, with $t\geq N+1$, via Remark \ref{rmk:columns_of_A_and_A_m_coincide}) implies that, for every $n\geq n_0$,
    $$
    C_k^{j_n}=C_k^{i_n-1}=C_k^{i_{n_0}-1}  \text{ for every } 0<k<2^{N}.
    $$
    Since $N$ is arbitrary, we have therefore that $(C^{j_n})_n$ converges, and by Theorem~\ref{Theorem: alpha well-defined iff sigma extended} (whose hypotheses are satisfied, see the paragraph before Remark \ref{rmk:columns_of_A_and_A_m_coincide}), the shift map extends continuously.
\end{proof}

\begin{proposition}\label{Propo: limit points of the columns of A infinite E_A}
    The set of limit points of the columns of $A$ is given by
    \begin{equation}\label{eq:roots_of_A_uncountable_empty_words}
        \mathcal{C}:=\left\{C\in \{0,1\}^{\mathbb{N}} : \begin{array}{l} R_C=\{k_n:n\in\mathbb{N}\}\subseteq\mathbb{N} \text{ such that } 2^{n-1}\leq k_n < 2^n  \\ 
    \text{ and } k_{n}\ \text{mod}\ 2^{n-1}= k_{n+1}\ \text{mod}\ 2^{n-1}\end{array}\right\}.
    \end{equation}
    Moreover, if $C\in\mathcal{C}$ and $R_C=\{k_n:n\in\mathbb{N}\}$ is as above, then $C^{k_n-1}\to C$.
\end{proposition}

\begin{remark} A brief comment about the set $\mathcal{C}$ before the proof of Proposition \ref{Propo: limit points of the columns of A infinite E_A}: the condition $k_n\bmod 2^{n-1}=k_{n+1}\bmod 2^{n-1}$ says that the residues $k_n-2^{n-1}\in\{0,\dots,2^{n-1}-1\}$ form a coherent sequence, i.e.\ they are the truncations of a unique $2$-adic integer $x$. This is the content of Proposition \ref{pro:bijection_between_E_A_and_Cantor_set}, and the map of Proposition \ref{propo: shift on X_A^N extends continuously} is $x\mapsto x-1$. 
\end{remark}

\begin{proof}
    Let $R_C=\{k_n:n\in\mathbb{N}\} \subseteq\mathbb{N}$ be a set such that $C\in \mathcal{C}$. Consider the sequence of columns $(C^{k_n-1})_{n>1}$ of $A$. For any $n\in\mathbb{N}$ and
    $1<m<n$, a recursive argument gives
    $$
    (k_m-1) \ \text{mod}\ 2^{m-1}=(k_n-1) \ \text{mod}\ 2^{m-1}.
    $$
    Consequently, by Lemma~\ref{Lemma: properties of matrix A_N infinite E_A} (2), we obtain $A(k_m,k_n-1)=1$.  
    For $m=1$, $A(k_1,k_n-1)=A(1,k_n-1)=1$, and for $m=n$,  $A(k_n,k_n-1)=1$. Moreover, by the same result from the aforementioned lemma, $k_m$ is the unique integer in $[2^{m-1},2^m)$ such that $A(k_m,k_n-1)=1$, i.e., $C^{k_n-1}_k=1$ if and only if $k=k_m$, $1\leq m\leq n$.
    In other words, $C^{k_n-1}$ coincides with $C$ in its first $2^n-1$ entries. Hence, the sequence $(C^{k_n-1})_{n>1}$ converges to $C$. Conversely, for any $j>1$, take $n\in\mathbb{N}$ such that $2^n\leq j+1 < 2^{n+1}$. Consider the set $\{k_m:1\leq m \leq n+1\}$, where $k_m:=2^{m-1}+(j+1) \ \text{mod}\ 2^{m-1}$, with $2^{m-1}\leq k_m<2^m$ and, for some $q:=q(m)\in\{1,2,\cdots, j\}$ 
    \begin{align*}
        j+1&=q2^{m-1}+(j+1) \ \text{mod}\ 2^{m-1}=q2^{m-1}+k_m \ \text{mod}\ 2^{m-1}.
    \end{align*}
    By Lemma \ref{Lemma: properties of matrix A_N infinite E_A} (1), $C_{k_m}^j = A(k_m,j)=1$ for every $1\leq m \leq n+1$. Moreover, observe that
    $$(k_{m+1}-k_m)\ \text{mod}\ 2^{m-1}=(2^{m-1}+(j+1) -(j+1) ) \ \text{mod}\ 2^{m-1}=0 \ \text{mod}\ 2^{m-1},$$
    that is, $k_{m}\ \text{mod}\ 2^{m-1}= k_{m+1}\ \text{mod}\ 2^{m-1}$ for every $1\leq m\leq n$. Therefore, there exists $C\in\mathcal{C}$ such that $C^j$ and $C$ have the same first $2^n$ entries. Hence, since $j$ is arbitrary, every limit point of the set of columns of $A$ lies in $\mathcal{C}$.
\end{proof}

\begin{remark}
   Since all the roots are infinite sets, and by Corollary \ref{cor:sufficient_condition_when_sigma_is_not_homeo}, we conclude that $\sigma$ is not a local homeomorphism. 
\end{remark}

\begin{proposition}\label{pro:bijection_between_E_A_and_Cantor_set}
There exists a homeomorphism $\Phi:E_A\longrightarrow \{0,1\}^{\mathbb{N}}$,
defined as follows. Let $(e,R_C)\in E_A$, where
$R_C=\{k_n:n\in\mathbb{N}\}$
for some $C\in\mathcal{C}$. We define
$\Phi(e,R_C)=x$,
where $x=(x_n)_n\in\{0,1\}^{\mathbb{N}}$ is given by
\begin{equation*}
    x_n=\begin{cases}
        0,\; \text{ if }\quad 2^n \leq  k_{n+1}<3\cdot 2^{n-1}, \\
        1,\; \text{ if }\quad 3\cdot 2^{n-1} \leq  k_{n+1}<2^{n+1}.
    \end{cases} 
\end{equation*}
In other words, $x_n$ is the coefficient of $2^{n-1}$ in the binary expansion of $k_{n+1}-2^n$, so that $\Phi(e,R_C)$ is the sequence of binary digits of the $2$-adic integer $x$ with $k_{n+1}=2^n+(x\bmod 2^n)$ for all $n$.
\end{proposition}

\begin{proof}
     We show that $\Phi$ is a function. Given $k_n$, we show that $k_{n+1}$ belongs either to the first half or to the second half of the interval $[2^n,2^{n+1})$, so $x_n$ is well-defined for every $n$. For injectivity, let $R_C=\{k_n:n\in\mathbb{N}\}$ and $R_{C'}=\{k'_n:n\in\mathbb{N}\}$ correspond to some distinct $C,C'\in\mathcal{C}$. Let $n_0$ be the smallest integer such that $k_{n_0}\neq k'_{n_0}$. Note that $n_0>1$. If $n_0=2$, it is trivial. If $n_0>2$, suppose that $k_{n_0},k'_{n_0}$ are in the same half of $[2^{n_0-1},2^{n_0})$. Without loss of generality, $k_{n_0},k'_{n_0}\in [2^{n_0-1},3\cdot 2^{n_0-2})$. Hence, $k_{n_0}=2^{n_0-1}+i$ and $k'_{n_0}=2^{n_0-1}+i'$ with $0\leq i,i'<2^{n_0-2}$. Since 
    $$i \ \text{mod} \ 2^{n_0-2}= k_{n_0} \ \text{mod} \ 2^{n_0-2}= k_{n_0-1} \ \text{mod} \ 2^{n_0-2}= k'_{n_0} \ \text{mod} \ 2^{n_0-2}=i' \ \text{mod} \ 2^{n_0-2},$$
    we have that $i=i'$. Then $\Phi(e,R_C)=(x_n)_n$ is different from $\Phi(e,R_{C'})=(x'_n)_n$, because $x_{n_0-1}\neq x'_{n_0-1}$. For surjectivity, given $(x_n)_n\in\{0,1\}^{\mathbb{N}}$, take $\{k_n:n\in\mathbb{N}\}$ defined recursively by $k_1=1$ and $k_{n+1}=2^n+x_n2^{n-1}+i_{n+1}$, where $0\leq i_{n+1}<2^{n-1}$ and  $i_{n+1}=k_n \ \text{mod} \ 2^{n-1}$; the corresponding $C$ belongs to $\mathcal{C}$, so $(e,R_C)\in E_A$ by Proposition \ref{Propo: limit points of the columns of A infinite E_A}. This proves that $\Phi$ is a bijection. We now show that $\Phi$ is sequentially continuous. Let $(\xi^m)_m$ be a sequence in $E_A$ converging to some $\xi\in E_A$. By Proposition \ref{Propo: limit points of the columns of A infinite E_A}, there exist $C \in \mathcal{C}$ and $D^m\in \mathcal{C}$, such that $\xi^m=(e,R_{D^m})$ and $\xi=(e,R_{C})$. Let $x^m:=\Phi(e,R_{D^m})$, and set $x:=\Phi(e,R_{C})$. Given $N\in\mathbb{N}$, since $\xi^m\to\xi$, there exists $M$ such that 
    $$\xi^m_{i^{-1}}=1 \iff C_i=1\quad 1\leq i<2^{N+1}\quad \text{for every }m\geq M.$$
    Equivalently, writing $R_C=\{k_n:n\in\mathbb{N}\}$, we have $k_n\in R_{D^m}$, $1\leq n<N+1$, for every $m\geq M$. Therefore, $x^m$ and $x$ coincide in their first $N$ coordinates for every $m\geq M$. Hence, $x^m$ converges to $x$. Since $E_A=\bigcap_{i\in\mathbb{N}}C_i^c$ is closed in the compact space $X_A$, the continuous bijection $\Phi$ is a homeomorphism.
\end{proof}

\begin{proposition}
    Let $R_C=\{k_n:n\in\mathbb{N}\}$ with $C\in\mathcal{C}$, and define the set $\big\{\overline{k}_n:n\in\mathbb{N}\big\}$ by  
    $$
        \overline{k}_n=\begin{cases}
        2^n-1, & \text{if } k_n=2^{n-1}, \\
        k_n-1, & \text{if } 2^{n-1}<k_n<2^{n}.
        \end{cases} 
    $$
    Then there exists $\overline{C}\in\mathcal{C}$ such that $R_{\overline{C}}=\big\{\overline{k}_n:n\in\mathbb{N}\big\}$, and furthermore, $\sigma(e,R_C)=(e,R_{\overline{C}})$.
\end{proposition}

\begin{proof}
It is straightforward to verify that, for every $n\in\mathbb{N}$, $2^{n-1}\leq \overline{k}_n<2^n$ and
\begin{align*}
    \overline{k}_{n} \ \text{mod}\ 2^{\,n-1}&=(k_{n}-1) \ \text{mod}\ 2^{\,n-1}=(k_{n+1}-1) \ \text{mod}\ 2^{\,n-1}=\overline{k}_{n+1} \ \text{mod}\ 2^{\,n-1}.
\end{align*}
Hence, there exists $C'\in\mathcal{C}$ such that $R_{C'}=\big\{\overline{k}_n:n\in\mathbb{N}\big\}$. In Proposition~\ref{Propo: limit points of the columns of A infinite E_A}, we showed that the sequence $(C^{k_{n}-1})_n$ converges to $C$. From the definition of $\overline{k}_{n+1}$ we have $(k_{n+1}-2)\bmod 2^n=(\overline{k}_{n+1}-1)\bmod 2^n$ for every $n>1$, and hence, by Lemma \ref{Lemma: properties of matrix A_N infinite E_A} (2) and Remark \ref{rmk:columns_of_A_and_A_m_coincide} applied to the row $i=k_{n+1}-1$, we get $A(k_{n+1}-1,\overline{k}_{n+1}-1)=1$. Then, by Theorem \ref{Theorem: alpha well-defined iff sigma extended} and Proposition \ref{propo: shift on X_A^N extends continuously}, there exists $\overline{C}\in\mathcal{C}$ such that $(C^{\overline{k}_{n}-1})_n$ converges to $\overline{C}$ and $\sigma(e,R_C)=(e,R_{\overline{C}})$. By Proposition \ref{Propo: limit points of the columns of A infinite E_A}, we obtain $C'=\overline{C}$, i.e., $R_{\overline{C}}=\big\{\overline{k}_n:n\in\mathbb{N}\big\}$.
\end{proof}

\section{Concluding remarks}

Using the faithful representation of $\mathcal{O}_A$ on $\mathfrak{B}(\ell^2(\Sigma_A))$ due to R. Exel and M. Laca \cite{EL1999}, we considered the normal unital completely positive map $\theta(Y)=\sum_{i\in\mathbb{N}}T_iYT_i^*$ on $\mathfrak{B}(\ell^2(\Sigma_A))$ and its restriction to $\pi(\mathcal{O}_A)$. When the image of this restriction is contained in $\pi(\mathcal{O}_A)$, $\theta$ corresponds, up to the $*$-isomorphism $\pi$, to the canonical UCP map on the Exel--Laca algebra $\mathcal{O}_A$. We also showed that the restriction $\alpha=\theta|_{\mathcal{D}_A}$ is always a $*$-homomorphism into the diagonal operators. 

Our main result, namely Theorem \ref{Theorem: alpha well-defined iff sigma extended}, provides equivalent conditions for the existence of this canonical UCP map. These conditions are:
(i) $\alpha(\mathcal{G}_A)\subseteq \mathcal{D}_A$;
(ii) $\alpha=\theta|_{\mathcal{D}_A}$ is a $*$-endomorphism of $\mathcal{D}_A$;
(iii) $\sigma$ extends continuously to $X_A$; or
(iv) the column convergence condition (5) on $A$ holds.

We emphasize that column-finiteness is necessary. These results generalize the existence of a canonical UCP map for Exel--Laca algebras initiated in \cite{MichiNakaHisaYoshi2022}; in particular, we do not require a unique empty-stem configuration. Indeed, we presented three classes of generalized countable Markov shifts for which such an extension exists: the single-empty-word, periodic renewal, and permutation renewal classes. Only the first requires the generalized shift space to have a unique empty word. The three classes are pairwise intersecting, and none contains another.

In all the classes above, the set $E_A$ of empty words is finite and satisfies $\sigma(E_A)=E_A$. The induced dynamics is a fixed point in the first case, a cyclic permutation in the second, and a prescribed permutation in the third. Proposition \ref{thm:E_A_shift_invariant} shows that if $X_A$ is compact and the shift map admits a continuous extension, then $E_A$ must be invariant, but not necessarily equal to its image under the dynamics. Examples \ref{Example: finitely many empty words and strict invariance on E_A} and \ref{Example: infinitely many empty words and strict invariance on E_A} give explicit cases in which $\sigma(E_A)$ is a proper subset of $E_A$. 


Furthermore, we proved that these continuous extensions of the dynamics are not necessarily local homeomorphisms. Theorem \ref{thm:sufficient_matrix_condition_X_A_isnt_homeo_local} gives a column-convergence condition on the matrix for when the extension is not a local homeomorphism, and Corollary \ref{cor:sufficient_condition_when_sigma_is_not_homeo} shows that if $\sigma$ is a local homeomorphism, then the root of each empty word---and hence of each finite word---must be finite. We also proved that the extended dynamics of every generalized countable Markov shift in the periodic and permutation renewal classes is a local homeomorphism. In the single-empty-word class, however, this is not true in general: the upper triangle shift and the lazy renewal shift are examples in which the extension of $\sigma$ is continuous but not a local homeomorphism.

Beyond these classes, we concluded the paper with an example outside the three aforementioned classes. It has uncountably many empty words, in bijection with a Cantor set, and its shift map admits a continuous extension. In this case every empty word has an infinite root, and hence $\sigma$ is not a local homeomorphism. Nevertheless, $E_A$ is invariant, and the dynamics restricted to the set of empty words corresponds to the inverse of the dyadic odometer.

We also mention that, in the course of this work, we explicitly characterized the Gelfand isomorphism for $\mathcal{D}_A$ in the Exel--Laca setting by mapping its generators to explicit characteristic functions. This may be useful in other studies of Exel--Laca algebras or similar structures.

Several natural questions arise from this paper. We note that, by Theorem \ref{thm:alpha_star_hom}, $\alpha=\theta|_{\mathcal{D}_A}$ is always a $*$-homomorphism, so $\alpha(\mathcal{D}_A)$, the C$^*$-algebra generated by $\alpha(\mathcal{G}_A)$, is always a commutative C$^*$-algebra of diagonal operators on $\ell^2(\Sigma_A)$; Theorem \ref{Theorem: alpha well-defined iff sigma extended} says that it is contained in $\mathcal{D}_A$ exactly when $\sigma$ extends continuously. It may be interesting to study the spectrum of $\alpha(\mathcal{D}_A)$, and the map it induces into $X_A$, in the cases where the shift map does not admit a continuous extension.

Another possible direction is to investigate more general classes of generalized countable Markov shifts for which the shift map extends continuously. In particular, it remains open to study examples in which the dynamics on $E_A$ is governed by a more general function than a permutation, possibly allowing infinitely many empty words, as in Examples \ref{Example: finitely many empty words and strict invariance on E_A} and \ref{Example: infinitely many empty words and strict invariance on E_A}. We also mention the possibility of generalizing results concerning conditions under which the dynamics is or is not a local homeomorphism.

Finally, another natural direction is the study of entropy and pressure in the generalized setting. In the finite-alphabet case, the Gurevich entropy of the shift map agrees with the noncommutative entropy of the canonical map on the Cuntz--Krieger algebra \cite{BocaGoldstein2000}, and an analogous result holds for pressure \cite{KerrPinzari2002}. Motivated by these results, one may continue the analysis developed in \cite{MichiNakaHisaYoshi2022} and investigate whether the Gurevich entropy of generalized shift maps that admit a continuous extension to the whole space $X_A$ coincides with the noncommutative entropy of the canonical map on Exel--Laca algebras. An analogous question may be posed for pressure.

\section*{Acknowledgements}

TR and IDG were supported by the Coordenação de Aperfeiçoamento de Pessoal de Nível Superior – Brasil (CAPES) – Finance Code 001. RB and IDG were supported by CNPq, specifically through grants 311658/2025-3 and 407527/2025-7. We thank G. Lübe for the discussions that helped improve Section \ref{sec:when_is_ homeo} and for contributing the results and proofs concerning the local homeomorphism property. We thank M. Dokuchaev and R. Exel for pointing out a mistake in a previous version of the paper and for very fruitful discussions. In addition, during the qualifying exam of IDG, Ruy Exel pointed out a couple of the results of the present article, together with their proofs, for example Theorem \ref{Theorem: If sigma is extended continuously, then A is column finite} and item (5) of Theorem \ref{Theorem: alpha well-defined iff sigma extended}. We are grateful for these very fruitful discussions, and we highlight the privilege of having Ruy Exel as a colleague in Brazil.

\bibliographystyle{plain}
\bibliography{bibliografia}

\end{remark}

\end{document}